\documentclass[11pt]{article}
\usepackage{amsthm, amsmath, amssymb, graphicx, url}
\usepackage[margin=1in]{geometry}
\usepackage{graphicx}
\newcommand{\smallimplies}{\mathrel{\text{\scalebox{0.7}{$\implies$}}}}
\usepackage{mathptmx}
\usepackage{latexsym, amscd, amsfonts, mathrsfs, stmaryrd, tikz-cd, mathrsfs, bbm, esint, listings, moreverb, hyperref, pifont, enumitem}
\usepackage[capitalize,nameinlink]{cleveref}
\usepackage{mathrsfs}
\usepackage[mathscr]{euscript}

\DeclareFontFamily{U}{mathx}{}
\DeclareFontShape{U}{mathx}{m}{n}{<-> mathx10}{}
\DeclareSymbolFont{mathx}{U}{mathx}{m}{n}
\DeclareMathAccent{\widehat}{0}{mathx}{"70}
\DeclareMathAccent{\widecheck}{0}{mathx}{"71}

\newcommand{\p}{\mathbb{P}}
\newcommand{\circnum}[1]{%
  \text{\ding{\the\numexpr #1+191}}%
}

\DeclareMathOperator{\argmax}{arg\,max}

\newtheorem{theorem}{Theorem}
\newtheorem{lemma}[theorem]{Lemma}

\theoremstyle{definition}
\newtheorem{Notations}{Notations}
\newtheorem{definition}[theorem]{Definition}

\newtheorem*{remark}{Remark}

\usepackage{times}

\begin{document}

\title{Exact recovery in Gaussian weighted stochastic block model and planted dense subgraphs:  Statistical and algorithmic thresholds }
\date{}
\author{ Aaradhya Pandey\thanks{\texttt{aaradhyapandey@princeton.edu}. Princeton University.} \and Sanjeev Kulkarni\thanks{\texttt{kulkarni@princeton.edu}. Princeton University.}}

\maketitle
\begin{abstract}
In this paper,  we study the exact recovery problem in the Gaussian weighted version of the Stochastic block model with two symmetric communities. We provide the information-theoretic threshold in terms of the signal-to-noise ratio (SNR) of the model and prove that when $\text{SNR} <1$, no statistical estimator can exactly recover the community structure with probability bounded away from zero. On the other hand, we show that when $\text{SNR} >1$, the Maximum likelihood estimator itself succeeds in exactly recovering the community structure with probability approaching one. Then, we provide two algorithms for achieving exact recovery. The Semi-definite relaxation as well as the spectral relaxation of the  Maximum likelihood estimator can recover the community structure down to the threshold value of $1$, establishing the absence of an information-computation gap for this model.
\par 
Next, we compare the problem of community detection with the problem of recovering a planted densely weighted community within a graph and prove that the exact recovery of two symmetric communities is a strictly easier problem than recovering a planted dense subgraph of size half the total number of nodes, by establishing that when the same $\text{SNR} < \frac{3}{2}$, no statistical estimator can exactly recover the planted community with probability bounded away from zero. More precisely, when $ 1 <\text{SNR} < \frac{3}{2}$ exact recovery of community detection is possible, both statistically and algorithmically, but it is impossible to exactly recover the planted community, even statistically, in the Gaussian weighted model.   Finally, we show that when $\text{SNR} >2$, the Maximum likelihood estimator itself succeeds in exactly recovering the planted community with probability approaching one. We also prove that the Semi-definite relaxation of the  Maximum likelihood estimator can recover the planted community structure down to the threshold value of $2$.
\end{abstract}

\section{Introduction and related works} \label{sec:intro}
Estimating community structures in graphs is a central problem in computer science,  statistical machine learning, and complex networks. The stochastic block model (SBM) is a random graph model \cite{HOLLAND1983109} wherein distinct communities exhibit varied connectivity patterns. This model has been extensively studied as a standard framework for investigating clustering and community detection (See \cite{goldenberg2009survey}).  Furthermore, it serves as a rich domain for exploring the intricate tradeoffs between information and computation inherent in combinatorial statistics and the broader field of data science. We refer the reader to \cite{abbe2023community}, \cite{moore2017computer} and references therein for more details. 

  \section{Gaussian weighted stochastic block model: description}
   In  SBM with two symmetric communities, a random graph $G_n = ([n], E_n)$ with vertex set $[n]:=\{1, 2, \cdots, n-1, n\}$ and edge set $E_n \subset  \binom{[n]}{2}:= \{\{i,j\}: i \neq j \in [n]\}$ has two equally sized clusters given by $\sigma^* \in \{\pm 1\}^{n}$   ($n$ even) such that $ \{i \in [n]: \sigma_i =1\} = \{ i\in [n]: \sigma_i =-1\} =\frac{n}{2}. $
  
  The vertices connect with probability $p$ within clusters and $q$ across clusters. The problem of exact recovery asks to output a balanced partition $\hat{\sigma}$ that recovers the unobserved partition $\sigma^*$  correctly with probability approaching one as $n \to  \infty$.  Indeed, one can only hope to recover the communities up to a global sign flip of the labels. It turns out that the critical scaling of the parameters for this problem is given by the following for constants $a, b > 0$
  \begin{equation} \label{def:CritSBM0}
  p=\frac{a \log n}{n}  \quad \text{ and } \quad q = \frac{b \log n}{n}.
  \end{equation}
  
   In the Gaussian weighted version of the stochastic block model (GWSBM), we denote a weighted graph $G_{n} \sim$ SBM$(n, \mu_1\,\mu_2, \tau^2)$ on $n$ vertices with labels $(\sigma_i^*)_{i\in [n]}$ such that  
   $\#\{i\in [n]: \sigma_i^*=1\}=\#\{i\in [n]: \sigma_i^*=-1\}=\frac{n}{2}$ and conditioned on this labeling $(\sigma_i^*)_{i\in [n]}$,  we have  Gaussian\footnote{ where $ N(\mu, \tau^2)$ denotes a \text{Normally distributed random variable with mean } $\mu$ \text{ and variance } $\tau^2$} weighted edges sampled independently for every pair $\{i,j\}$.
   \begin{equation}\label{def:GWSBM0}
A_{ij}=A_{ji} \sim
\begin{cases}
  N(\mu_1, \tau^2)  & \text{if } \sigma^*(i) =\sigma^*(j)  \\
   N(\mu_2, \tau^2) & \text{if }  \sigma^*(i)\neq \sigma^*(j) 
\end{cases}
   \end{equation} 
   
     We focus on the assortative case $\mu_1>\mu_2$, and work with the critical scaling of  $\mu_1$ and $\mu_2$ as
    \begin{equation} \label{def:CritGWSBM0}
     \mu_1= \alpha \sqrt{\frac{\log n}{n}}, \quad   \mu_2= \beta\sqrt{\frac{\log n}{n}} 
    \end{equation}
    
  for constants $\alpha >\beta $. Further,  for $ \tau >0$, we define the signal-to-noise ratio (SNR)  as
  \begin{equation} \label{def:SNRp}
      \text{SNR}:= \frac{(\alpha -\beta)^2}{8\tau^2}.
  \end{equation}
  
 The problem of exact recovery asks to recover the  community labeling $(\sigma_i^*)_{i\in [n]}$ exactly up to a global sign flip, having observed an instance of the weighted graph $G_n\sim$ SBM$(n, \mu_1\,\mu_2, \tau^2)$ or equivalently a sample of the random $n \times n $ symmetric weighted adjacency matrix $(A_{ij})_{i,j,\in [n]}$ of the graph. 
 Let us  denote the parameter space $\Theta_n$  with 
 $\mathbf{1}_n := (1, \cdots, 1)\in \mathbb{R}^n$ as
 \begin{equation}
 \Theta_n = \{\sigma\in \{\pm1\}^n, \langle\sigma, \mathbf{1}_n\rangle=0\}.
 \end{equation}
 
  We define the agreement ratio between $\hat{\sigma},  \sigma  \in \{\pm 1\}^n$ as  
  \[
  M(\hat{\sigma}, \sigma)= \max \left(\frac{1}{n}\sum_{i=1}^n \mathbf{1}_{\hat{\sigma}(i)=\sigma(i)}, \frac{1}{n}\sum_{i=1}^n \mathbf{1}_{-\hat{\sigma}(i)=\sigma(i)}\right).
  \]

Then, we say that the problem of exact recovery is solvable if there exists an estimator that upon observing an instance of the weighted random graph $G_n \sim$ SBM$(n,\mu_1, \mu_2,\tau^2)$ with an unobserved $\sigma^* \in \Theta_n$,   outputs a community label $\hat{\sigma} \in \Theta_n
$ such that  $\mathbb{P} \left[M(\hat{\sigma},\sigma^*)=1\right] \to 1.$ We say the problem of exact recovery is unsolvable otherwise. That is, with probability bounded away from zero, every estimator fails to exactly recover the community labels, as the number of vertices grows.
\section{Gaussian weighted stochastic block model:  results and proof sketch}
In this section, we present our results along with the ideas of the proof for GWSBM in comparison with the breakthrough works of \cite{abbe2014exact}, \cite{Mossel_2016}, \cite{hajek2016achieving}, \cite{abbe2019entrywise} that established an important set of results for the SBM to inspire a Gaussian weighted version. Detailed proofs of all the theorems and lemmas are provided in the appendix.  
\par 
The literature for community detection and related problems is vast, and a weighted version similar to ours has been considered in the literature  \cite{7447159}, \cite{jog2015informationtheoretic}, \cite{xu2018optimal}. More precisely, results similar to Theorems ~\eqref{thm:Negative0}, and ~\eqref{thm: Positive0} has been established in  \cite{esmaeili2021community} under certain assumptions. \cite{bandeira2015random} established a result similar to Theorem ~\eqref{thm:SDP0} for the $\mathbb{Z}_2$ synchronization problem, using a different semi-definite program ~\eqref{eq:SDP0}.
\par
To explore the information-theoretic limits, we start by examining the maximum a posteriori (MAP) estimation, designed to maximize the probability of accurately reconstructing communities. In the context of uniform community assignment, (which is the case for us as we are fixing an unknown community labeling $\sigma^*$ with balanced community sizes, which could be chosen uniformly from all the possible $\binom{n}{\frac{n}{2}}$ choices of community assignments), MAP estimation is equivalent to maximum likelihood estimation (MLE). We refer to \cite{abbe2023community}[Section 3.1] for more details.
\par 
Next, we observe that the MLE for SBM and GWSBM  is given by the following.
    \begin{equation} \label{eq:MLE0}
   \hat{\sigma}_{MLE}= \underset{\sigma \in \Theta_n}{ \argmax} \hspace{.1 cm}f_A(\sigma) \quad \quad \text{  where  } \quad f_A(\sigma):=\sum_{i,j} A_{ij}\sigma_i \sigma_j 
    \end{equation}
    
Under the critical scaling ~\eqref{def:CritSBM0} in the stochastic block model,  it was shown in \cite{abbe2014exact} that if $|\sqrt{a} -\sqrt{b}| <\sqrt{2}$ then MLE fails to exactly recover the true community structure $\sigma^*$ with probability bounded away from zero as $n\to \infty$.  In the Gaussian weighted version of the stochastic block model ~\eqref{def:GWSBM0} we establish the following stronger statistical impossibility result.
\begin{theorem}[\textbf{Statistical impossibility of exact recovery}] \label{thm:Negative0}
     Let $G_n \sim $ SBM$(n,\mu_1, \mu_2, \tau^2)$ with $\mu_1, \mu_2$ in the critical scaling regime ~\eqref{def:CritGWSBM0}, and $\alpha > \beta$, $\tau >0$ be constants. If SNR$<1$ then MLE fails in recovering the community label $\sigma^*$ with high probability. 
     \[
     \underset{n \to \infty}{\limsup}\hspace{.1cm}\p(M(\hat{\sigma}_{MLE}, \sigma^*)=1) = 0 \quad \iff  \quad \underset{n \to \infty}{\liminf}\hspace{.1cm}\p(M(\hat{\sigma}_{MLE}, \sigma^*)\neq 1) =1.
     \]
\end{theorem}

 We prove this by establishing that whenever SNR$<1$, with probability approaching one we can find vertices $i$ and $j$ in each community such that swapping them increases the likelihood ~\eqref{eq:MLE0} and hence causes MLE to fail to recover the ground truth $\sigma^*$ exactly. Without loss of generality, we assume that the unknown community is given by $\sigma^* =(1, \cdots, 1, -1, \cdots, -1) \in \Theta_n$, and $G_n \sim $ SBM$(n, \mu_1, \mu_2, \tau^2)$ conditioned on $\sigma^*$. We denote the resulting communities  as
    \[C_+=\left\{1, \cdots, \frac{n}{2}\right\},  \quad C_{-}=\left\{\frac{n}{2}+1, \cdots, n\right\}.\]
    
    More precisely, we define the set of bad pairs of vertices in $G_n$ with $\sigma^*[i\leftrightarrow j]$ denoting the vector obtained by swapping the values of coordinates $i$ and $j$ in $\sigma^*$.
    \begin{equation}\label{def:bad0}
    \textit{B}(G_n) := \{(i,j): i\in C_+, j\in C_{-}, f_A(\sigma^*)< f_A(\sigma^*[i\leftrightarrow j])\}.
    \end{equation}

From ~\eqref{eq:MLE0} we observe that  $f_A(\sigma^*)< f_A(\sigma^*[i\leftrightarrow j]) $ is the same as the event that the likelihood of the partition $\sigma^*[i\leftrightarrow j]$ is larger than the planted one $\sigma^*$. So it reduces to the following.
\begin{lemma}
   MLE fails to solve the exact recovery with high probability if $\textit{B}(G_n)$ is non-empty with high probability. More precisely, we have the following
   \begin{equation} \label{eq:NegLB0}
   \underset{n \to \infty}{\liminf}\hspace{.1cm}\p(M(\hat{\sigma}_{MLE}, \sigma^*)\neq1) \geqslant \underset{n \to \infty}{\liminf}\hspace{.1cm} \p (\textit{B}(G_n)\neq \emptyset)
   \end{equation}
\end{lemma}

We now analyze the condition $f_A(\sigma^*)< f_A(\sigma^*[i\leftrightarrow j])$ and show that  $\underset{n \to \infty}{\lim}\hspace{.1cm}\p (\textit{B}(G_n)\neq \emptyset) =1.$
\begin{definition}
    For a vertex $i \in C_+$, and $j \in C_-$ we define the following.
    \[
    d_+(i)= \sum_{v\in C_+}A_{iv}, \quad  d_+(j)=  \sum_{v\in C_-}A_{jv}, \quad d_-(i)= \sum_{v\in C_-}A_{iv},
    \]
    \[
    d_-(j)=  \sum_{v\in C_+}A_{jv}, \quad d_-(i\setminus j)=\sum_{v\in C_- \setminus j}A_{iv}, \quad  d_-(j\setminus i)=\sum_{v\in C_+ \setminus i}A_{jv}
    \]
\end{definition}

It is worth observing that $d_+(i)$ and $ d_+(j)$ sums the edge weights that $i$ and $j$ have with their own communities $C_+$ and  $C_-$, and $d_-(i)$ and $ d_-(j)$ sums the edge weights that $i$ and $j$ have with their opposite communities $C_-$ and  $C_+$ respectively.  With the notations above we have the following local condition for a bad pair $\{i,j\}$.
\begin{lemma}
 Let $i\in C_+, j\in C_-$. Then we have the following equivalence.
    \[
    f_A(\sigma^*)< f_A(\sigma^*[i\leftrightarrow j]) \quad \iff \quad  d_+(i)+d_+(j) < d_-(i\setminus j)+d_-(j\setminus i).
    \]
\end{lemma}

Finally, we analyze $\textit{B}(G_n)$ ~\eqref{def:bad0} the set of bad pairs of vertices by analyzing sets $\textit{B}_+(G_n)$  and $ \textit{B}_-(G_n)$ with individual vertices satisfying certain conditions, and apply the first and second-moment methods \cite{10.5555/3002498}. We now define the following sets with diverging constants $c_n \to \infty $ to be chosen later according to the proof.

\begin{definition}  The set of bad vertices (instead of bad pairs) in each community in $G_n$ is given by
\[
\textit{B}_+(G_n):=\left\{i: i \in C_+, d_{+}(i) < d_{-}(i)-c_n\right\}, 
\]
\[
\textit{B}_-(G_n):=\left\{j: j \in C_-, d_{+}(j) < d_{-}(j)-c_n\right\},
\]
\end{definition}

 \begin{lemma} \label{lem:Bto10}
     With the notations above, if $\textit{B}_+(G_n)$ and $\textit{B}_-(G_n)$ are non-empty with probability approaching $1$, then $\textit{B}(G_n)$ is non-empty with probability approaching  $1$ as $n \to  \infty$.
\end{lemma}

It remains to establish that if SNR$<1$, we have  $\mathbb{P}\left[ \textit{B}_+(G_n)\neq \emptyset \right] = \mathbb{P}\left[ \textit{B}_-(G_n)\neq \emptyset \right]=1 - o(1)$.
\begin{lemma} \label{lem:SM0}
 If $\text{SNR}<1$, then $\textit{B}_+(G_n)$ and  $\textit{B}_-(G_n)$ are non empty with probability approaching $1$.
\[
  \mathbb{P}\left[ \textit{B}_+(G_n)\neq \emptyset \right]  = 
  \mathbb{P}\left[ \textit{B}_-(G_n)\neq \emptyset \right]=1 - o(1).
\]
\end{lemma}

We now apply the second-moment method to the  random variable $Z_+ = \sum_{i \in C_+}\mathbf{1}_{d_{+}(i) < d_{-}(i)-c_n}$.
\[
\mathbb{P}\left[ \textit{B}_+(G_n)\neq \emptyset \right]  =\mathbf{P}[Z_+ \geqslant 1] \geqslant \frac{\mathbb{E}\left[Z_+\right]^2}{\mathbb{E}\left[Z_+^2\right]}
\]

With this lower bound on the probability, it reduces to proving the following two limits.
\begin{equation}
n \mathbb{P}\left[ d_+(1) < d_-(1) -c_n\right] \to \infty,
\end{equation}
\begin{equation}
\frac{\mathbb{P}\left[d_+(1) < d_-(1) -c_n \cap d_+(2) < d_-(2) -c_n\right]}{\mathbb{P}\left[d_+(1) < d_-(1) -c_n\right] \mathbb{P}\left[d_+(1) < d_-(1) -c_n\right]}  \to 1 .
\end{equation}

\begin{lemma} If $Z \sim N(0,1)$ and SNR$<1$, then
\[
n \mathbb{P}\left[ d_+(1) < d_-(1) -c_n\right]  =  n \mathbb{P}\left[Z > 
     (1+o(1)) \frac{(\alpha - \beta)\sqrt{\log n}}{2\tau}\right] = n^{1 -\text{SNR} +o(1)} \to  \infty
\]
\end{lemma}

Finally, we prove that the events $E_1 =\{d_+(1) < d_-(1) -c_n\}$ and $E_2= \{ d_+(2) < d_-(2) -c_n\}$ are asymptotically independent,  and  for notational purposes we define $a_n \sim b_n$ if $\lim_{n \to \infty} \frac{a_n}{b_n}=1$. 
\vspace{-.25 cm}
\begin{lemma} Let $Z_1, Z_2,$ and $ Z$ be IID $N(0,1)$ random  variables, and $s= \frac{\alpha -\beta}{2\tau}$. Then we have
\begin{align}
\frac{\mathbb{P}\left[E_1 \cap E_2\right]}{\mathbb{P}\left[E_1\right] \mathbb{P}\left[E_2\right]}
\sim \frac{\mathbb{P}\left[Z_1 > s\sqrt{\log n} +\frac{Z}{\sqrt{n-2}}\cap Z_2 > s\sqrt{\log n} +\frac{Z}{\sqrt{n-2}}\right]}{\mathbb{P}\left[Z_1 > s\sqrt{\log n} +\frac{Z}{\sqrt{n-2}}\right] \mathbb{P}\left[Z_1 > s\sqrt{\log n}+\frac{Z}{\sqrt{n-2}}\right]} \sim 1
\end{align}
\end{lemma}

The proof of the two lemmas above and hence the proof of Theorem ~\eqref{thm:Negative0} depends on the fact that $d_+(1\setminus2), d_+(2\setminus 1), A_{12}, d_-(1), d_-(2)$ are jointly independent collections of random variables and satisfy the following distributional identities.
\[
d_+(1\setminus2):= \sum_{i\in C_+ \setminus 2}A_{i1} \overset{d}{=}N\left(\left(\frac{n}{2}-2\right)\mu_1, \left(\frac{n}{2}-2\right)\tau^2\right),
\]
\[
d_+(2\setminus1):= \sum_{i\in C_+ \setminus 1}A_{i2} \overset{d}{=}N\left(\left(\frac{n}{2}-2\right)\mu_1, \left(\frac{n}{2}-2\right)\tau^2\right),
\]
\[
d_-(1)=\sum_{j\in C_-}A_{1j}\overset{d}{=} N\left(\frac{n}{2}\mu_2, \frac{n}{2}\tau^2\right), \quad
d_-(2)=\sum_{j\in C_-}A_{2j} \overset{d}{=} N\left(\frac{n}{2}\mu_2, \frac{n}{2}\tau^2\right).
\]

Now, we turn towards proving statistical possibility results. Under the critical scaling ~\eqref{def:CritSBM0} in SBM, it was proven in \cite{abbe2014exact}  if $|\sqrt{a} -\sqrt{b}| >\sqrt{2}$ then MLE can exactly recover the true community structure $\sigma^*$ with probability approaching one as $n\to \infty$. In the Gaussian weighted version of the stochastic block model ~\eqref{def:GWSBM0} we establish the following statistical possibility result. 
\begin{theorem}[\textbf{Statistical possibility of exact recovery}] \label{thm: Positive0}
    Let  $G_n \sim$ SBM$(n, \mu_1, \mu_2, \tau^2)$ ~\eqref{def:GWSBM0} with $\mu_1, \mu_2$ in the critical scaling regime ~\eqref{def:CritGWSBM0}, and $\alpha > \beta$, $\tau >0$ be constants. Then MLE can recover the  $\sigma^*$ whenever $\text{SNR} >1$. More precisely, $ \mathbb{P}\left[M(\hat{\sigma}_{\text{MLE}}, \sigma^*)= 1\right] \to 1$ if $\text{SNR} >1$.
\end{theorem}

The proof of this theorem relies on analyzing the inequality $f_A(\hat{\sigma}) > f_A(\sigma^*)$ 
 and observing that  MLE ~\eqref{eq:MLE0} fails to recover the true community labels $\sigma^*$  if and only if there exists $ \hat{\sigma}$ such that $f_A(\hat{\sigma}) > f_A(\sigma^*)$. For a subset of vertices $S_+ \subseteq C_+$ and $S_- \subseteq C_-$ we define 
\[
d_+(S_+) =\sum_{i \in S_+, k \in C_+\setminus S_+}A_{ik}, \quad 
d_+(S_-) =\sum_{j \in S_-, l \in C_-\setminus S_-}A_{jl}
\]
\[
d_-(S_+\setminus S_-) =\sum_{i \in S_+, j \in C_-\setminus S_-}A_{ij}, \quad 
d_-(S_-\setminus S_+) =\sum_{j \in S_-, i \in C_+\setminus S_+}A_{ij}
\]

 \begin{lemma} $\exists$ $\hat{\sigma}  \in \Theta_n$ with  
     $f_A(\hat{\sigma}) > f_A(\sigma^*)$  $\smallimplies$ $\exists  S_+ \subseteq C_+, S_- \subseteq C_-$ such that $1\leqslant |S_+|=|S_-| \leqslant \frac{n}{4}$ and  
     \begin{equation} \label{eq:Pos0}
     d_+(S_+) +d_+(S_-)<  d_-(S_+\setminus S_-) +d_-(S_-\setminus S_+).
     \end{equation} 
 \end{lemma}
 We simply use union bound to prove that MLE ~\eqref{eq:MLE0}  solves exact recovery ~\eqref{thm: Positive0} as soon as \text{SNR} $>1$  by proving that such events of the kind ~\eqref{eq:Pos0} occur with probability approaching zero. We define
 \[
 E:=\{\exists S_+ \subseteq C_+, S_- \subseteq C_- : 1 \leqslant |S_+|=|S_-| \leqslant \frac{n}{4}, d_-(S_-\setminus S_+) + d_-(S_+ \setminus S_-) > d_+(S_+)+d_-(S_-) \} 
 \]
 \[
 E_k= \{\exists S_+ \subseteq C_+, S_- \subseteq C_- : |S_+|=|S_-| =k, d_-(S_-\setminus S_+) + d_-(S_+ \setminus S_-) > d_+(S_+)+d_-(S_-)\}
 \]
 
 \begin{lemma} With the notations above, if SNR$>1$ then $\mathbb{P}\left[E\right] =  \mathbb{P}\left[\cup_{k=1}^{\frac{n}{4}}E_k \right] \leqslant \sum_{k=0}^{\frac{n}{4}}\mathbb{P}\left[E_k\right] \to 0$.
 \end{lemma}
 
 The proof of this lemma relies on that $ d_-(S_-\setminus S_+), d_-(S_+ \setminus S_-), d_+(S_+), d_-(S_-)$ are jointly independent random variables satisfying the following distributional identities.
     \[
     d_+(S_+)\overset{d}{=}  N( k(n/2 -k) \mu_1, k(n/2 -k) \tau^2), \quad 
     d_+(S_-)\overset{d}{=}  N( k(n/2 -k) \mu_1, k(n/2 -k) \tau^2)
     \]
     \[
     d_-(S_+\setminus S_-) \overset{d}{=} N( k(n/2 -k) \mu_2, k(n/2 -k) \tau^2), \quad
     d_-(S_-\setminus S_+) \overset{d}{=} N( k(n/2 -k) \mu_2, k(n/2 -k) \tau^2)
     \]
     
      If SNR$>1$, we have the following for $k=1$ with $\text{SNR}= 1+2\delta$ for some $\delta >0$.
     \[
     \mathbb{P}\left[E_1\right] \leqslant  \frac{n^2}{4}  \mathbb{P}\left[ N(0,1) \geqslant \sqrt{(n/2 -1)} \left(\frac{\mu_1- \mu_2}{\tau}\right)\right] \leqslant \frac{n^2}{4}  e^{-\frac{2 
     (n -2)  \text{SNR}\log  n}{n}}  = \frac{1}{4}  \frac{1}{n^{4\delta + o(1)}} \to 0.
     \]

 We observe that MLE ~\eqref{eq:MLE0} is a discrete optimization problem over an exponential-sized solution space and not a polynomially tractable problem in the worst-case sense, and also in the approximate-case sense \cite{doi:10.1137/050640904}. So, we look for potential algorithmic solutions to recover the ground truth $\sigma^*$. Under the critical scaling ~\eqref{def:CritSBM0} in SBM it was proven in \cite{hajek2016achieving} that a semi-definite relaxation (SDP) of MLE ~\eqref{eq:MLE0} solves exact recovery if $|\sqrt{a} -\sqrt{b}| > \sqrt{2}$. In the GWSBM ~\eqref{def:GWSBM0}, we establish the following algorithmic possibility result whenever SNR$> 1$.
       To start, we observe that MLE ~\eqref{eq:MLE0} can be rewritten as the following problem.
\begin{equation} \label{def:MLE0R}
\max_{Y, \sigma} \langle A, Y \rangle : Y = \sigma \sigma^T, \, Y_{ii} = 1, \, i \in [n], \, \langle \mathbf{J}, Y \rangle = 0. 
\end{equation}

 $Y = \sigma \sigma^T$ is a rank-one positive semi-definite matrix. Relaxing this condition by dropping the rank-one restriction, we obtain the following convex relaxation of ~\eqref{def:MLE0R}:
\begin{equation} \label{eq:SDP0}
\hat{Y}_{\text{SDP}} = \arg\max_{Y} \langle A, Y \rangle : Y \succeq 0, \, Y_{ii} = 1, \, \langle \mathbf{J}, Y \rangle = 0.
\end{equation}

Let $Y^*=\sigma^*\sigma^{*^{\top}}$ and $\mathcal{Y}_n \triangleq\left\{\sigma \sigma^{\top}: \sigma \in\{-1,1\}^n, \langle \sigma ,\mathbf{1}\rangle =0\right\}$. The following result establishes the algorithmic tractability of the exact recovery problem under the  SDP procedure.

\begin{theorem}[\textbf{Semi-definite relaxation achieves exact recovery}]\label{thm:SDP0}
Let $G_n \sim$ SBM$(n, \mu_1, \mu_2, \tau^2)$ ~\eqref{def:GWSBM0} with $\mu_1, \mu_2$ in the critical scaling regime ~\eqref{def:CritGWSBM0}, and $\alpha > \beta$, $\tau >0$ be constants.  If  SNR   $>1$   then   $\min _{Y^* \in \mathcal{Y}_n} \mathbb{P}\left\{\hat{Y}_{\mathrm{SDP}}=Y^*\right\}\to  1  \hspace{.1 cm} \text{ as } \hspace{.1 cm} n \rightarrow \infty.$
\end{theorem}

Following \cite{hajek2016achieving} we prove the theorem by first producing a deterministic sufficient condition for SDP to achieve exact recovery. Let $\mathbf{J}$ denote an $n \times n$ matrix with all entries equal to $1$ and $\lambda_2 (S)$ denote the second smallest eigenvalue of a symmetric $n \times n$ matrix $S$.
\begin{lemma}
 Suppose  $ \exists$ $D^{*}=\operatorname{diag}\left\{d_{i}^{*}\right\}$ with $\{d^*_i\}_{i=1}^{n} \in \mathbb{R}^n$ and $\lambda^{*} \in \mathbb{R}$ such that 
 \begin{equation} \label{eq:DualCert0}
  S^{*} := D^{*}-A+\lambda^{*} \mathbf{J} \hspace{.1 cm} \text{satisfies} \hspace{.1 cm}  S^{*} \succeq 0, \lambda_{2}\left(S^{*}\right)>0, \hspace{.1 cm} \text{and} \hspace{.1 cm} S^{*} \sigma^{*}=0.
 \end{equation}
 
Then SDP recovers the true solution. More precisely, $\widehat{Y}_{\mathrm{SDP}}=Y^{*}$ is the unique solution to ~\eqref{eq:SDP0}.
\end{lemma}

The proof of this lemma follows from the weak duality of semi-definite programming, and now we make a clever choice of $D^*$ and $\lambda^*$ such that $S^*$ satisfies all the conditions above ~\eqref{eq:DualCert0} with probability approaching one.
More precisely, for all $1\leqslant i \leqslant n$ we take $d_{i}^{*}=\sum_{j=1}^{n} A_{i j} \sigma_{i}^{*} \sigma_{j}^{*}$ and choose $\lambda^{*}= \frac{\mu_1+\mu_2}{2}$.  By definition of $S^*= D^* -A+\lambda^*\mathbf{J}$ with $D^*$ and $\lambda^*$ as above, we have $S^*\sigma^* = 0$, and it reduces to proving the following.
\begin{equation}\label{eq:SDPMain0}
  \mathbb{P}\left[S^* \succeq 0\cap \lambda_2(S^{*})>0\right] =\mathbb{P}\left[\inf _{x \perp \sigma^{*},\|x\|=1} \langle x, S^{*} x \rangle >0\right] \to 1
\end{equation}

Simplifying the quantity $\langle x, S^{*} x \rangle$ requires us to prove the following.
\begin{equation}
  \mathbb{P}\left[ \min _{i \in[n]} d_{i}^{*}+\mu_1 -\|A-\mathbb{E}[A]\| >0\right] \to  1
\end{equation}

From  bounds  on norms of Gaussian random matrices \cite{Vershynin_2018}, we have  for an absolute constant $c >0$ $\mathbb{P}\left[\|A-\mathbb{E}[A]\| \geqslant 3\tau\sqrt{n}\right] \leqslant 2e^{-cn} $. Therefore, it is enough to prove the following 
\begin{equation}
    \mathbb{P}[ \min _{i \in[n]} d_{i}^{*}> 4\tau\sqrt{n}] =  \mathbb{P}[ \cap_{i=1}^{n} \{d_{i}^{*}> 4\tau\sqrt{n}\}] \to  1
\end{equation}

We prove this using a union bound argument and observing that  for all $1\leqslant i \leqslant n$
\[
d_{i}^*=\sum_{j=1}^{n} A_{i j} \sigma_{i}^{*} \sigma_{j}^{*} \overset{ d}{=} N\left(\left(\frac{n}{2}-1\right)\mu_1, \left(\frac{n}{2}-1\right)\tau^2\right) - N\left(\frac{n}{2}\mu_2, \frac{n}{2}\tau^2\right),
\]
\[
\implies \sum_{i=1}^{n} \mathbb{P}\left[ d^*_i \leqslant 4\tau \sqrt{n} \right] = n \mathbb{P}\left[ d^*_i \leqslant 4\tau \sqrt{n} \right] \leqslant \frac{1}{n^{\text{SNR} -1 +o(1)}} \to  0.
\]

\par
This proves that there is no information computation gap \cite{bandeira2018notes} in this model as SDP ~\eqref{eq:SDP0} achieves exact recovery until the statistical threshold produced by  MLE ~\eqref{eq:MLE0}. We end the analysis of this model  ~\eqref{def:GWSBM0} by proving that even spectral relaxation of MLE ~\eqref{eq:MLE0} achieves exact recovery until the same threshold.  In SBM, it was proven in \cite{abbe2019entrywise} that a spectral relaxation of MLE ~\eqref{eq:MLE0} solves exact recovery whenever $|\sqrt{a} -\sqrt{b}| > \sqrt{2}$.  In the Gaussian weighted version of the stochastic block model ~\eqref{def:GWSBM0}, we establish the following algorithmic possibility result.
\begin{theorem}[\textbf{Spectral relaxation achieves exact recovery}] \label{thm:Spec0}
    Let $G_n \sim$ SBM$(n, \mu_1, \mu_2, \tau^2)$ ~\eqref{def:GWSBM0}  with $\mu_1, \mu_2$ in the critical scaling regime ~\eqref{def:CritGWSBM0},  $\alpha > \beta$, $\tau >0$ be constants, and we assume that $A_{ii} \overset{IID}{\sim} N(\mu_1, \tau^2)$.  Let $B:=  A - \frac{n (\mu_1 +\mu_2)}{2} \frac{\mathbf{1}_n}{\sqrt{n}}\frac{\mathbf{1}_n}{\sqrt{n}}^{\top}$. If $\text{SNR}$ $>1$  then $\hat{\sigma}:= \text{sign} (u_1)$ achieves exact recovery where $u_1$ is the eigenvector corresponding to the largest eigenvalue $\lambda_1$ of $B$. More precisely, if \text{SNR} $>1$ then $\mathbb{P} \left[ M(\hat{\sigma}, \sigma^*) =1\right] \to 1$.
\end{theorem}

 The proof of this theorem is based on a more general spectral separation result obtained in \cite{abbe2019entrywise}[Theorem 1.1] and we defer the detailed proof of the above result in the appendix. Now we analyze a Gaussian weighted planted dense subgraph model of linear proportional size.
\section{Gaussian weighted planted dense subgraph model: description}
In comparison to the SBM, the binary version of the planted dense subgraph problem \cite{ariascastro2013community} 
has a single  (denser) cluster of size $\gamma n$ for  $\gamma \in (0,1)$, and vertices within the planted community connect with probability $p$ and every other pair connect with probability $q$.
In the Gaussian weighted planted dense subgraph model (GWPDSM) denoted as PDSM$(n, \mu_1,\mu_2, \tau^2)$ we have a weighted graph $G_{n}$ on $n$  vertices with binary labels $(\zeta_i^*)_{i\in [n]}$ such  that  
\[
\#\{i\in [n]: \zeta_i^*=1\}= \gamma n, \quad 
\#\{i\in [n]: \zeta_i^*=0\}= (1-\gamma) n
\] 

and conditioned on this labeling $(\zeta_i^*)_{i\in [n]}$, for each pair  $\{i,j\}$, we have a weighted edge 
\begin{equation}\label{def:GWPDSM0}
A_{ij}=A_{ji} \sim
\begin{cases}
  N(\mu_1, \tau^2) & \text{if } \zeta^*(i) =\zeta^*(j) =1   \\
   N(\mu_2, \tau^2) & \text{otherwise}   
\end{cases}
\end{equation} 

sampled independently for every distinct pair $\{i,j\}$.  To simplify the writing, we focus on the assortative case, $\mu_1>\mu_2$ with the critical scaling ~\eqref{def:CritGWSBM0}  of $\mu_1= \alpha \sqrt{\frac{\log n}{n}}, \mu_2= \beta\sqrt{\frac{\log n}{n}}.$ We further have our SNR  given by $\text{SNR}:= \frac{(\alpha -\beta)^2}{8\tau^2}$ with $\alpha >\beta $ and $ \tau >0$ constants.
\par
 The problem of exact recovery of the densely weighted community asks to recover the planted community labeling $(\zeta_i^*)_{i\in [n]}$ exactly having observed an instance of the weighted graph $G_n\sim$ PDSM$(n, \mu_1\,\mu_2, \tau^2)$. Let us  denote the parameter space $\Theta^{\gamma}_n$ with
 $\mathbf{1}_n := (1, \cdots, 1)\in \mathbb{R}^n$ as 
 \begin{equation} \label{eq:SOL1}
 \Theta^{\gamma}_n = \{\zeta\in \{0, 1\}^n, \langle\zeta, \mathbf{1}_n\rangle=\gamma n\},
 \end{equation}
 
 We define the agreement ratio between  $\hat{\zeta},  \zeta  \in \{0, 1\}^n$ as     $M(\hat{\zeta}, \zeta)= \frac{1}{n}\sum_{i=1}^n \mathbf{1}_{\hat{\zeta}(i)=\zeta(i)} $.

Then, we say that the problem of exact recovery is solvable if having observed an instance of the weighted random graph $G_n \sim$ PDSM$(n,\mu_1, \mu_2,\tau^2)$ with an unobserved $\zeta^* \in \Theta^{\gamma}_n$, we output a community label $\hat{\zeta} \in \Theta^{\gamma}_n $ such that $\p (M(\hat{\zeta},\zeta^*)=1) \to 1.$  We say the problem of exact recovery is unsolvable otherwise. That is, with probability bounded away from zero, every estimator fails to exactly recover the densely weighted community, as the number of vertices grows.
\section{Gaussian weighted planted dense subgraphs: results and proof sketch}
 In this section, we present our results along with the ideas of the proof for GWPDSM in comparison with the works \cite{hajek2016achieving} \cite{hajek2016achieving1} that established an important set of results for the binary version. Detailed proofs of all the theorems and lemmas are provided in the appendix.
 \par
 To explore information theoretic limits, we analyze the maximum likelihood estimator (MLE).
 \vspace{-0.4 em}
\begin{equation} \label{eq:MLEPDSM0}
   \hat{\zeta}_{MLE}= \underset{\zeta \in \Theta^{\gamma}_n}{ \argmax} \hspace{.1 cm}g_A(\zeta) \quad \quad \text{where } \quad g_A(\zeta):=\sum_{i,j} A_{ij}\zeta_i \zeta_j 
   \vspace{-0.4 em}
\end{equation}

Under the critical scaling ~\eqref{def:CritSBM0} in the binary version, it was proven in \cite{hajek2016achieving} that below a certain threshold, exact recovery of the planted community is not solvable. In particular, MLE ~\eqref{eq:MLEPDSM0} fails to recover the ground truth $\zeta^*$ with high probability  as $n \to  \infty$. In GWPDSM ~\eqref{def:GWPDSM0} we establish the following statistical impossibility result.
\begin{theorem}[\textbf{Statistical impossibility of exact recovery of the planted community}] \label{thm:Negative1'}
     Let $\gamma \in  (0,1)$ and $G_n \sim $ PDSM$(n, \mu_1, \mu_2, \tau^2)$   with $\mu_1, \mu_2$ in the critical scaling regime ~\eqref{def:CritGWSBM0}, and $\alpha > \beta$, $\tau >0$ be constants.
      If  $\gamma$ SNR$<\frac{3}{4}$ then MLE fails in recovering the community label $\zeta^*$ with high probability
     \[ 
     \underset{n \to \infty}{\limsup}\hspace{.1cm}\p(M(\hat{\zeta}_{MLE}, \zeta^*)=1) = 0 \quad  \iff \quad  \underset{n \to \infty}{\liminf}\hspace{.1cm}\p(M(\hat{\zeta}_{MLE}, \zeta^*)\neq 1) = 1.
     \]
\end{theorem}

  We prove this by establishing that whenever $\gamma$SNR$<\frac{3}{4}$, with probability approaching one we can find a vertex $i$ in the planted community and a vertex $j$ from the rest such that swapping them increases the likelihood ~\eqref{eq:MLEPDSM0} and hence causes MLE to fail to recover the ground truth $\zeta^*$ exactly. Without loss of generality, we assume that the planted community is given by $\zeta^* =(\mathbf{1}_{\gamma n}, \mathbf{0}_{(1-\gamma)n}) \in \Theta_n^{\gamma}$, and $G_n \sim $ PDSM$(n, \mu_1, \mu_2, \tau^2)$ conditioned on $\zeta^*$. We denote the planted community  as
    \[C^* := \left\{ i \in [n]: \zeta_i^*=1\right\}=\left\{1, \cdots, \gamma n\right\}, \quad
    [n] \setminus C^*:=\left\{\gamma n +1, \cdots, n\right\}.\]
    \vspace{-.5 cm}
    \begin{definition} 
       We define the set of bad pairs of vertices in $G_n$ with $\zeta^*[i\leftrightarrow j]$ denoting the vector obtained by swapping the values of coordinates $i$ and $j$ in $\zeta^*$.
    \begin{equation} \label{def:Bad1}
    \mathscr{C}(G_n) := \{(i,j): i\in C^*, j\in [n] \setminus C^*, g_A(\zeta^*)< g_A(\zeta^*[i\leftrightarrow j])\},
    \end{equation}
\end{definition}

From ~\eqref{eq:MLEPDSM0} we observe that $g_A(\zeta^*)< g_A(\zeta^*[i\leftrightarrow j]) $ is the same as the event that the likelihood of the partition $\zeta^*[i\leftrightarrow j]$ is larger than the planted one $\zeta^*$. So it reduces to the following.
\begin{lemma}
   MLE fails to solve the exact recovery with high probability if $\mathscr{C}(G_n)$ is non-empty with high probability . More precisely, we have the following 
   \[
   \underset{n \to \infty}{\liminf}\hspace{.1cm}\p(M(\hat{\zeta}_{MLE}, \zeta^*)\neq1) \geqslant \underset{n \to \infty}{\liminf}\hspace{.1cm} \p (\mathscr{C}(G_n)\neq \emptyset)
   \]
\end{lemma}

We analyze $g_A(\zeta^*)< g_A(\zeta^*[i\leftrightarrow j]) $ and show that 
$\underset{n \to \infty}{\lim}\hspace{.1cm}\p (\mathscr{C}(G_n)\neq \emptyset) =1.$
\begin{definition}
    For a vertex $i \in C^*$, and $j \in [n]\setminus C^*$ we define the following.
    \[
    e(i, C^*) = \sum_{ k \in C^*} A_{ik} =e(i, C^* \setminus \{i\}), \quad 
    e(j, C^* \setminus \{i\}) = \sum_{ k \in C^*  \setminus \{i\}} A_{jk}
    \]
\end{definition}

It is worth observing that $e(i, C^*)$  sums the edge weights that $i $ has with $C^*$.
\begin{lemma}
With the notations above and $i \in C^*$, and $j \in [n]\setminus C^*$ we have the following.
    \[
    g_A(\zeta^*)< g_A(\zeta^*[i\leftrightarrow j]) \quad  \iff \quad  e(i, C^*) < e(j, C^* \setminus \{i\}).
    \]
\end{lemma}  

With the lemma above we apply the second-moment method to $Z = \sum_{ i \in C^*, j \in [n]\setminus C^*} Z_{ij}$ with
\begin{equation}
Z_{ij}=\mathbf{1}_{\mathscr{C}_{ij}}
\quad \text{with} \quad 
\mathscr{C}_{ij} : =\{e(i, C^*) < e(j, C^*\setminus \{i\})\}
\end{equation}
\begin{equation}
\mathscr{C}(G_n) =\{(i,j): i \in C^*, j \in [n]\setminus C^*, e(i, C^*)< e(j,C^*\setminus \{i\})\} = \cup_{ i \in C^*, j \in [n]\setminus C^* } \mathscr{C}_{ij}.
\end{equation}

We have $\mathbb{P}\left[\mathscr{C}(G_n) \neq \emptyset\right] =\mathbb{P}\left[Z \geqslant 1\right]$ and it reduces to proving $\mathbb{P}\left[Z \geqslant1 \right] \to 1$ if $ \gamma \text{SNR} <\frac{3}{4}$. Applying the second-moment method to $Z$ it further reduces to proving the following two limits.
\begin{equation}
  \sum_{(i,j) \in C^* \times [n]\setminus C^*} \mathbb{P}\left[\mathscr{C}_{ij}\right]\to \infty
  \end{equation}
  \begin{equation} \label{eq:SMC'}
   \frac{\sum_{(i_1, j_1) \neq (i_2, j_2)   \in  C^* \times [n]\setminus C^*} \mathbb{P}\left[ \mathscr{C}_{(i_1, j_1)}\cap \mathscr{C}_{(i_2, j_2)}\right]}{\left(\sum_{(i,j) \in C^* \times [n]\setminus C^*} \mathbb{P}\left[\mathscr{C}_{ij}\right]\right)^2} \to 1 .
  \end{equation}
\begin{lemma} \label{lem:First1'}
 Let $C = C(\alpha, \beta, \gamma, \tau): = \frac{\gamma (1- \gamma )\tau \sqrt{2} }{(\alpha -\beta) \sqrt{2\pi \gamma }} $. If $ \gamma \text{SNR} < 1$ then for $Z,Z_1$ IID $N(0,1)$    we have
\begin{equation} \label{eq:First1'}
\sum_{(i,j) \in C^* \times [n]\setminus C^*} \mathbb{P}\left[\mathscr{C}_{ij}\right] \sim  \gamma (1-\gamma)n^2 \mathbb{P}\left[ \left\{Z_1  > \frac{Z}{ \sqrt{2\gamma n -3}} + s \sqrt{\log n}\right\}\right] \sim C \frac{n^{2(1-\gamma \text{SNR})}}{\sqrt{\log n}} \to \infty.
\end{equation}
\end{lemma}

 The proof relies on that fact that $e(i, C^*)$ and $e(j, C^* \setminus \{i\})$ are independent and  we have 
 \begin{equation}
 e(i, C^*) \overset{d}{=} N((\gamma n -1)\mu_1,(\gamma n -1)\tau^2), \quad 
 e(j, C^*\setminus \{i\}) \overset{d}{=} N((\gamma n -1)\mu_2,(\gamma n -1)\tau^2)
 \end{equation}
 \vspace{-.5 cm}
\begin{lemma} \label{lem:Second1'}
If $ \gamma \text{SNR} <  \frac{3}{4}$ then, we have the following   
\[
\frac{\sum_{(i_1, j_1) \neq (i_2, j_2)   \in  C^* \times [n]\setminus C^*} \mathbb{P}\left[ \mathscr{C}_{(i_1, j_1)}\cap \mathscr{C}_{(i_2, j_2)}\right]}{\left(\sum_{(i,j) \in C^* \times [n]\setminus C^*} \mathbb{P}\left[\mathscr{C}_{ij}\right]\right)^2} \to 1 
\]
\end{lemma}

The proof builds on the fact that for $i_1 \neq i_2 \in C^* $ and $j_1 \neq j_2 \in [n] \setminus C^*$ we have   $e_{i_1, i_2}, e(i_1, C^* \setminus \{i_2\}), e(i_2, C^* \setminus \{i_1\}),  e(j_1, C^* \setminus \{i_1\}), e(j_2, C^* \setminus \{i_2\})$ are jointly independent Gaussians with
\[
e(i_1, C^* \setminus \{i_2\}) \overset{d}{=} N( (\gamma n -2)\mu_1, (\gamma n-2) \tau^2), \quad
e(i_2, C^* \setminus \{i_1\}) \overset{d}{=} N( (\gamma n -2)\mu_1, (\gamma n-2) \tau^2)
\]
\[
e(j_1, C^* \setminus \{i_1\}) \overset{d}{=}  N( (\gamma n -1)\mu_2, (\gamma n-1) \tau^2), \quad
e(j_2, C^* \setminus \{i_2\}) \overset{d}{=} N( (\gamma n -1)\mu_2, (\gamma n-1) \tau^2)
\]

\begin{lemma}
Let  $Z_1, Z_2,$ and $ Z$ be IID $N(0,1)$ with $s=\frac{ (\alpha -\beta) \sqrt{\gamma}}{\tau \sqrt{2}}$. For $i_1\neq i_2 \in C^*, j_1\neq j_2 \in [n]\setminus C^*$
    \begin{equation}\label{eq:SM1'}
\frac{\mathbb{P}[\mathscr{C}_{(i_1, j_1)}\cap \mathscr{C}_{(i_2, j_2)}]}{\mathbb{P}[\mathscr{C}_{i_1,j_1}]\mathbb{P}[\mathscr{C}_{i_1,j_1}]} \sim  \frac{\mathbb{P}  \left[ \left\{Z_1  > \frac{Z}{ \sqrt{2\gamma n -3}} + s \sqrt{\log n}\right\} \cap  \left\{Z_2  > \frac{Z}{ \sqrt{2\gamma n -3}} + s \sqrt{\log n}\right\} \right]}{\mathbb{P}  \left[ Z_1  > \frac{Z}{ \sqrt{2\gamma n -3}} + s \sqrt{\log n}\right] \mathbb{P}  \left[Z_2  > \frac{Z}{ \sqrt{2\gamma n -3}} + s \sqrt{\log n}\right]} \sim 1 
   \end{equation}
\end{lemma}
 So, the leading term  in ~\eqref{eq:SMC'} is   $\sum_{i_1 \neq i_2  \in  C^*,j_1\neq  j_2 \in [n]\setminus C^*} \mathbb{P}\left[ \mathscr{C}_{(i_1, j_1)}\cap \mathscr{C}_{(i_2, j_2)}\right] \sim \left(\sum_{(i,j) \in C^* \times [n]\setminus C^*} \mathbb{P}\left[\mathscr{C}_{ij}\right]\right)^2.$ It is rather surprising that the subleading orders in the sum ~\eqref{eq:SMC'} does  go to zero only if $\gamma \text{SNR} < \frac{3}{4}$.
\[
\frac{\sum_{i_1 =i_2  \in  C^*,j_1\neq  j_2 \in [n]\setminus C^*} \mathbb{P}\left[ \mathscr{C}_{(i_1, j_1)}\cap \mathscr{C}_{(i_2, j_2)}\right]}{\left(\sum_{(i,j) \in C^* \times [n]\setminus C^*} \mathbb{P}\left[\mathscr{C}_{ij}\right]\right)^2} 
\sim
\frac{\mathbb{P}\left[  \mathscr{C}_{(i, j_1)}\cap \mathscr{C}_{(i, j_2)}\right]}{ \gamma n \mathbb{P}\left[\mathscr{C}_{ij}\right]^2} 
\sim
0  \quad ( \text{if} \hspace{.1 cm} \gamma \text{SNR} < \frac{3}{4})
\]
\[
\frac{\sum_{i_1 \neq i_2  \in  C^*,j_1 =  j_2 \in [n]\setminus C^*} \mathbb{P}\left[ \mathscr{C}_{(i_1, j_1)}\cap \mathscr{C}_{(i_2, j_2)}\right]}{\left(\sum_{(i,j) \in C^* \times [n]\setminus C^*} \mathbb{P}\left[\mathscr{C}_{ij}\right]\right)^2}
\sim
\frac{\mathbb{P}\left[  \mathscr{C}_{(i_1, j)}\cap \mathscr{C}_{(i_2, j)}\right]}{(1-\gamma )n \mathbb{P}\left[\mathscr{C}_{ij}\right]^2}  \sim
0 \quad (\text{if} \hspace{.1 cm} \gamma \text{SNR} < \frac{3}{4})
\]

 This is because for $i \in C^*$ and $j_1 \neq j_2 \in [n] \setminus C^*$ or $i_1 \neq i_2 \in C^*$ and $j \in [n] \setminus C^*$ we have  \footnote{$ \Phi(z) = \p [N(0,1) > z] =\int_{z}^{\infty} \phi(t) dt 
 \text{  with  } \phi(t) =\frac{1}{\sqrt{2\pi}}e^{-\frac{t^2}{2}}$} \footnote{ $a_n \simeq b_n \implies a_n \leqslant \text{ poly }(\log n) b_n  \text{  and  } b_n \leqslant \text{ poly } (\log n) a_n $}
 \begin{align*}
 \frac{\mathbb{P}[\mathscr{C}_{(i, j_1)}\cap \mathscr{C}_{(i, j_2)}]}{n \p [\mathscr{C}_{ij}]^2}
 \sim 
 \frac{\mathbb{P}[\mathscr{C}_{(i_1, j)}\cap \mathscr{C}_{(i_2, j)}]}{n \p [\mathscr{C}_{ij}]^2} 
 &
 \sim
 \frac{\mathbb{P}  \left[ \left\{Z_1 > Z + s\sqrt{2 \log n}  \right\} \cap \left\{Z_2 >  Z + s\sqrt{2 \log n} \right\} \right]}{n \mathbb{P}  \left[ Z_1 > Z + s\sqrt{2 \log n}\right]^2 } 
 \\
 &
 \sim 
 \frac{\int_{\mathbb{R}} \Phi^2(z +s\sqrt{2 \log n}) \phi(z) dz}{ n \left(\frac{1}{ s \sqrt{\log n}}\phi (s \sqrt{\log n})\right)^2} 
 \hspace{.1 cm} 
 \simeq
  n^{\frac{4 \gamma \text{SNR}}{3} -1} 
 \to 0
 \end{align*}
 
 Now, we turn towards proving statistical possibility results. In Gaussian weighted planted dense subgraph model ~\eqref{def:GWPDSM0} we establish the following statistical possibility result if $\gamma$SNR$>1$.
 \begin{theorem}[\textbf{Statistical possibility of exact recovery of the planted community}] \label{thm: Positive1'}
    Let $G_n\sim$ \\ PDSM$(n, \mu_1, \mu_2, \tau^2)$ ~\eqref{def:GWPDSM0}, MLE  recovers  $\zeta^*$ if $ \gamma \text{SNR} >1$. More precisely, $ \mathbb{P}[M(\hat{\zeta}_{\text{MLE}}, \zeta^*)= 1] \to 1$.
\end{theorem}

The proof of this theorem relies on analyzing the inequality $g_A(\hat{\zeta}) > g_A(\zeta^*)$ and observing that  MLE ~\eqref{eq:MLEPDSM0} fails to recover  $\zeta^*$  if and only if there exists $ \hat{\zeta} \in \Theta_n^{\gamma}$ such that $g_A(\hat{\zeta}) > g_A(\zeta^*)$.
For  subsets of vertices $A,B \subseteq [n]$ with $A\cap B = \emptyset$ we define the following
\[
e(A, B) =\sum_{p \in A,q \in B} A_{pq}, \quad 
e(A) = \sum_{p \in A, q \in A} A_{pq}.
\]
\vspace{-.5 cm}
\begin{lemma} $\exists$ $\hat{\zeta} \in \Theta_n^{\gamma}$ with  
     $g_A(\hat{\zeta}) > g_A(\zeta^*)$  $\smallimplies$ $\exists  S_+ \subseteq C^*, S_- \subseteq [n] \setminus C^*$ such that $ 1\leqslant |S_+| =|S_-| \leqslant \min(\gamma n, (1-\gamma)n)$ and 
     \begin{equation} \label{eq:Pos1'}
     e(S_-) +
     e(C^* \setminus S_+ , S_-) > 
     e(S_+)+
     e(C^* \setminus S_+, S_+)
     \end{equation} 
 \end{lemma}
 
 We now prove that MLE ~\eqref{eq:MLEPDSM0}  solves exact recovery ~\eqref{thm: Positive1'} as soon as  $\gamma$ \text{SNR} $>1$  by proving that such events of the kind ~\eqref{eq:Pos1'} occur with probability approaching zero. We define 
 \[
 \begin{split}
E:=\{  \exists S_+ \subseteq C^*,& S_- \subseteq [n] \setminus C^* : 1\leqslant |S_+|=|S_-| \leqslant 
 \min(\gamma n, (1-\gamma)n), \\  &   e(S_-) + e(C^* \setminus S_+ , S_-) > 
 e(S_+) + e(C^* \setminus S_+, S_+) \}
\end{split}
 \]
 \[
 E_l= \{\exists S_+ \subseteq C^*, S_- \subseteq [n] \setminus C^* : |S_+|=|S_-| =l, e(S_-) +
     e(C^* \setminus S_+ , S_-) > 
     e(S_+)+
     e(C^* \setminus S_+, S_+)\}
 \]
 \vspace{-.5 cm}
 \begin{lemma}\label{lem:Second2'}
 Let $K=\gamma n$. If $\gamma$ \text{SNR} $>1$ then $\mathbb{P}\left[E\right] =  \mathbb{P}\left[\cup_{l=1}^{\min(K, n-K)}E_l \right] \leqslant \sum_{l =1}^{\min(K, n-K)} \mathbb{P}[E_l] \to 0.$
 \end{lemma}
 
The proof of this lemma follows from observing that  $e(S_-), e(S_+), e(C^* \setminus S_+, S_-),$  and $e(C^*\setminus S_+, S_+)$ are jointly independent random variables with 
     \[
     e(S_+)\overset{d}{=}  N\left( \binom{l}{2} \mu_1, \binom{l}{2} \tau^2\right), \quad 
     e(S_-)\overset{d}{=}  N\left( \binom{l}{2} \mu_2, \binom{l}{2} \tau^2\right)
     \]
     \[
     e(C^*\setminus S_+, S_+) \overset{d}{=} N\left( l(K -l) \mu_1, l(K -l) \tau^2 \right), \quad 
     e(C^*\setminus S_+, S_-) \overset{d}{=} N\left( l(K -l) \mu_2, l(K -l) \tau^2\right).
     \]
     
    From  Gaussian tail bounds \cite{Vershynin_2018} we have  for $l=1$ with $ \gamma \text{SNR}= 1+2\delta$.
     \[ 
     \mathbb{P}\left[E_1\right] \leqslant 
     K(n-K) \frac{1}{n^{ \frac{2(K-1) \text{SNR}}{n}}}\leqslant \gamma (1-\gamma) n^2 n^{-2 \gamma  
       \text{SNR} (1+o(1))} =  \frac{\gamma (1-\gamma)}{n^{4\delta +o(1)} }  \to 0
     \] 
     
     We finish this section by establishing that SDP achieves exact recovery of the planted community as soon as $\gamma $SNR$>1$ in comparison to the work \cite{hajek2016achieving} in the binary version. To start, we observe that MLE ~\eqref{eq:MLEPDSM0} can be rewritten as the following discrete optimization problem.
\begin{equation}\label{eq:MLEPDSM0'}
\max_{Z, \zeta} \hspace{.1 cm} \langle A, Z \rangle :  Z =\zeta \zeta^{\top},  Z_{ii} \leqslant  1, \hspace{.1 cm} \forall 
 \hspace{.1 cm} i \in [n],  Z_{ij} \geqslant  0, \hspace{.1 cm} \forall \hspace{.1 cm} i, j \in[n], \langle\mathbf{I}, Z\rangle=\gamma n,  \langle\mathbf{J}, Z \rangle =\gamma^2 n^2.
\end{equation} 
 $Z=\zeta \zeta^{\top}$ is positive semidefinite and rank-one. Removing the rank-one restriction leads to the following convex relaxation of ~\eqref{eq:MLEPDSM0'}, which is a semidefinite program.
\begin{equation} \label{def:SDPPDSM0}
 \hat{Z}_{\mathrm{SDP}}=\underset{Z}{\arg \max } \langle A, Z\rangle:  Z \succeq 0,  Z_{ii} \leqslant 1, \hspace{.1 cm} \forall \hspace{.1 cm} i \in [n], Z_{ij} \geqslant 0, \hspace{.1 cm} \forall \hspace{.1 cm} i, j \in[n], \langle\mathbf{I}, Z\rangle=\gamma n, \langle\mathbf{J}, Z\rangle=\gamma^2 n^2 .
\end{equation}

Let $Z^*=\zeta^*\zeta^{*^{\top}}$ correspond to the true cluster and define $\mathcal{Z}_n=\left\{\zeta \zeta^{\top}: \zeta \in\{0,1\}^n, \langle \zeta, \mathbf{1} \rangle =\gamma n\right\}$. 
\begin{theorem}[\textbf{Semi-definite relaxation achieves exact recovery of the planted community}] \label{thm:SDP1'}
    Let $G_n \sim$ PDSM$(n, \mu_1, \mu_2, \tau^2)$  and $\gamma \in  (0,1).$ If  $\gamma \text{SNR} \hspace{.05 cm}>1  \hspace{.05 cm} \text{then} \hspace{.05 cm}
\min _{Z \in \mathcal{Z}_n} \mathbb{P}\left[\hat{Z}_{\mathrm{SDP}}=Z^*\right]\to 1 \hspace{.1 cm} \text{as} \hspace{.1 cm} n \rightarrow \infty.
$
\end{theorem} 

Following \cite{hajek2016achieving} we prove this theorem by first producing a deterministic sufficient condition for SDP to achieve exact recovery. Let $\mathbf{I}$ be $n\times n$ identity matrix.
\begin{lemma}\label{lem:DetPDSM0}
 Suppose $\exists$  $D^{*}=\operatorname{diag}\left\{d_{i}^{*}\right\}$ with $\{d^*_i\}_{i=1}^{n} \in  \mathbb{R}_{\geqslant 0}^{n}$,  $B^{*} \in \mathbb{R}^{n\times n}_{\geqslant 0},$    and $\lambda^{*}, \eta^{*} \in \mathbb{R}$ such that 
 \begin{equation}\label{eq:First0}
 S^{*} \triangleq D^{*}-B^{*}-A+\eta^{*} \mathbf{I}+\lambda^{*} \mathbf{J} \hspace{.1 cm} \text{satisfies} \hspace{.1 cm} S^{*} \succeq 0, \lambda_{2}\left(S^{*}\right)>0, \hspace{.1 cm} \text{and}  \hspace{.1 cm}S^{*}  \zeta^{*}  =0,
 \end{equation}
 \begin{equation}\label{eq:Second0}
 \hspace{.1 cm}
d_{i}^{*}\left(Z_{i i}^{*}-1\right)  =0 \quad \forall i, \quad 
B_{i j}^{*} Z_{i j}^{*}  =0 \quad \forall i, j,
 \end{equation}

Then SDP recovers the true solution. More precisely, $\widehat{Z}_{\mathrm{SDP}}=Z^{*}$ is the unique solution to ~\eqref{def:SDPPDSM0}.
\end{lemma}

The proof of this lemma follows from the weak duality of semi-definite programming, and now we make a clever choice of $B^*, D^*$  with  $\lambda^* = \frac{\mu_1+\mu_2}{2}$  and $ \eta^* =2\|A-\mathbb{E}[A]\|$  satisfying all the conditions in  ~\eqref{lem:DetPDSM0} with probability approaching one. We take $D^{*}=\operatorname{diag}\left\{d_{i}^{*}\right\}$ with \footnote{ $A_i$ refers to the $i$th row of A} $d^*_i:= \langle (A-\eta^*\mathbf{I}- \lambda^*\mathbf{J})_i, \zeta^*\rangle \mathbf{1}_{i \in C^*}$ and   $B^{*}_{ij}:
=b_{i}^{*} \mathbf{1}_{\left\{i \notin C^{*}, j \in C^{*}\right\}}+b_{j}^{*} \mathbf{1}_{\left\{i \in C^{*}, j \notin C^{*}\right\}}$ with  $b_{i}^{*} \gamma n :=  \langle (\lambda^* \mathbf{J} -A)_i, \zeta^* \rangle \mathbf{1}_{i \notin C^{*}}.$
\par 
It follows immediately that $d_{i}^{*}\left(Z_{i i}^{*}-1\right)  =0, 
B_{i j}^{*} Z_{i j}^{*}  =0$  $\forall$ $i, j,$ and $S^*\zeta* =0$. It remains to prove 
\[
\min \left(\mathbb{P}\left[\cap_{i \notin C^*} \{b_i^*\geqslant 0\} \right ], \quad  \mathbb{P}\left[\cap_{i \in C^*} \{d_i^*\geqslant 0\}\right], \quad  \mathbb{P}\left[ \{S^* \succeq 0\} \cap \{\lambda_{2}\left(S^{*}\right)>0\}\right] \right)\to 1
\]

We prove this simply using union bound by observing that the following holds.
 \begin{equation}
 d^*_i = X_i - 6\tau \sqrt{n} - \lambda^{*} \gamma n \geqslant 0  \hspace{.1 cm} \forall \hspace{.1 cm}  i \in C^*, \quad 
 b_i^* \gamma n = \lambda^{*} \gamma n - X_i  \geqslant 0 \hspace{.1 cm} \forall \hspace{.1 cm} i \notin C^*
 \end{equation}
 \[ \text{where  } X_i \sim N(\gamma n \mu_1, \gamma n \tau^2) \mathbf{1}_{i \in C^*} +
      N(\gamma n \mu_2, \gamma n \tau^2) \mathbf{1}_{i \notin C^*}\]
      
   We now turn to the comparison of the two models discussed and establish that GWSBM ~\eqref{def:GWSBM0} is a strictly easier problem to solve than GWPDSM ~\eqref{def:GWPDSM0} with parameter $\gamma =\frac{1}{2}$.
\section{A comparison of exact recovery in two models and conclusion}

We observe that Theorems ~\eqref{thm:Negative0}, ~\eqref{thm: Positive0}, ~\eqref{thm:SDP0}, ~\eqref{thm:Spec0} obtained for GWSBM ~\eqref{def:GWSBM0} prove that if $\text{SNR} <1$, then the task of exact recovery of two balanced communities is impossible, even statistically. But, as soon as $\text{SNR} >1$,  the task of exact recovery is possible, even algorithmically. On the other hand Theorems ~\eqref{thm:Negative1'}, ~\eqref{thm: Positive1'}, ~\eqref{thm:SDP1'} obtained for GWPDSM (with $\gamma = \frac{1}{2}$) ~\eqref{def:GWPDSM0} prove that whenever $\text{SNR} < \frac{3}{2}$, the task of exact recovery of one planted community of size  $\frac{n}{2}$  is statistically impossible. But, as soon as $\text{SNR} >2$, then the task of exact recovery of one planted community of size $\frac{n}{2}$ is algorithmically possible.  So, when $1<\text{SNR} < \frac{3}{2}$, the problem of exact recovery of two balanced communities is algorithmically solvable, whereas the problem of exact recovery of one planted community of size $\frac{n}{2}$ is even statistically unsolvable. This proves that indeed community detection is a strictly easier problem than the planted dense subgraph model as considered in the models ~\eqref{def:GWSBM0}, ~\eqref{def:GWPDSM0}. It is also interesting to see the inability of first and second-moment methods to match the upper ~\eqref{lem:Second2'} and lower bounds ~\eqref{lem:Second1'} for the Gaussian weighted planted dense subgraph problem.

\bibliographystyle{alpha}
\bibliography{ref}
    \section{Organization}
In this supplementary document, we prove all the Theorems and Lemmas in the article in detail.
\subsection{Gaussian weighted stochastic block model}
Proof of the theorem \ref{thm: Negative} on the statistical impossibility of exact recovery for GWSBM  culminates with the proofs of lemmas  \ref{lem: MLE}, \ref{lem:MLEa}, \ref{lem:MLEb}, \ref{lem:Bto1}, \ref{lem:MLEc}, \ref{lem:MLEd}, \ref{lem:MLEe}, \ref{lem:MLEf}, and \ref{lem:jointGauss}.  Proof of the theorem \ref{thm: Positive} on the statistical possibility of exact recovery for GWSBM culminates with the proofs of lemmas  \ref{lem:MLEg} and \ref{lem:MLEh}. Proof of the theorem  \ref{thm:SDP} on semi-definite relaxation achieving exact recovery for GWSBM culminates with the proofs of lemmas  \ref{lem:exist}, \ref{lem:uniq}, and establishing a high probability event \ref{sec:hpe}. Proof of the theorem  \ref{thm:Spec1} on  spectral relaxation achieving exact recovery for GWSBM culminates with the proof of  the $l_{\infty}$ eigenvector bound using a spectral separation result \ref{thm:Spec} due to \cite{abbe2019entrywise}.
\subsection{Gaussian weighted planted dense subgraph model}
Proof of the theorem \ref{thm:Negative1}  on statistical impossibility of exact recovery of the planted community for GWPDSM culminates with the proofs of lemmas \ref{lem: MLE1}, \ref{lem:MLEa'}, \ref{lem:MLEb'}, \ref{lem:MLEc'}, \ref{lem:First1}, \ref{lem:Second1}. Proof of the theorem \ref{thm: Positive1} on statistical possibility of exact recovery of the planted community for GWPDSM culminates with the proofs of lemmas  \ref{lem:MLEpos1}, \ref{lem:MLEpos2}. Proof of the theorem \ref{thm:SDP1} on semi-definite relaxation achieving exact recovery for GWPDSM culminates with the proofs of lemmas \ref{lem:SDPeuiv}, \ref{lem:DetPDSM}, and establishing a high probability event \ref{sec:hpe1}.
    \section{Gaussian weighted stochastic block model}

This section defines Gaussian weighted Stochastic block model with two symmetric communities.

\subsection{Model description}

Let SBM$(n, \mu_1,\mu_2, \tau^2)$ denote a weighted graph $G_{n}$ on $n$ (even) vertices with binary labels $(\sigma_i^*)_{i\in [n]}$ such  that  

\[\#\{i\in [n]: \sigma_i^*=1\}=\#\{i\in [n]: \sigma_i^*=-1\}=\frac{n}{2}\] 

and conditioned on this labeling $(\sigma_i^*)_{i\in [n]}$, for each pair of distinct nodes $i,j \in [n]$, we have a weighted edge 
\begin{equation}\label{def:GWSBM}
A_{ij}=A_{ji} \sim
\begin{cases}
  N(\mu_1, \tau^2) & \text{if } \sigma^*(i) =\sigma^*(j)  \\
   N(\mu_2, \tau^2) & \text{if }  \sigma^*(i)\neq \sigma^*(j) 
\end{cases}
\end{equation} 

sampled independently for every distinct pair $\{i,j\}$.  To simplify the writing, we focus on the assortative case, $\mu_1>\mu_2$.  We further define the signal-to-noise ratio of the model SBM$(n, \mu_1,\mu_2, \tau^2)$ as the following.
\begin{equation}  \label{eq:SNR_n}
    \text{SNR}_n: =  \frac{(\mu_1-\mu_2)^2 n}{8\tau^2 \log n},
\end{equation}
\begin{equation}
    \text{SNR}:= \lim_{n\to\infty} \text{SNR}_n.
\end{equation}

To simplify our writings further, we work with the critical scaling of $\mu_1$ and $\mu_2$ as the following.
\begin{equation} \label{def:scaling}
\mu_1= \alpha \sqrt{\frac{\log n}{n}}, \mu_2= \beta\sqrt{\frac{\log n}{n}}.
\end{equation}

In this notation, we have our signal-to-noise ratio given by the following with $\alpha >\beta $ and $ \tau >0$ constants.

\begin{equation} \label{eq:SNR}
\text{SNR}:= \frac{(\alpha -\beta)^2}{8\tau^2}.
\end{equation}

 The problem of exact recovery asks to recover the planted community labeling $(\sigma_i^*)_{i\in [n]}$ exactly (up to a global sign flip) having observed an instance of the weighted graph $G_n\sim$ SBM$(n, \mu_1\,\mu_2, \tau^2)$ or equivalently the random $n \times n $ symmetric (weighted) adjacency matrix $(A_{ij})_{i,j,\in [n]}$ of the graph. 
 Let us  denote the parameter space $\Theta_n$ as 
 \begin{equation} \label{eq:SOL}
 \Theta_n = \{\sigma\in \{\pm1\}^n, \langle\sigma, \mathbf{1}_n\rangle=0\},
 \end{equation}
 
 where 
 $\mathbf{1}_n := (1, \cdots, 1)\in \mathbb{R}^n$. To be more precise, we define the agreement ratio between two community vectors $\hat{\sigma},  \sigma  \in \{\pm 1\}^n$ as 
 \begin{equation} \label{def:agr}
     M(\hat{\sigma}, \sigma)= \max \left(\frac{1}{n}\sum_{i=1}^n \mathbf{1}_{\hat{\sigma}(i)=\sigma(i)}, \frac{1}{n}\sum_{i=1}^n \mathbf{1}_{-\hat{\sigma}(i)=\sigma(i)}\right).
 \end{equation}

Then, we say that the problem of exact recovery is solvable if having observed an instance of the weighted random graph $G_n \sim$ SBM$(n,\mu_1, \mu_2,\tau^2)$ with an unobserved $\sigma^* \in \Theta_n$, we output a community label $\hat{\sigma} \in \Theta_n
$ such that 
\begin{equation} \label{def:ER}
    \mathbb{P} \left[(M(\hat{\sigma},\sigma^*)=1)\right] \to 1.
\end{equation} 

We say the problem of exact recovery is unsolvable otherwise. That is, with probability bounded away from zero, every estimator fails to exactly recover the community labels, as the number of vertices grows.
 
 \par
 
 Without loss of generality, we assume that the unobserved labels $(\sigma_i^*)_{i\in [n]}$ are such that  $\sigma_i^*=1$  for all $1\leqslant i \leqslant \frac{n}{2}$ and $\sigma_i^*=-1$  for all $\frac{n}{2}+1\leqslant i \leqslant  n$. In that case, the weighted adjacency matrix $A$ has the following form (except that the diagonal terms  $A_{ii} =0 $ for all $1\leqslant i\leqslant n$). 

\begin{equation} \label{def:MatA}
A=\begin{bmatrix}
\begin{array}{c|c}
N(\mu_1, \tau^2) & N(\mu_2, \tau^2) \\[1ex]
\hline
\\[-2ex]
N(\mu_2, \tau^2) & N(\mu_1, \tau^2)
\end{array}
\end{bmatrix}.
\end{equation}

\subsection{Maximum likelihood estimation} 

In this model, we define the maximum likelihood estimator as the following.

\begin{equation} \label{def:MLE}  
\hat{\sigma}_{MLE}:= \underset{\sigma \in \Theta_n}{\argmax} \hspace{.1cm} \log \p(G|((\sigma_i)_{i\in [n]})= \sum_{\sigma_i\sigma_j =1}\frac{-(A_{ij}-\mu_1)^2}{2\tau^2} + \sum_{\sigma_i\sigma_j =-1}\frac{-(A_{ij}-\mu_2)^2}{2\tau^2} -\log c_n,
\end{equation}

where $c_n =(\tau \sqrt{2 \pi})^{\binom{n}{2}}$. Further simplifying the right-hand side we have the following.

\begin{lemma} \label{lem: MLE}
    Let $G_n \sim \text{SBM}(n,\mu_1, \mu_2, \tau^2)$ with an unknown $\sigma^* \in \Theta_n$. Then we have the MLE given by the following.
    \begin{equation} \label{eq:MLE}
   \hat{\sigma}_{MLE}= \underset{\sigma \in \Theta_n}{ \argmax} \hspace{.1 cm}f_A(\sigma)  \quad \quad \text{where} f_A(\sigma):=\sum_{i,j} A_{ij}\sigma_i \sigma_j 
\end{equation}
\end{lemma}
\begin{proof}
    We have the following from the definition of the maximum likelihood estimator \ref{def:MLE}.
    \[\hat{\sigma}_{MLE}=  \underset{\sigma \in \Theta_n}{\argmax} \sum_{\sigma_i\sigma_j =1}\frac{-(A_{ij}-\mu_1)^2}{2\tau^2} + \sum_{\sigma_i\sigma_j =-1}\frac{-(A_{ij}-\mu_2)^2}{2\tau^2}.\]
    
 We start by simplifying the right-hand side of the expression above.
 
\begin{align*}
& \hspace{ .38 cm} \sum_{\sigma_i\sigma_j = 1}\frac{-(A_{ij}-\mu_1)^2}{2\tau^2} + \sum_{\sigma_i\sigma_j = -1}\frac{-(A_{ij}-\mu_2)^2}{2\tau^2} \\
&= \sum_{\sigma_i\sigma_j = 1}\frac{-A_{ij}^2 - \mu_1^2 + 2A_{ij}\mu_1}{2\tau^2} + \sum_{\sigma_i\sigma_j = -1}\frac{-A_{ij}^2 - \mu_2^2 +2A_{ij}\mu_2}{2\tau^2} \\
&= \sum_{\{i,j\}} \frac{-A_{ij}^2}{2\tau^2} + \frac{-\mu_1^2}{2\tau^2}\times \frac{n}{2}\left(\frac{n}{2}-1\right)+  \frac{-\mu_2^2}{2\tau^2} \times \frac{n}{2}\frac{n}{2} + \sum_{\sigma_i\sigma_j =1} \frac{A_{ij}\mu_1}{\tau^2} + \sum_{\sigma_i\sigma_j =-1} \frac{A_{ij}\mu_2}{\tau^2},
\end{align*}

where the index $\{ i,j\}$ means to sum over all the $\binom{n}{2}$ edges of the graph $G_n$. The total number of within-community edges and across-community edges are given by the following.
\[E_{\text{in}}= E_+ +E_-= \frac{n}{2}(\frac{n}{2}-1),\]
\[E_{\text{out}}= \frac{n}{2}\frac{n}{2}.\]
We further observe that the right hand of the penultimate expression simplifies to the following.

\begin{align*}
    & \hspace{.25 cm} \sum_{\sigma_i\sigma_j =1} A_{ij}\mu_1 + \sum_{\sigma_i\sigma_j =-1} A_{ij}\mu_2 \\
    & = \sum_{\{i,j\}} A_{ij}\mu_1 1_{\sigma_i\sigma_j=1} +\sum_{\{i,j\}} A_{ij}\mu_2 1_{\sigma_i\sigma_j =-1} \\
    & = (\mu_1-\mu_2)\sum_{\{i,j\}} A_{ij} 1_{\sigma_i\sigma_j=1} +\sum_{\{i,j\}} A_{ij}\mu_2 \\
    & = \frac{(\mu_1-\mu_2)}{2}\sum_{\{i,j\}} A_{ij} (\sigma_i\sigma_j+1) +\sum_{\{i,j\}} A_{ij}\mu_2 \\
    & = \frac{(\mu_1-\mu_2)}{4}\sum_{i,j} A_{ij} \sigma_i\sigma_j + \frac{(\mu_1-\mu_2)}{2}\sum_{\{i,j\}} A_{ij}  +\sum_{\{i,j\}} A_{ij}\mu_2
\end{align*} 

As most of the terms above are independent of $\sigma$, and $\mu_1 >\mu_2$, we have the following expression for MLE.

   \[\hat{\sigma}_{MLE}= \underset{\sigma \in \Theta_n}{ \argmax} \sum_{i,j} A_{ij}\sigma_i \sigma_j
   \]
 
\end{proof}

We observe that by the symmetry of the function $f_A$  both $\hat{\sigma}_{MLE}$ and $-\hat{\sigma}_{MLE}$   maximize the objective function.

    \subsection{Statistical impossibility}
In this section, we present the statistical impossibility result, providing an information theoretic limit for the exact recovery problem in Gaussian weighted graphs by showing that MLE fails to recover the community label in a low SNR regime. More precisely, we prove that when the signal-to-noise ratio (SNR) is less than one, no algorithm (efficient or not) can exactly recover the community labels with probability approaching one.

 \subsubsection{The statement}
\begin{theorem}[\textbf{Statistical impossibility of exact recovery}] \label{thm: Negative}
     Let $G_n \sim $ SBM$(n,\mu_1, \mu_2, \tau^2)$ with $\mu_1, \mu_2$ in the critical scaling regime \ref{def:scaling}, and $\alpha > \beta$, $\tau >0$ be constants with SNR = $\frac{(\alpha- \beta)^2}{8\tau^2}$. If SNR$<1$ then MLE fails in recovering the community label $\sigma^*$ with high probability. 
     \[ \underset{n \to \infty}{\limsup}\hspace{.1cm}\p(M(\hat{\sigma}_{MLE}, \sigma^*)=1) =0 \quad \iff \quad \underset{n \to \infty}{\liminf}\hspace{.1cm}\p(M(\hat{\sigma}_{MLE}, \sigma^*)\neq 1) = 1.\]
\end{theorem}

\begin{remark} \label{rem: MLE}
To explore the information-theoretic limits, we examine the Maximum A Posteriori (MAP) estimation, designed to maximize the probability of accurately reconstructing communities. In the context of uniform community assignment, (which is the case for us as we are fixing an unknown community labeling $\sigma^*$ with balanced community sizes, which could be chosen uniformly from all the possible $\binom{n}{\frac{n}{2}}$ choices of community assignments), MAP estimation is equivalent to Maximum Likelihood estimation (MLE).  We refer to \cite{abbe2023community}[Section 3.1] for more details on this.
\par 
Consequently, if MLE  fails in exactly recovering communities with high probability as the network size (n) grows, there is no algorithm, whether efficient or not, capable of reliably achieving this task. Therefore, to prove that exact recovery is not solvable, it is enough to show that MLE is unable to recover the planted community $\sigma^*$ with high probability. 
\par 
It is noteworthy that MLE, in this scenario, entails identifying a balanced 
 cut (bisection) \ref{eq:MLE} in the graph that minimizes the number of weighted edges crossing the cut—a solution to the NP-hard min-bisection problem. While ML serves to delineate the fundamental limit, its inefficiency in providing a practical algorithm prompts further consideration in a subsequent phase.
\end{remark}
 \par
 
Without loss of generality, we assume that the planted community is given by $\sigma^* =(\mathbf{1}_{\frac{n}{2}}, -\mathbf{1}_{\frac{n}{2}}) \in \Theta_n$, and $G_n \sim $ SBM$(n, \mu_1, \mu_2, \tau^2)$ conditioned on $\sigma^*$. We denote the resulting communities  as
    \[C_+=\left\{1, \cdots, \frac{n}{2}\right\}, C_{-}=\left\{\frac{n}{2}+1, \cdots, n\right\}.\]

\begin{definition} 

We also define the set of bad pairs of vertices in $G_n$ by the following.
    \[\mathcal{B}(G_n) := \{(i,j): i\in C_+, j\in C_{-}, f_A(\sigma^*)< f_A(\sigma^*[i\leftrightarrow j])\},\]
    
    where  $\sigma^*[i\leftrightarrow j]$ denotes the vector obtained by swapping the values of coordinates $i$ and $j$ in $\sigma^*$.
\end{definition}

From \ref{lem: MLE} we observe that the event $f_A(\sigma^*)< f_A(\sigma^*[i\leftrightarrow j]) $ is the same as the event that the likelihood of the partition $\sigma^*[i\leftrightarrow j]$ is larger than the planted one $\sigma^*$. Hence, we have the following result.

\begin{lemma} \label{lem:MLEa}
   MLE fails to solve the exact recovery with high probability  if $\mathcal{B}(G_n)$ is non-empty with high probability. More precisely, we have the following lower bound on the error.
   \[\underset{n \to \infty}{\liminf}\hspace{.1cm}\p(M(\hat{\sigma}_{MLE}, \sigma^*)\neq1) \geqslant  \underset{n \to \infty}{\liminf}\hspace{.1cm} \p (\mathcal{B}(G_n)\neq \emptyset)\]
\end{lemma}

\begin{proof}
    If $\exists (i,j) \in \mathcal{B}(G_n)$, then we can swap the coordinates $i$ and $j$ in $\sigma^*$ and increase the likelihood of the partition (or equivalently $f_A(\cdot)$), thus obtaining a different partition than the planted one ($\sigma^*$) that is more likely than the planted one. More precisely, we have that $\mathcal{B}(G_n) \neq \emptyset \implies M(\hat{\sigma}_{MLE}, \sigma^*)\neq 1$. Therefore,
    \begin{align*}
        \p(M(\hat{\sigma}_{MLE}, \sigma^*)\neq1) \geqslant   \p (\mathcal{B}(G_n)\neq \emptyset)
    \end{align*} 
\end{proof}
With the lemma above, the impossibility of exact recovery reduces to proving that if SNR$<1$ then 
\[\underset{n \to \infty}{\liminf}\hspace{.1cm}\p (\mathcal{B}(G_n)\neq \emptyset) >0.\]

We will show this by proving a stronger result \ref{lem:Bto1} that 
\[
\underset{n \to \infty}{\liminf}\hspace{.1cm}\p (\mathcal{B}(G_n)\neq \emptyset) =1.
\]
 \subsubsection{First step -- the first moment method}
 
We now analyze the condition $f_A(\sigma^*)< f_A(\sigma^*[i\leftrightarrow j])$ for $i\in C_+$ and $j\in C_-$.

\par
\begin{definition}
    For a vertex $i \in C_+$, and $j \in C_-$ we define the following.
    \[d_+(i)= \sum_{v\in C_+}A_{iv},\]
    \[d_+(j)=  \sum_{v\in C_-}A_{jv},\]
    \[d_-(i)= \sum_{v\in C_-}A_{iv},\]
    \[d_-(j)=  \sum_{v\in C_+}A_{jv},\]
    \[d_-(i\setminus j)=\sum_{v\in C_- \setminus j}A_{iv},\]
    \[d_-(j\setminus i)=\sum_{v\in C_+ \setminus i}A_{jv}\]
\end{definition}
It is worth observing that $d_+(i)$ and $ d_+(j)$ count the total number of weighted edges that $i$ and $j$ have with their own communities $C_+$ and  $C_-$, and $d_-(i)$ and $ d_-(j)$ count the total number of weighted edges that $i$ and $j$ have with their opposite communities $C_-$ and  $C_+$ respectively. Furthermore, $d_-(i\setminus j)$ and $d_-(j\setminus i)$  count the total number of weighted edges that $i$ has with community $C_-$ ignoring vertex $j$ and the total number of weighted edges that $j$ has with community $C_+$ ignoring vertex $i$.
\begin{lemma} \label{lem:MLEb}
With the notations above and $i\in C_+, j\in C_-$ we have the following.
 
    \[f_A(\sigma^*)< f_A(\sigma^*[i\leftrightarrow j]) \quad \iff \quad   d_+(i)+d_+(j) < d_-(i\setminus j)+d_-(j\setminus i).\]
\end{lemma}  
\begin{proof}
We recall the definition of MLE estimator \ref{def:MLE} given by (recall $A_{uu}:=0$ for all $1\leqslant u\leqslant n$)
\[f_A(\sigma^*) :=  \sum_{1\leqslant u,v \leqslant n} A_{uv} \sigma^*_u\sigma^*_v =2\sum_{1\leqslant u<v \leqslant n} A_{uv} \sigma^*_u\sigma^*_v.\]

Therefore, one has the following chain of equivalences for $i\in C_+, j\in C_-$.
    \[
    f_A(\sigma^*)< f_A(\sigma^*[i\leftrightarrow j]
    \]
    \[
    \leftrightarrow  \hspace{.1 cm} \sum_{1\leqslant u<v \leqslant n} A_{uv} \sigma^*_u\sigma^*_v < \sum_{1\leqslant u<v \leqslant n} A_{uv} \sigma^*_u[i\leftrightarrow j]\sigma^*_v[i\leftrightarrow j])
    \]
    \[ \leftrightarrow \hspace{.1 cm} \sum_{v\in C_+}A_{iv} - \sum_{v\in C_- \setminus j}A_{iv} + \sum_{v\in C_-}A_{jv} - \sum_{v\in C_+ \setminus i}A_{jv} - A_{ij} < -\sum_{v\in C_+}A_{iv}  +  \sum_{v\in C_- \setminus j}A_{iv} - \sum_{v\in C_-}A_{jv} + \sum_{v\in C_+ \setminus i}A_{jv} - A_{ij} 
    \]
    \[
    \leftrightarrow d_+(i) - d_-(i\setminus j) + d_+(j)- d_-(j\setminus i) < -d_+(i) + d_-(i\setminus j) - d_+(j)  + d_-(j\setminus i)
    \]
    \[
    \leftrightarrow  d_+(i)+d_+(j) < d_-(i\setminus j)+d_-(j\setminus i)
    \]
\end{proof}

With the lemma above we have our set of bad pairs given by the following.
\[
\mathcal{B}(G_n) =\{(i,j): i \in C_+, j \in C_-, d_+(i)+d_+(j)< d_-(i\setminus j)+d_-(j\setminus i)\}.
\]

Finally, to apply the first and second-moment methods to our setting, we now define the following sets with diverging constants $c_n \to \infty $ to be chosen later according to the proof.

\begin{definition}  The set of bad vertices (instead of bad pairs) in each community in $G_n$ is given by
\[
\mathcal{B}_+(G_n):=\left\{i: i \in C_+, d_{+}(i) < d_{-}(i)-c_n\right\}, 
\]
\[
\mathcal{B}_-(G_n):=\left\{j: j \in C_-, d_{+}(j) < d_{-}(j)-c_n\right\},
\]
\end{definition}

 \begin{lemma} \label{lem:Bto1}
     
 If $\mathcal{B}_+(G_n)$ and $\mathcal{B}_-(G_n)$ are non-empty with probability approaching $1$, then $\mathcal{B}(G
 _n)$ is non-empty with probability approaching  $1$ as $n \to  \infty$.
\end{lemma}
\begin{proof}
We define the following events for $i \in C_+$ and $j \in C_-$, we define the following events. 
\[
E_i:= \{d_{+}(i) < d_{-}(i)-c_n\}
\]
\[
E_j:=\{ d_{+}(j) < d_{-}(j)-c_n\}
\]
\[
E_{ij}:= \{d_{+}(i)+d_{+}(j) < d_{-}(i)+d_{-}(j)-2c_n\}
\]
\[
E:=\cup_{(i,j): i\in C_+, j\in C_-} E_{ij}
\]
\[
F_{ij}:= \{d_+(i)+d_+(j)< d_-(i\setminus j)+d_-(j\setminus i)\}
\]
\[
F:= \cup_{(i,j): i\in C_+, j\in C_-} F_{ij}
\]
\[
G_{ij}:=\{A_{ij}< c_n\}
\]
\[G :=\cap_{(i,j): i\in C_+, j \in C_-} G_{ij}\]

Then we have the following equivalence of events.
\[
E_+:=
\{\mathcal{B}_+(G_n) \neq \emptyset\}= \cup_{i\in C_+}E_i
\]
\[E_-:=
\{\mathcal{B}_-(G_n) \neq \emptyset\} = \cup_{j\in C_-}E_j
\]
\[
F:= \cup_{(i,j): i\in C_+, j\in C_-} F_{ij}= \{\mathcal{B}(G_n) \neq \emptyset\}
\]

On the event $E_i\cap E_j$ we have 
\[ 
d_{+}(i)+d_{+}(j) < d_{-}(i)+d_{-}(j)-2c_n \implies E_i\cap E_j \subset E_{ij}.
\]

 We also have that $E_+\cap E_- \subset E $ because $\cup_{i\in C_+}E_i \bigcap 
 \cup_{j\in C_-}E_j \subset \cup_{(i,j): i \in C_+, j \in C_-}E_{ij}$.
 \par
 Further, $d_{-}(i)+d_{-}(j)-2c_n = d_{-}(i\setminus j)+d_{-}(j\setminus i) + 2A_{ij}-2c_n$. So, whenever $A_{ij}< c_n$ we have  \[d_{-}(i)+d_{-}(j)-2c_n < d_{-}(i\setminus j)+d_{-}(j\setminus i) \implies E_{ij}\cap G_{ij} \subset F_{ij}\]

Finally, we have $G \cap E \subset F$  because  $\cup_{(i,j): i\in C_+, j \in C_-} E_{ij} \bigcap \cap_{(i,j): i\in C_+, j \in C_-}G_{ij} \subset  \cup_{(i,j): i\in C_+, j \in C_-} F_{ij}$.
\par 
 In the probabilistic sense, we have the following.
 \[
 \mathbb{P}\left[\mathcal{B}_+(G_n) \neq \emptyset\right] = \mathbb{P}\left[ E_+ \right] \to 1 \quad \text{and} \quad  \mathbb{P}\left[\mathcal{B}_-(G_n) \neq \emptyset\right] =\mathbb{P}\left[  E_-\right] \to 1
 \] 
 \[
 \implies \mathbb{P}\left[\mathcal{B}_+(G_n) \neq \emptyset \cap \mathcal{B}_-(G_n) \neq \emptyset \right] =\mathbb{P}\left[E_+\cap E_-\right] \to 1
 \]
 \[
 \implies \mathbb{P}\left[ \exists (i,j): i 
 \in  C_+, j \in C_-,  d_{+}(i)+d_{+}(j) < d_{-}(i)+d_{-}(j)-2c_n\right]= \mathbb{P}\left[E\right] \to 1.
 \]

 Combining the above, to show that $\mathbb{P}\left[F\right] \to  1$, it remains to show that $\mathbb{P}\left[G\right] \to  1$, which we prove next.
\end{proof}

\begin{lemma} \label{lem:MLEc}
With the notations above, we have $\mathbb{P}\left[G\right] \to 1$.
\end{lemma}
\begin{proof}
    We observe that it is enough to show by union bound and symmetry of the model that for a fixed $i \in C_-, j \in C_+$
    \[\mathbb{P}\left[G^c\right] = \mathbb{P}\left[ \cup_{(i,j): i\in C_+, j \in C_-} G^c_{ij}\right] \leqslant \frac{n^2}{4} \mathbb{P}\left[A_{ij}\geqslant c_n \right] \to 0
    \]
    
    We recall that $A_{ij} \sim N(\mu_2, \tau^2)$ and therefore $\mathbb{P}\left[A_{ij}\geqslant c_n \right] =\mathbb{P}\left[N(0,1)\geqslant \frac{c_n-\mu_2}{\tau} \right] $. We now take $c_n\geqslant \tau\sqrt{2C\log n}$ for some large absolute constant $C>0$ ($C=3$ works) and by standard Gaussian tail bounds \ref{eq:Gtail} we have the following.
    \[n^2 \mathbb{P}\left[A_{ij}\geqslant c_n\right] \leqslant n^2 e^{-(1+o(1)) C \log n} \leqslant n^{(2-C)(1+o(1))} \to 0.\]
\end{proof}
\subsubsection{Second step -- the second moment method}
\begin{lemma} \label{lem:MLEd}
 If $\text{SNR}=\frac{(\alpha -\beta)^2}{8\tau^2}<1$, then $\mathcal{B}_+(G_n)$ and  $\mathcal{B}_-(G_n)$ are non empty with probability approaching $1$.
\[
\mathbb{P}\left[E_+\right] = \mathbb{P}\left[ \mathcal{B}_+(G_n)\neq \emptyset \right]  =1 - o(1)
\]
\[\mathbb{P}\left[E_-\right]=\mathbb{P}\left[ \mathcal{B}_-(G_n)\neq \emptyset \right]=1 - o(1)\]
\end{lemma}

\begin{proof}  We begin the proof by recalling the second-moment method, which says that we have the following inequality for an integer-valued random variable $Z \geqslant 0$.
\begin{equation}\label{eq:SMM}
\mathbb{P}\left[Z \geqslant 1 \right] \geqslant \frac{\mathbb{E}\left[Z\right]^2}{ \mathbb{E}\left[Z^2\right]}
\end{equation}

The proof of this follows just from Cauchy-Schwarz inequality applied to $Z\mathbf{1}_{Z\geqslant 1} $ as shown below. 
\begin{equation}\label{lem:SMMproof}
\mathbb{E} \left[Z\right]^2= \mathbb{E} [Z \mathbf{1}_{Z>0}]^2 =  \mathbb{E} [Z \mathbf{1}_{Z\geqslant 1}]^2 \leqslant \mathbb{E} [Z^2] \mathbb{P}\left[Z\geqslant 1\right]
\end{equation}

Now, we construct two integer-valued random variables non-negative random variables $Z_+$ and $Z_-$.
\[
Z_+ = \sum_{i \in C_+} \mathbf{1}_{d_{+}(i) < d_{-}(i)-c_n}= \sum_{i \in C_+} \mathbf{1}_{E_i}.
\]
\[
Z_- = \sum_{j \in C_-} \mathbf{1}_{d_{+}(j) < d_{-}(j)-c_n}= \sum_{j \in C_-} \mathbf{1}_{E_j}.
\]

Further, we have the following identities for the relevant events.
\[ 
\{ \exists i \in \mathcal{B}_+(G_n)\}=\{\mathcal{B}_+(G_n) \neq \emptyset\} =\{Z_+ \geqslant 1\}
\]
\[\{ \exists j \in \mathcal{B}_-(G_n)\} =\{\mathcal{B}_-(G_n) \neq \emptyset\} =\{Z_- \geqslant 1\}\]

We now provide the proof for the case of $Z_+$ and observe that by symmetry the case of $Z_-$ is exactly similar. 
 Let
\[
D_i:=\mathbf{1}_{\left\{d_{+}(i) <  d_{-}(i)-c_n\right\}} .
\]

By the bound obtained in the second-moment method, we have
\[
\begin{aligned}
& \mathbb{P}\left\{Z_+ \geqslant 1\right\} \geqslant   \frac{\mathbb{E}\left[Z_+\right]^2}{ \mathbb{E}\left[Z_+^2\right]} \geqslant\frac{\left(\mathbb{E} \left[\sum_{i \in C_+} D_i\right]\right)^2}{\mathbb{E}\left[\left(\sum_{i \in C_+} D_i\right)^2\right]}=  \frac{\left(\sum_{i \in C_+}\mathbb{P}\left[E_i\right] \right)^2}{ \sum_{i \in C_+} \mathbb{P}\left[E_i\right] +\sum_{i_1 \neq  i_2  \in C_+} \mathbb{P}\left[ E_{i_1}\cap E_{i_2}\right]}\\
& =\frac{\left((n / 2) \mathbb{P}\left[E_1\right]\right)^2}{(n / 2) \mathbb{P}\left[E_1\right]+(n / 2)(n / 2-1) \mathbb{P}\left[E_1\cap E_2\right]} = \frac{1}{ \frac{1}{(n / 2) \mathbb{P}\left[E_1\right]} + (1 - \frac{2}{n})\frac{\mathbb{P}\left[E_1\cap E_2\right]}{ \mathbb{P}\left[E_1\right]\mathbb{P}\left[E_2\right] }}
\end{aligned}
\]

and the last bound tends to $1$ if the following two conditions hold, which we will prove in the following lemma
\begin{align}
n \mathbb{P}\left[E_1\right] & \to \infty, \\
\frac{\mathbb{P}\left[E_1\cap E_2\right]}{\mathbb{P}\left[E_1\right] \mathbb{P}\left[E_2\right]} &  \to 1 .
\end{align}

\end{proof}

It remains to show that both the conditions hold as soon as $\text{SNR} < 1$ and reveals the bound in the theorem. We also remark that the second condition amounts to say that the events  $E_1$ and $E_2$ are asymptotically independent.

\begin{lemma} \label{lem:MLEe}
Let $D_i:=\mathbf{1}_{\left\{d_{+}(i) <  d_{-}(i)-c_n\right\}} =\mathbf{1}_{E_i}$.  If $\text{SNR} < 1$ then, we have the following   
\[
n \mathbb{P}\left[ E_1 \right]  \to \infty.
\]
\end{lemma}

\begin{proof}
 We recall that if $X_1 \sim N(\alpha_1, \beta_1^2)$, $X_2 \sim N(\alpha_2, \beta_2^2)$ and $X_1 $ is independent of $X_2$, then $X_1+X_2 \sim N(\alpha_1+\alpha_2, \beta_1^2+\beta_2^2).$
 We now observe the following distributional identities. 
 \[
 d_+(i)= \sum_{v\in C_+}A_{iv} \sim N((\frac{n}{2}-1) \mu_1, (\frac{n}{2} -1)\tau^2) ,
 \]
 \[ 
 d_-(i)= \sum_{v\in C_-}A_{iv} \sim N(\frac{n \mu_2}{2}, \frac{n\tau^2}{2})
 \]
 
 With the identities above and the fact that $d_+(i)$ is independent of $d_-(i)$, we have the following.
 \begin{align*}
     \mathbb{P}\left[E_1\right]  = \mathbb{P} \left\{d_{+}(i) <  d_{-}(i)-c_n\right\} 
 & =\mathbb{P}\left[ N((\frac{n}{2}-1) \mu_1, (\frac{n}{2} -1)\tau^2) < N(\frac{n \mu_2}{2}, \frac{n\tau^2}{2}) -c_n \right] \\
     & = \mathbb{P}\left[N(0, (n-1)\tau^2) > c_n -\mu_1 +\frac{n}{2}(\mu_1 -\mu_2)\right] \\
     & = \mathbb{P}\left[N(0,1) > \frac{c_n -\mu_1}{\tau\sqrt{n-1}} +\frac{(\alpha - \beta)\sqrt{\log n}}{2\tau \sqrt{1- \frac{1}{n}}}\right] \\
     & = \mathbb{P}\left[N(0,1) > 
     (1+o(1)) \frac{(\alpha - \beta)\sqrt{\log n}}{2\tau}\right]
 \end{align*}

 We take $c_n \to \infty$ such that the following happens.
 \begin{equation} \label{eq:c_n}
 o(1) =\frac{2(c_n -\mu_1)}{ (\alpha -\beta) \sqrt{n \log n}} \to  0 \iff \frac{c_n}{ \sqrt{ n \log n}} \to 0.
 \end{equation}
 
 We now recall the standard tail bound for standard Gaussian random variables for any $t > 0$ with 
 $\phi(t)=\frac{e^{-\frac{t^2}{2}}}{\sqrt{2\pi}}$
 \begin{equation}\label{eq:Gtail}
  \frac{e^{-\frac{t^2}{2}}}{(t +\frac{1}{t})  \sqrt{2\pi}} \leqslant \Phi(t)=\mathbb{P}\left[ N(0,1) \geqslant 
 t \right] \leqslant \frac{e^{-\frac{t^2}{2}}}{t \sqrt{2\pi}}
 \end{equation}
 \begin{equation}
     \Phi(t) \sim \frac{\phi(t)}{t} \quad \text{as} \quad t \to \infty
 \end{equation}
 
With the inequalities above, we have for $t_n= (1+o(1)) \frac{(\alpha - \beta)\sqrt{\log n}}{2\tau} \to \infty $ 
\begin{align*}
     n \mathbb{P}\left[E_1 \right] 
     & = 
     n \mathbb{P}\left[N(0,1) \geqslant 
     (1+o(1)) \frac{(\alpha - \beta)\sqrt{\log n}}{2\tau}\right] \\
     &  \geqslant \frac{n e^{-(1+o(1)) \text{SNR} \log n}}{(t_n+\frac{1}{t_n}) 
     \sqrt{2 \pi}} \\
     & = n^{(1 - SNR) (1+ o(1))} \to  \infty 
\end{align*}
\end{proof}

We remark that the result holds verbatim even in the general model \ref{eq:SNR_n} as soon as $\lim_{n \to \infty} \text{SNR}_n =\text{SNR} <1$.

\begin{lemma} \label{lem:MLEf}
Let $D_i:=\mathbf{1}_{\left\{d_{+}(i) <  d_{-}(i)-c_n\right\}} =\mathbf{1}_{E_i}$. Then, we have the following.
\[
\frac{\mathbb{P}\left[E_1 \cap E_2\right]}{\mathbb{P}\left[E_1\right] \mathbb{P}\left[E_2\right]} =\frac{\mathbb{P}\left[E_1|E_2\right]}{\mathbb{P}\left[E_1\right]}   \to 1 .
\]
\end{lemma}
\begin{proof}

We start by observing the following distributional identities.
\[
d_+(1\setminus2):= \sum_{i\in C_+ \setminus 2}A_{i1} \overset{d}{=}N\left(\left(\frac{n}{2}-2\right)\mu_1, \left(\frac{n}{2}-2\right)\tau^2\right),
\]
\[
d_+(2\setminus1):= \sum_{i\in C_+ \setminus 1}A_{i2} \overset{d}{=}N\left(\left(\frac{n}{2}-2\right)\mu_1, \left(\frac{n}{2}-2\right)\tau^2\right),
\]
\[
d_-(1)=\sum_{j\in C_-}A_{1j}\overset{d}{=} N\left(\frac{n}{2}\mu_2, \frac{n}{2}\tau^2\right)
\]
\[
d_-(2)=\sum_{j\in C_-}A_{2j} \overset{d}{=} N\left(\frac{n}{2}\mu_2, \frac{n}{2}\tau^2\right)\]
\[A_{12} \overset{d}{=} \mu_1 +\tau N(0,1)
\]

We also have that $d_+(1\setminus2), d_+(2\setminus 1), A_{12}, d_-(1), d_-(2)$ are jointly independent collections of  random variables. Let $Z_1, Z_2, Z$ be three IID standard Gaussian random variables and denote 
$s=\frac{(\alpha -\beta)}{2\tau}$.
    \begin{align*}
    E_1 = \{d_+(1) < d_-(1) -c_n\} & =\{A_{12} +N\left(\left(\frac{n}{2}-2\right)\mu_1, \left(\frac{n}{2}-2\right)\tau^2\right) < N\left(\frac{n}{2}\mu_2, \frac{n}{2}\tau^2\right) -c_n\} \\
    & =\{A_{12} < \frac{n}{2}(\mu_2 -\mu_1) +2\mu_1-c_n +\tau\sqrt{n-2}Z_1\} \\
    & =\{ \frac{Z}{\sqrt{n-2}} +\frac{(c_n-\mu_1)}{\tau \sqrt{n-2}} +\frac{(\alpha -\beta)\sqrt{\log n}}{2\tau \sqrt{1-\frac{2}{n}}} < Z_1 \}
    \end{align*}
    
    \begin{align*}
    E_2 = \{d_+(2) < d_-(2) -c_n\}  & =\{A_{12} +N\left(\left(\frac{n}{2}-2\right)\mu_1, \left(\frac{n}{2}-2\right)\tau^2\right) < N\left(\frac{n}{2}\mu_2, \frac{n}{2}\tau^2\right) -c_n\} \\
    & = \{A_{12} < \frac{n}{2}(\mu_2 -\mu_1) +2\mu_1-c_n +\tau\sqrt{n-2}N(0,1)\} \\
    & =\{ \frac{Z}{\sqrt{n-2}} +\frac{(c_n-\mu_1)}{\tau \sqrt{n-2}} +\frac{(\alpha -\beta)\sqrt{\log n}}{2\tau \sqrt{1-\frac{2}{n}}} < Z_2 \}
    \end{align*}
    
Combining all the above we can write the following with $1+o(1)=\frac{\sqrt{n}}{\sqrt{n-2}}+\frac{c_n -\mu_1}{\tau s \sqrt{(n-2) \log n}}$.

\begin{equation}\label{eq:SM}
\frac{\mathbb{P}\left[E_1\cap E_2\right]}{ \mathbb{P}\left[E_1\right] \mathbb{P}\left[E_2\right]} =\frac{\mathbb{P}\left[Z_1 \geqslant s\sqrt{\log n}(1+o(1)) +\frac{Z}{\sqrt{n-2}}\cap Z_2 \geqslant s\sqrt{\log n}(1+o(1)) +\frac{Z}{\sqrt{n-2}}\right]}{\mathbb{P}\left[Z_1 \geqslant s\sqrt{\log n}(1+o(1)) +\frac{Z}{\sqrt{n-2}}\right] \mathbb{P}\left[Z_1 \geqslant s\sqrt{\log n}(1+o(1)) +\frac{Z}{\sqrt{n-2}}\right]}
\end{equation}
\begin{align*}
\mathbb{P}\left[Z_1 \geqslant s\sqrt{\log n}(1+o(1)) +\frac{Z}{\sqrt{n-2}}\right] & =\mathbb{P}\left[N(0, 1+\frac{1}{n-2}) \geqslant  s\sqrt{\log n}(1+o(1))\right] 
\\
& \sim \frac{1}{s\sqrt{\log n}}\phi\left( (1+o(1))s \sqrt{\log n} \frac{\sqrt{n-2}}{\sqrt{n-1}}\right)
\\
& \sim \frac{1}{s\sqrt{\log n}} \phi(s \sqrt{\log n})
\end{align*}

It remains to be established that the numerator in 
\ref{eq:SM} $ \sim \left(\frac{\phi(s \sqrt{\log n})}{s\sqrt{\log n}} \right)^2$ which we prove next.
\end{proof}

\begin{lemma}\label{lem:jointGauss}
    With the notations above, we have for IID standard Gaussians $Z,Z_1, Z_2$ as $n \to  \infty$
    \[\mathbb{P}\left[Z_1 \geqslant s\sqrt{\log n}(1+o(1)) +\frac{Z}{\sqrt{n-2}}\cap Z_2 \geqslant s\sqrt{\log n}(1+o(1)) +\frac{Z}{\sqrt{n-2}}\right]\sim \left(\frac{\phi(s \sqrt{\log n})}{s\sqrt{\log n}} \right)^2 =:r_n\]
\end{lemma}

\begin{proof}
We start the proof by defining the following events.
\[
A_1 =\{Z_1 \geqslant s\sqrt{\log n}(1+o(1)) +\frac{Z}{\sqrt{n-2}}\}
\]
\[
A_2= \{Z_2 \geqslant s\sqrt{\log n}(1+o(1)) +\frac{Z}{\sqrt{n-2}}\}
\]

    We subdivide $A_1\cap A_2$  into two subevents $\{|Z|\leqslant n^{1/4}\}$ and $\{|Z|> n^{1/4}\}$ and let 
    \[
    I_n=\int_{|z| \leqslant n^{1/4}}\Phi\left(s\sqrt{\log n}(1+o(1)) +\frac{z}{\sqrt{n-2}}\right)^2 \phi(z) dz.
    \]
    
    From the bounds on the Gaussian  tails \ref{eq:Gtail}, we also have 
    \begin{equation} \label{eq:GTF}
    \mathbb{P}\left[|Z| \geqslant n^{1/4}\right] \sim 2 \frac{1}{n^{1/4}} \phi(n^{1/4}) =o(r_n)
    \end{equation}
    
    \begin{equation*}
         I_n \leqslant  \mathbb{P}\left[A_1\cap A_2\right] \leqslant I_n + o(r_n) 
    \end{equation*}

    It remains to establish that $I_n \sim r_n$. As a first step we observe that uniformly over $|z|\leqslant n^{1/4}$ we have 
    \[
    \Phi\left(s\sqrt{\log n}(1+o(1)) +\frac{z}{\sqrt{n-2}}\right)^2 \sim \Phi\left(s\sqrt{\log n}(1+o'(1))\right)^2 \sim r_n
    \]
    
    Therefore, we have $I_n \sim r_n (1- \mathbb{P}\left[|Z| \geqslant n^{1/4}\right]) \sim r_n(1- o(r_n))$. Now we provide a detailed proof of this.
    \par 
    More precisely, using the fact that $\Phi(\cdot)$ is monotone decreasing non-negative function and therefore for $|z| \leqslant n^{1/4}$
   
    \[
     a_n\leqslant \Phi\left(s\sqrt{\log n}(1+o(1)) +\frac{z}{\sqrt{n-2}}\right)^2 \leqslant b_n 
    \]
    \[
    a_n=\Phi\left(s\sqrt{\log n}(1+o(1)) +\frac{n^{1/4}}{\sqrt{n-2}}\right)^2 \sim r_n
    \]
    \[
    b_n=\Phi\left(s\sqrt{\log n}(1+o(1)) -\frac{n^{1/4}}{\sqrt{n-2}}\right)^2 \sim r_n
    \]

    Combining all the above we have the following.
    \[
    r_n \int_{|z|\leqslant n^{1/4}} \phi(z)dz \sim a_n\int_{|z|\leqslant n^{1/4}} \phi(z)dz \leqslant I_n \leqslant b_n\int_{|z|\leqslant n^{1/4}} \phi(z)dz \sim r_n \int_{|z|\leqslant n^{1/4}} \phi(z)dz
    \]
    
    It remains to observe that from \ref{eq:GTF} we have $\int_{|z|\leqslant n^{1/4}} \phi(z)dz =1 -o(r_n) \sim 1$.
\end{proof}
    \subsection{Statistical possibility}
In this section, we prove that as soon as the $\text{SNR} >1$, we have that MLE can recover the unobserved community $\sigma^*$ exactly with probability approaching one. We prove this result in two steps. On the impossibility side,  we already showed that below the threshold, there exists a bad pair of vertices in each community that can be swapped and thus placed in the wrong community while increasing the likelihood (equivalently reducing the cut). Here we show that the min-bisection \ref{eq:MLE} gives the planted bisection, by showing that there is no possibility to swap any set of $k$
 vertices from each community and reduce the cut for any $k \in \{1, \cdots, \frac{n}{4}\}$ above the threshold $\text{SNR} >1$.
\subsubsection{The statement}
\begin{theorem}[\textbf{Statistical possibility of exact recovery}] \label{thm: Positive}
    For the Gaussian weighted stochastic block model  SBM$(n, \mu_1, \mu_2, \tau^2)$ \ref{def:GWSBM}, MLE can exactly recover the community labelling $\sigma^*$ as soon as $\text{SNR} >1$. More precisely, $ \mathbb{P}\left[M(\hat{\sigma}_{\text{MLE}}, \sigma^*)= 1\right] \to 1$.
\end{theorem}
\begin{Notations} Let the ground truth $\sigma^*=\{\mathbf{1}_{\frac{n}{2}}, -\mathbf{1}_{\frac{n}{2}}\} \in \Theta_n$ [\ref{eq:SOL}] and we have the true communities $C_+:=C_+(\sigma^*)=\{1, \cdots, \frac{n}{2}\}$, $C_-:=C_-(\sigma^*)=\{\frac{n}{2}+1, \cdots, n\}$. Further, we let  $\hat{\sigma}  \in  \Theta_n$ with $\hat{\sigma}(1)= +1$. We  denote the predicted community labels  with $\hat{C}_+=\{i \in [n]: \hat{\sigma}(i)=+1\}, \hat{C}_-=\{i \in [n]: \hat{\sigma}(i)=-1\}$. 
We observe that  MLE \ref{eq:MLE} fails to recover the true community labels $\sigma^*$  if and only if there exists $ \hat{\sigma}$ such that $f_A(\hat{\sigma}) > f_A(\sigma^*)$.
For a subset of vertices $S_+ \subseteq C_+$ and $S_- \subseteq C_-$ we define the following
\[
d_+(S_+) =\sum_{i \in S_+, k \in C_+\setminus S_+}A_{ik}
\]
\[
d_+(S_-) =\sum_{j \in S_-, l \in C_-\setminus S_-}A_{jl}
\]
\[
d_-(S_+\setminus S_-) =\sum_{i \in S_+, j \in C_-\setminus S_-}A_{ij}
\]
\[
d_-(S_-\setminus S_+) =\sum_{j \in S_-, i \in C_+\setminus S_+}A_{ij}
\]

\end{Notations}

\subsubsection{The proof -- union bound}
 \begin{lemma} \label{lem:MLEg}
 $\exists$ $\hat{\sigma} \in \Theta_n$ with 
     $f_A(\hat{\sigma}) > f_A(\sigma^*)$  $\smallimplies$ $\exists  S_+ \subseteq C_+, S_- \subseteq C_-$ such that $ 1\leqslant |S_+|=|S_-| \leqslant \frac{n}{4}$ and we have 
     \begin{equation} \label{eq:Pos}
     d_+(S_+) +d_+(S_-)<  d_-(S_+\setminus S_-) +d_-(S_-\setminus S_+).
     \end{equation} 
 \end{lemma}
 \begin{proof}  
 For $\sigma \in \Theta_n$ let $I(\sigma) = \sum_{\sigma_i\sigma_j=1} A_{ij}$ and $O(\sigma)= \sum_{\sigma_i\sigma_j=-1} A_{ij}$. We observe that $I(\sigma)+O(\sigma)= \sum_{ij}A_{ij}$ is independent of $\sigma$, and therefore
     \[
     f_A(\sigma)= \sum_{i,j} A_{ij}\sigma_i\sigma_j=  \sum_{\sigma_i\sigma_j=1} A_{ij} - \sum_{\sigma_i\sigma_j=-1} A_{ij} = I(\sigma)- O(\sigma)
     \]
     \[
         f_A(\sigma) > f_A(\sigma^*) \iff I(\sigma)-O(\sigma) >  I(\sigma^*)-O(\sigma^*) \iff I(\sigma) > I(\sigma^*)
     \]
     
     Now given $\hat{\sigma}$ and $\sigma^*$ we define $S_+:=C_+\setminus\hat{C}_+$, $S_-:= C_-\cap \hat{C}_+$, and observe that $f_A(\hat{\sigma}) > f_A(\sigma^*)$ implies that $|S_+| \geqslant 1$ and $|S_-| \geqslant 1$. As $f_A(\hat{\sigma})= f_A(-\hat{\sigma})$, by choosing $\hat{\sigma}$ correctly (along with the fact $|C_+|=|C_-| = \frac{n}{2}= |\hat{C}_+|= |\hat{C}_-|$) we can ensure that $|S_+|= |S_-| \leqslant \frac{n}{4}$.
     \[
     I(\hat{\sigma})= 
     \sum_{i, i' \in C_+\setminus S_+} A_{ii'} +
     \sum_{j, j' \in S_-} A_{jj'}+ 
     \sum_{ i\in C_+\setminus S_+, j \in S_-} A_{ij} +
     \sum_{i, i' \in S_+} A_{ii'} +
     \sum_{j, j' \in C_- \setminus S_-} A_{jj'} + 
     \sum_{ i\in S_+, j \in C_-\setminus S_-} A_{ij} 
     \]
     \[
     I(\sigma^*) = 
     \sum_{i, i' \in S_+} A_{ii'} +
     \sum_{i, i' \in C_+\setminus S_+} A_{ii'} +
     \sum_{i \in S_+, k \in C_+\setminus S_+} A_{ik} + \sum_{j, j' \in S_-} A_{jj'} +
     \sum_{j, j' \in C_- \setminus S_-} A_{jj'} +
     \sum_{j \in S_-, l \in C_-\setminus S_-} A_{jl}
     \]

     Combining the above, we have the following equivalence.
     \begin{align*}
     f_A(\hat{\sigma}) > f_A(\sigma^*) 
     &  
     \iff  \sum_{ i\in C_+\setminus S_+, j \in S_-} A_{ij} +  \sum_{ i\in S_+, j \in C_-\setminus S_-} A_{ij}  > \sum_{i \in S_+, k \in C_+\setminus S_+} A_{ik} + \sum_{j \in S_-, l \in C_-\setminus S_-} A_{jl} 
     \\
     &
     \iff d_-(S_-\setminus S_+) + d_-(S_+ \setminus S_-) > d_+(S_+)+d_-(S_-)
     \end{align*}
 \end{proof}

 Now, to prove that MLE \ref{def:MLE} can solve exact recovery \ref{thm: Positive} as soon as \text{SNR} $>1$  by proving that such events of the kind \ref{eq:Pos} occur with probability approaching zero. We define the following events as the following.
 \[
 E:=\{\exists S_+ \subseteq C_+, S_- \subseteq C_- : 1\leqslant |S_+|=|S_-| \leqslant \frac{n}{4}, d_-(S_-\setminus S_+) + d_-(S_+ \setminus S_-) > d_+(S_+)+d_-(S_-) \} =\bigcup_{k=1}^{\frac{n}{4}}E_k
 \]
 \[
 E_k= \{\exists S_+ \subseteq C_+, S_- \subseteq C_- : |S_+|=|S_-| =k, d_-(S_-\setminus S_+) + d_-(S_+ \setminus S_-) > d_+(S_+)+d_-(S_-)\}
 \]
 \begin{lemma} \label{lem:MLEh}
    \[
          \mathbb{P}\left[E\right] =  \mathbb{P}\left[\bigcup_{k=1}^{\frac{n}{4}}E_k \right] \to 0
    \]
 \end{lemma}
 \begin{proof}
     We prove this by observing the following for a fixed pair $(S_+, S_-)$ of size $k$.
     \[
     \mathbb{P}[E] \leqslant \sum_{ 1\leqslant k \leqslant \frac{n}{4}} \mathbb{P}[E_k] \leqslant  \sum_{ 1\leqslant k \leqslant \frac{n}{4}}  \binom{\frac{n}{2}}{k}^2 \mathbb{P}\left[ d_-(S_-\setminus S_+) + d_-(S_+ \setminus S_-) > d_+(S_+)+d_-(S_-)\right]\to 0
     \]
     
     We now observe that $ d_-(S_-\setminus S_+), d_-(S_+ \setminus S_-), d_+(S_+), d_-(S_-)$ are jointly independent random variables with 
     \[
     d_+(S_+)\overset{d}{=}  N( k(n/2 -k) \mu_1, k(n/2 -k) \tau^2) 
     \]
     \[
     d_+(S_-)\overset{d}{=}  N( k(n/2 -k) \mu_1, k(n/2 -k) \tau^2)
     \]
     \[
     d_-(S_+\setminus S_-) \overset{d}{=} N( k(n/2 -k) \mu_2, k(n/2 -k) \tau^2)
     \]
     \[
     d_-(S_-\setminus S_+) \overset{d}{=} N( k(n/2 -k) \mu_2, k(n/2 -k) \tau^2)
     \]
     
     Combining the above distributional identities we have the following simplification.
     \begin{align*}
     & \mathbb{P}\left[ d_-(S_-\setminus S_+) + d_-(S_+ \setminus S_-) > d_+(S_+)+d_-(S_-)\right] 
     \\
     &
     =\mathbb{P}\left[2\tau\sqrt{k(n/2 -k)}  N(0,1)> 2k(n/2 -k)(\mu_1- \mu_2)\right] 
     \\
     & 
     = \mathbb{P}\left[ N(0,1) > \sqrt{k(n/2 -k)} \left(\frac{\mu_1- \mu_2}{\tau}\right)\right]
     \\
     & 
     \leqslant  e^{-\frac{2 
     k(n -2k)  \text{SNR}\log  n}{n}}
     \end{align*}
     
     From the standard Gaussian tail bounds \ref{eq:Gtail} we have the following for $k=1$ with $\text{SNR}= 1+2\delta$.
     \[
     \binom{n/2}{k}^2  \mathbb{P}\left[ N(0,1) > \sqrt{k(n/2 -k)} \left(\frac{\mu_1- \mu_2}{\tau}\right)\right] \leqslant \frac{1}{4} n^2 e^{-\frac{2 
     (n -2)  \text{SNR}\log  n}{n}}  = \frac{1}{4}  \frac{1}{n^{2\delta + o(1)}} 
     \]

     More generally, let $ \epsilon = \frac{k}{n}$. Then we apply the bounds on the binomial coefficients for $1 \leqslant k \leqslant p$.
     \begin{equation} \label{eq:Binom}
    \left(\frac{p}{k}\right)^k \leqslant \binom{p}{k} \leqslant \left(\frac{ep}{k}\right)^k
     \end{equation}

     We choose  $\epsilon_{\delta}$ be such that $(1- 2\epsilon) \text{SNR} = (1-2\epsilon) (1+2 \delta) \geqslant  1+\delta \implies  2\epsilon \leqslant \frac{\delta}{1+2\delta} =:2\epsilon_{\delta}$. For $\epsilon \leqslant \epsilon_{\delta}$ 
     \[
      \binom{\frac{n}{2}}{k}^2   e^{-\frac{2 
     k(n -2k)  \text{SNR}\log  n}{n}} \leqslant \left(\frac{e}{2k}
     \right)^{2k} e^{2k\log n -\frac{2 
     k(n -2k)  \text{SNR}\log  n}{n}}\leqslant  e^{2k \log n (1- (1- 2\epsilon) \text{SNR})} \leqslant \frac{1}{n^{2k\delta}}
     \]
     
     For $\epsilon \geqslant \epsilon_{\delta}$  we use the simpler binomial identity $\binom{p}{k} \leqslant 2^{p}$.
     \[
     \binom{\frac{n}{2}}{k}^2   e^{-\frac{2 
     k(n -2k)  \text{SNR}\log  n}{n}} \leqslant 2^{n}  e^{-2 
     \epsilon_{\delta}(1 -2\epsilon_{\delta})  \text{SNR} n \log  n}  \leqslant e^{ -n \log n (2 
     \epsilon_{\delta}(1 -2\epsilon_{\delta}) - o(1))}
     \]
     \begin{align*}
         \sum_{ \frac{1}{n}\leqslant \frac{k}{n}  \leqslant \frac{1}{4}}  \binom{\frac{n}{2}}{k}^2  e^{-\frac{2 
     k(n -2k)  \text{SNR}\log  n}{n}} \leqslant \sum_{ \frac{1}{n} \leqslant \frac{k}{n}  \leqslant \epsilon_{\delta}} \frac{1}{n^{2k\delta}} + \sum_{ \epsilon_{\delta} \leqslant \frac{k}{n}  \leqslant  \frac{1}{4}} \frac{1}{e^{n \log n (2 
     \epsilon_{\delta}(1 -2\epsilon_{\delta}) - o(1))}} \to 0.
     \end{align*}
     
 \end{proof}
    \subsection{Semi-definite relaxation}

So far, we have established that the task of exact recovery in our Gaussian weighted stochastic block model is statistically achievable \ref{thm: Positive} by the MLE estimation problem \ref{def:MLE} as soon as $\text{SNR} >1$. But, that is a discrete optimization problem over an exponential-sized solution space and not a polynomially tractable problem in the worst case sense, and also in the approximate-case sense. We now look for potential algorithmic solutions to recover the community labeling $\sigma^*$.

\subsubsection{The statement}

In this section, we prove that the Semi-definite relaxation of the MLE problem can solve the exact recovery problem with probability approaching one as soon as $SNR >1$.  We recall that the community membership under the Gaussian weighted stochastic block model can be represented by a vector $\sigma \in \{ \pm 1 \}^n$ such that $\sigma_i = +1$ if vertex $i$ is in the first community and $\sigma_i = -1$ otherwise. Let $\sigma^*$ correspond to the true community labeling such that  $\sigma^*:[n]\to \{\pm 1\}$ with  $\langle \sigma^*, \mathbf{1}\rangle =0$, where $\mathbf{1}$ is a vector in $\mathbb{R}^{n}$ with all entries equal to $1$. Then the maximum likelihood estimator of $\sigma^*$ for the case can be stated as the following problem which maximizes the number of in-cluster edges minus the number of out-cluster edges.
\begin{equation}\label{def:MLE1}
\max_{\sigma \in \Theta_n }   f_A(\sigma)= \sum_{i,j} A_{ij} \sigma_i\sigma_j = I(\sigma) -O(\sigma)  
\end{equation}
\[
 I(\sigma)= \sum_{\sigma_i\sigma_j =+1} A_{ij}
\]
\[
O(\sigma)=  \sum_{\sigma_i\sigma_j =-1} A_{ij}
\]
This is equivalent to solving the NP-hard minimum graph bisection problem. we will instead consider one of its convex relaxation, the semi-definite relaxation (SDP). Let $Y = \sigma \sigma^T$. Then $Y_{ii} = 1$ is equivalent to $\sigma_i = \pm 1$, and $\langle \sigma,\mathbf{1} \rangle  = 0$ if and only if $\langle Y, \mathbf{J} \rangle = 0$, where $\mathbf{J}$ is the $n \times n$ matrix of all ones. Therefore,  the MLE \ref{def:MLE1} can be rewritten as
\begin{equation} \label{def:MLE2}
\max_{Y, \sigma} \langle A, Y \rangle : Y = \sigma \sigma^T, \, Y_{ii} = 1, \, i \in [n], \, \langle \mathbf{J}, Y \rangle = 0. 
\end{equation}
We observe that the matrix $Y = \sigma \sigma^T$ is a rank-one positive semi-definite matrix. If we relax this condition by dropping the rank-one restriction, we obtain the following convex relaxation of \ref{def:MLE2}, which is a semi-definite program:
\begin{equation} \label{eq:SDP}
\hat{Y}_{\text{SDP}} = \arg\max_{Y} \langle A, Y \rangle : Y \succeq 0, \, Y_{ii} = 1, \, \langle \mathbf{J}, Y \rangle = 0.
\end{equation}

Let $Y^*=\sigma^*\sigma^{*^{\top}}$ and $\mathcal{Y}_n \triangleq\left\{\sigma \sigma^{\top}: \sigma \in\{-1,1\}^n, \langle \sigma ,\mathbf{1}\rangle =0\right\}$. The following result establishes the algorithmic tractability of the exact recovery problem under the  SDP procedure.

\begin{theorem}[\textbf{Semi-definite relaxation achieves exact recovery}] \label{thm:SDP}
Under the Gaussian weighted stochastic block model \ref{def:GWSBM} SBM$(n, \mu_1, \mu_2, \tau^2)$, whenever $\text{SNR}>1$ 
\[
\min _{Y^* \in \mathcal{Y}_n} \mathbb{P}\left\{\hat{Y}_{\mathrm{SDP}}=Y^*\right\}\to  1  \hspace{.1 cm} \text{as} \hspace{.1 cm} n \rightarrow \infty.
\]

\end{theorem}

\begin{remark}
    It is important to highlight that Theorem \ref{thm:SDP} extends naturally to the semi-random model as explored in \cite{FEIGE2001639}. In this model, a graph instantiated from the GWSBM allows a monotone adversary to resample in-cluster edges with a higher mean Gaussian (maintaining the same variance) and resample cross-cluster edges with a lower mean Gaussian (with the same variance), arbitrarily. Despite the potential perception of enhancing community structure visibility, such an adversarial model has been known to disrupt various procedures relying on degrees, local search, or graph spectrum, as discussed in \cite{FEIGE2001639} (for the binary model). The SDP \ref{eq:SDP}, by its design, exhibits robustness against such a monotone adversary, a characteristic observed in \cite{FEIGE2001639} and proven in \cite{Chen_2014}. 
 \end{remark}
 
\subsubsection{A deterministic condition}

The following lemma provides a deterministic sufficient condition for the success of SDP \ref{eq:SDP}.

 \begin{lemma} \label{lem:exist}
 Suppose  $ \exists$ $D^{*}=\operatorname{diag}\left\{d_{i}^{*}\right\}$ with $\{d^*_i\}_{i=1}^{n} \in \mathbb{R}^n$ and $\lambda^{*} \in \mathbb{R}$ such that 
 \begin{equation} \label{eq:DualCert}
  S^{*} \triangleq D^{*}-A+\lambda^{*} \mathbf{J} \hspace{.1 cm} \text{satisfies} \hspace{.1 cm}  S^{*} \succeq 0, \lambda_{2}\left(S^{*}\right)>0, \hspace{.1 cm} \text{and} \hspace{.1 cm} S^{*} \sigma^{*}=0.
 \end{equation}
 
Then SDP recovers the true solution. More precisely, $\widehat{Y}_{\mathrm{SDP}}=Y^{*}$ is the unique solution to \ref{eq:SDP}.
\end{lemma}
Proof.  For auxiliary Lagrangian multipliers denoted by $S \succeq 0, D=\operatorname{diag}\left\{d_{i}\right\}$, and $\lambda \in \mathbb{R}$, we define the Lagrangian dual function of the original semi-definite problem \ref{eq:SDP} as

\begin{equation}
L(Y, S, D, \lambda) :=\langle A, Y\rangle+\langle S, Y\rangle+\langle D, \mathbf{I}-Y\rangle+ \langle-\lambda\mathbf{J}, Y\rangle = \langle A+S-D -\lambda \mathbf{J}, Y \rangle  +\langle D, \mathbf{I} \rangle
\end{equation}

 Then for any $Y$ satisfying the constraints in \ref{eq:SDP}, we have the following weak duality inequality.

\begin{equation}
\langle A, Y\rangle \stackrel{(a)}{\leqslant} L\left(Y, S^{*}, D^{*}, \lambda^{*}\right)\stackrel{(b)}{=}\left\langle D^{*}, I\right\rangle \stackrel{(c)}{=}\left\langle D^{*}, Y^{*}\right\rangle \stackrel{(d)}{=}\left\langle A+S^{*}-\lambda^{*} \mathbf{J}, Y^{*}\right\rangle \stackrel{(e)}{=}\left\langle A, Y^{*}\right\rangle
\end{equation}

where $(a)$ holds because $ S^*\succeq 0 \implies \left\langle S^{*}, Y\right\rangle \geqslant 0,  Y_{ii}=1 $ $\forall$ $1\leqslant i \leqslant n,$ and $\langle \mathbf{J}, Y\rangle =0$,  $(b)$ holds because  $S^* =D^* -A + \lambda^* \mathbf{J}$. $(c)$ holds because $D^*$ is diagonal and $Y^*_{ii}=1$ $\forall$ $1\leqslant i \leqslant n$. Further, $(d)$ holds by  \ref{eq:DualCert}. Finally, $(e)$ holds because  $\langle \mathbf{J}, Y^* \rangle =0$ and $\left\langle S^{*}, Y^{*} \right\rangle =\sigma^{*^{\top}} S^{*} \sigma^{*}=0$ by \ref{eq:DualCert} again. Therefore, $Y^{*}$ is an optimal solution to \ref{eq:SDP}. 
\par 
We now establish its uniqueness. To this end, suppose $Y_2$ be another optimal solution to \ref{eq:SDP}. Then we have,

\[
\left\langle S^{*}, Y_2\right\rangle\stackrel{(a)}{=}\left\langle D^{*}-A+\lambda^{*} \mathbf{J}, Y_2\right\rangle \stackrel{(b)}{=}\left\langle D^{*}-A +\lambda^* \mathbf{J}, Y^*\right\rangle \stackrel{(c)}{=}\left\langle S^{*}, Y^{*}\right\rangle=0 .
\]

where  $(a)$ holds by \ref{eq:DualCert}, and  $(b)$ holds because $\langle\mathbf{J}, Y_2\rangle=0= \langle\mathbf{J}, Y^*\rangle$.  Further,   $\langle A, Y_2\rangle=\left\langle A, Y^{*}\right\rangle$ by the optimality of both $Y^*$ and $Y_2$. Finally, $Y_{2_{ii}}=Y_{i i}^{*}=1$ for all $ 1\leqslant i \leqslant n \implies $ $\langle D^*, Y_2\rangle = \langle D^*, Y^*\rangle$.  From \ref{eq:DualCert}, we have  $Y_2 \succeq 0, S^{*} \succeq 0$ with $\lambda_{2}\left(S^{*}\right)>0, Y_2$ must be a multiple of $Y^{*}=\sigma^{*}\sigma^{*^{\top}}$. Because $Y_{2_{ii}}=1$ for all $i \in[n], Y_2=Y^{*}$.

\begin{lemma} \label{lem:uniq}
 With all the notations above, we have that $\langle S^*, Y_2\rangle = 0$  $\implies Y_2= Y^*$.
\end{lemma}
\begin{proof}
    With eigenvalues  $0= \lambda_1 < \lambda_2 \leqslant \lambda_3  \cdots  \leqslant \lambda_n$  and the corresponding  orthonormal set of eigenvectors $ v_1 = \frac{\sigma^*}{\sqrt{n}}, v_2, \cdots, v_n$ the eigenvalue decomposition of $S^*$  is given by the following representation.
    \[
    S^*= \sum_{i=1}^n \lambda_i v_iv_i^*,
    \]
  
 Similarly, the eigenvalue decomposition of $Y_2$ is  given by the following representation.
    \[
    Y_2= \sum_{i=1}^n \alpha_i w_iw_i^*,
  \]
  
  where $ \alpha_1 \geqslant  \alpha_2 \cdots  \geqslant \alpha_n \geqslant 0$ are the eigenvalues and $ w_1, w_2, \cdots, w_n$ are orthonormal set of eigenvectors.
  \[
  \langle S^* , Y_2 \rangle =0 \smallimplies \sum_{i,j} \lambda_i\alpha_j \langle v_i, w_j \rangle ^2 =0 \smallimplies \lambda_i\alpha_j \langle v_i, w_j \rangle ^2 =0 \hspace{.001 cm} \forall \hspace{.001 cm} i,j \geqslant 1 \smallimplies \alpha_j \langle v_i, w_j \rangle ^2 =0 \hspace{.001 cm} \forall \hspace{.001 cm} i \geqslant 2,j \geqslant 1
  \]
  \[
   Y_{2_{ii}}=1 \hspace{.1cm } \forall \hspace{.001cm} 1\leqslant i \leqslant n \implies Y_2 \neq 0 \smallimplies \alpha_1 >0 \smallimplies w_1 \perp \{v_2, \cdots, v_n\}\smallimplies w_1w_1^{\top}= v_1v_1^{\top}= \frac{1}{n}\sigma^*\sigma^{*^{\top}} 
  \]
   
  Finally, we have that  $\alpha_2 =\alpha_3 =\cdots =\alpha_n=0$ because otherwise for any $j\geqslant 2$
  \[
  \alpha_j >0 \implies \langle v_i, w_j \rangle ^2 =0  \hspace{.1 cm} \forall \hspace{.1 cm} i \geqslant 2 \implies  w_j \perp \{v_2, \cdots, v_n\}\implies w_jw_j^{\top}= v_1v_1^{\top}= \frac{1}{n}\sigma^*\sigma^{*^{\top}}
  \]

  But, $\{w_j\}_{j=1}^n$ are by construction orthonormal set of vectors and hence the contradiction.

  Summarizing the above, we have that $ Y_2= \alpha_1v_1v_1^{\top}$ but $\text{Tr}(Y_2)=n \implies \alpha_1 =n$ and therefore $Y_2= \sigma^*\sigma^{*^{\top}}=Y^*$.
\end{proof}

\subsubsection{A high probability event and the proof} \label{sec:hpe}

We are now in a position to provide a proof of Theorem  \ref{thm:SDP}. From \ref{eq:DualCert}  we need to construct the carefully chosen dual certificates $S^*, D^*, \lambda^*$ satisfying all the criteria in \ref{eq:DualCert}. We start by constructing $D^{*}=\operatorname{diag}\left\{d_{i}^{*}\right\}$. 
From the conditions $S^*= D^* -A+\lambda^*\mathbf{J}$ and $S^*\sigma^* =0$ we have $D^*\sigma^*= A\sigma^*$ because $\mathbf{J}\sigma^*=0$. More precisely, we have for all $1\leqslant i \leqslant n$

\[
d_{i}^{*}=\sum_{j=1}^{n} A_{i j} \sigma_{i}^{*} \sigma_{j}^{*}
\]

We choose $\lambda^{*}= \frac{\mu_1+\mu_2}{2}$. It remains to show that $S^{*}=D^{*}-A+\lambda^{*} \mathbf{J}$ satisfies  \ref{eq:DualCert} with probability approaching one.

By construction,  $S^{*} \sigma^{*}=0$ because   $D^{*} \sigma^{*}=A \sigma^{*}$ and $\mathbf{J} \sigma^{*}=0$.  
It remains to verify that $S^{*} \succeq 0$ and $\lambda_{2}\left(S^{*}\right)>0$ with probability approaching one, which is equivalent to showing the following

\[
  \mathbb{P}\left[S^* \succeq 0\cap \lambda_2(S^{*})>0\right] =\mathbb{P}\left[\inf _{x \perp \sigma^{*},\|x\|=1} \langle x, S^{*} x \rangle >0\right] \to 1
\]

We observe  that $\mathbb{E}[A]=\frac{\mu_1-\mu_2}{2} Y^{*}+\frac{\mu_1+\mu_2}{2} \mathbf{J}-\mu_1 \mathbf{I}$ and $Y^{*}=\sigma^{*}\sigma^{*^{\top}}$. Thus for any $x$ such that $x \perp \sigma^{*}$ and $\|x\|=1$,

\begin{align}
\langle x, S^{*} x \rangle  & =\langle x, D^{*} x \rangle - \langle x, \mathbb{E}[A] x \rangle +\lambda^{*} \langle x, \mathbf{J} x \rangle - \langle x, (A-\mathbb{E}[A]) x \rangle  \\
& = \langle x, D^{*} x \rangle  +  \frac{\mu_2- \mu_1}{2}\langle x, Y^* x \rangle + \left(\lambda^{*}- \frac{\mu_1+\mu_2}{2}\right) \langle x, \mathbf{J} x \rangle +\mu_1- \langle x, (A-\mathbb{E}[A]) x \rangle \\
& \stackrel{(a)}{=} \langle  x, D^{*} x \rangle + \mu_1 -\langle x, (A-\mathbb{E}[A]) x \rangle  \geqslant \min _{i \in[n]} d_{i}^{*}+\mu_1 -\|A-\mathbb{E}[A]\| .
\end{align}

where $(a)$ holds since $\lambda^{*} = \frac{\mu_1+\mu_2}{2}$ and $\left\langle x, \sigma^{*}\right\rangle=0$. Therefore, it suffices to prove the following.
\begin{equation}
\mathbb{P}\left[ \min _{i \in[n]} d_{i}^{*}+\mu_1 -\|A-\mathbb{E}[A]\| >0\right] \to  1
\end{equation}

From standard bounds \cite{Vershynin_2018} on norms of Gaussian random matrices, we have that for an absolute constant $c >0$
\begin{equation} \label{def:MatA-A^*}
A=\begin{bmatrix}
\begin{array}{c|c}
N(0, \tau^2) & N(0, \tau^2) \\[1ex]
\hline
\\[-2ex]
N(0, \tau^2) & N(0, \tau^2)
\end{array}
\end{bmatrix}.
\end{equation}
\begin{equation}\label{eq:GNorm}
\mathbb{P}\left[\|A-\mathbb{E}[A]\| \geqslant 3\tau\sqrt{n}\right] \leqslant 2e^{-cn}.
\end{equation}

Therefore, it is enough to prove the following as $|\mu_1| =\left|\alpha \frac{\sqrt{\log n}}{\sqrt{n}}\right| \leqslant \tau \sqrt{n} $ for $n $ large enough.
\begin{equation}
    \mathbb{P}\left[ \min _{i \in[n]} d_{i}^{*}> 4\tau\sqrt{n}\right] =  \mathbb{P}\left[ \bigcap_{i=1}^{n} \{d_{i}^{*}> 4\tau\sqrt{n}\}\right] \to  1
\end{equation}

We prove this using a union bound argument and proving the following.
\[
\sum_{i=1}^{n} \mathbb{P}\left[ d^*_i \leqslant 4\tau \sqrt{n} \right] = n \mathbb{P}\left[ d^*_i \leqslant 4\tau \sqrt{n} \right]\to  0
\]

 To prove this, we observe that for all $1\leqslant i \leqslant n$, 
 \[
 d_{i}^*=\sum_{j=1}^{n} A_{i j} \sigma_{i}^{*} \sigma_{j}^{*} \overset{ d}{=} N\left(\left(\frac{n}{2} -1\right)\mu_1, (\frac{n}{2}-1)\tau^2\right) - N\left(\frac{n}{2}\mu_2, \frac{n}{2}\tau^2\right).
 \]

 \begin{align*}
 n\mathbb{P}\left[ d^*_i \leqslant 4\tau \sqrt{n} \right]  \leqslant n \mathbb{P}\left[ N(0,1) \geqslant \frac{n(\mu_1 -\mu_2)}{2\tau \sqrt{n}}(1+ o(1)\right] \leqslant ne^{-\text{SNR} (1+o(1)) \log n} = \frac{1}{n^{\text{SNR} -1 +o(1)}} \to  0.
 \end{align*}

    \subsection{Spectral relaxation}
In this section, we prove that spectral relaxation achieves the exact recovery up to the statistical threshold. Our result in this section relies on a more general spectral result that we explain below along with proving the assumptions. 

\subsubsection{A more general spectral separation result}
Consider a general symmetric random matrix $B \in \mathbb{R}^{n \times n}$ with independent entries on and above its diagonal. To be more precise, this is a sequence of random matrices with growing dimensions. In our context, this will be the 
 shifted adjacency matrix of the Gaussian weighted stochastic block model \ref{def:GWSBM} as $n \to \infty$. We denote its eigenvalues $\lambda_1 \geqslant \cdots \geqslant \lambda_n $ and their corresponding orthonormal set of eigenvectors $u_1, \cdots, u_n$ with the spectral decomposition given by the following. 

\[
B=\sum_{j=1}^n \lambda_j u_ju_j^{\top}.
\]

Suppose its expectation $B^*=\mathbb{E} B \in \mathbb{R}^{n \times n}$ is low-rank and has $r$ nonzero eigenvalues. 
\begin{enumerate}

\item[(a)]

We assume that $r=\Theta(1)$, and these $r$ eigenvalues are positive and in descending order $\lambda_1^* \geqslant  \lambda_2^* \geqslant \cdots \geqslant \lambda_r^* >0,$ and $\lambda_1^* \asymp \lambda_r^*$. Their corresponding eigenvectors are denoted by $u_1^*, \cdots, u_r^* \in \mathbb{R}^n$. In other words, we have the spectral decomposition of $\mathbb{E}[B]$ given by the following
\[
B^*=\sum_{j=1}^r \lambda_j^* u_j^*u_j^{*^\top}.
\]

For a fixed $k \in[r]$ we define the spectral gap with the convention that $\lambda_0^*=+\infty$ and $\lambda_{n+1}^*=-\infty$.
\[
\Delta^*= (\lambda_{k-1}^*-\lambda_k^*) \wedge (\lambda_k^*-\lambda_{k+1}^*).
\]

\item[(b)]

We assume $B$ concentrates under the spectral norm. More precisely,  $\exists $ $\gamma=\gamma_n=o(1)$ such that 
\[ \mathbb{P}\left[\left\|B-B^*\right\|_2 \geqslant \gamma_n \Delta^*\right] \to 0.\]

As a direct consequence of $(b)$ and because of Weyl's inequality 
\[
\left|\lambda_k-\lambda_k^*\right| \leqslant \left\|B-B^*\right\|_2,
\]

we have that the fluctuation of $\lambda_k$ is much smaller than the gap $\Delta^*$,  Thus, $\lambda_k$ is well separated from other eigenvalues, including the bulk $n-r$ eigenvalues whose magnitudes are at most $\|B-B^*\|_2$.

\item[(c)]
In addition, we assume that $B$ concentrates in a row-wise sense. More precisely, there exists a continuous non-decreasing function $\varphi =\varphi_n: \mathbb{R}_{+} \rightarrow \mathbb{R}_{+}$, such that 
$\varphi(0)=0, \varphi(x) / x$  is non-increasing, and for any   $w \in \mathbb{R}^n
$
\[
\sum_{m=1}^{n}
\mathbb{P} 
\left[ \left| \langle \left(B-B^*\right)_m,  w \rangle \right| \geqslant  \Delta^*\|w\|_{\infty} \varphi\left(\frac{\|w\|_2}{\sqrt{n}\|w\|_{\infty}}\right) \right]\to 0 .
\]
\end{enumerate}

where, the notation $\left(A-A^*\right)_m$. means the $m$-th row vector of $A-A^*$,
 $||w||_2$, $||w||_{\infty}$ are the $l_2$ and $l_{\infty}$ norms on $\mathbb{R}^n$.
\begin{theorem}[\cite{abbe2019entrywise}] \label{thm:Spec}
Fix $k \in[r]$. Suppose assumptions (a), (b), and (c) hold, and $\left\|u_k^*\right\|_{\infty} \leqslant \gamma_n$. Then, with  the notations $O(\cdot)$ and $o(\cdot)$ hide dependencies on $\varphi(1)$ we have
\[
\mathbb{P}\left[
\left\|u_k-  \frac{Bu_k^*}{\lambda_k^*}\right\|_{\infty} \wedge \left\|u_k+ \frac{Bu_k^*}{ \lambda_k^*}\right\|_{\infty}=O\left((\gamma+\varphi(\gamma))\left\|u_k^*\right\|_{\infty}\right)=o\left(\left\|u_k^*\right\|_{\infty}\right),\right] \to  1
\]
\end{theorem}

\begin{Notations}
Under the Gaussian weighted stochastic block model   \ref{def:MatA} we further assume that $A_{ii} \overset{IID}{\sim} N(\mu_1, \tau^2)$. 
\begin{equation} \label{def:MatA'}
    A:=\begin{bmatrix} 
    \begin{array}{c|c}
    N(\mu_1, \tau^2) & N(\mu_2, \tau^2) \\[1ex]
    \hline
    \\[-2ex]
    N(\mu_2, \tau^2) & N(\mu_1, \tau^2)
    \end{array}
    \end{bmatrix}.
    \end{equation}
    \begin{equation} \label{def:MatA^*}
A^*:= \mathbb{E}[A] =\begin{bmatrix}
\begin{array}{c|c}
\mu_1 & \mu_2 \\[1ex]
\hline
\\[-2ex]
\mu_2 & \mu_1
\end{array}
\end{bmatrix} = \lambda^* u^* u^{*^{\top}} + \lambda_1^* u_1^* u_1^{*^{\top}}.
\end{equation}
    \begin{equation}
    \lambda^* := \frac{n (\mu_1 +\mu_2)}{2}, \quad  \text{with} \quad u^* := \frac{\mathbf{1}_n}{\sqrt{n}} 
    \end{equation}
    \begin{equation}
    \lambda_1^*  := \frac{n (\mu_1 -\mu_2)}{2} \quad  \text{with}  \quad u_1^*  := \frac{\left(\mathbf{1}_{\frac{n}{2}}, -\mathbf{1}_{\frac{n}{2}}\right)}{\sqrt{n}} = \frac{\sigma^*}{\sqrt{n}}
    \end{equation}
    \begin{equation} \label{def:MatB}
    B:= A - \frac{n (\mu_1 +\mu_2)}{2} \frac{\mathbf{1}_n}{\sqrt{n}}\frac{\mathbf{1}_n}{\sqrt{n}}^{\top} =A - \lambda^* u^* u^{*^{\top}}
    \end{equation}
    \begin{equation}  \label{eq:MatB^*}
    B^*= \mathbb{E}[B] =\mathbb{E}[A] - \lambda^* u^* u^{*^{\top}} = \lambda_1^* u_1^* u_1^{*^{\top}} 
    \end{equation}
    
    Let $u_1$ be the eigenvector corresponding to the largest eigenvalue $\lambda_1$ of $B$.
    
\end{Notations}

\begin{theorem}[\textbf{Spectral relaxation achieves exact recovery}] \label{thm:Spec1}
    Under the Gaussian weighted stochastic block model \ref{def:GWSBM} and  \ref{def:MatA} with the assumption that $A_{ii} \overset{IID}{\sim} N(\mu_1, \tau^2)$ we have the shifted adjacency matrix $B$ as defined above.
    If $\text{SNR}$ $>1$  then $\hat{\sigma}:= \text{sign} (u_1)$ achieves exact recovery where $u_1$ is the eigenvector corresponding to the largest eigenvalue $\lambda_1$ of $B$. More precisely, if \text{SNR} $>1$

    \[
    \mathbb{P} \left[ M(\hat{\sigma}, \sigma^*) =1\right] \to 1
    \]

\end{theorem}
We break the symmetry of the problem by assuming that the first coordinate of $u_1$ has to be positive, and therefore 

\begin{equation}\label{def:ER1}
\{M(\hat{\sigma}, \sigma^*) =1\}  \iff G^+ \cap G^-
\end{equation}
\[
G^+:=\left\{ u_{1_{i}} >0 \hspace{.1 cm} \forall \hspace{.1 cm}  1 \leqslant i \leqslant \frac{n}{2}\right\} 
\]
\[
G^-:=\left\{u_{1_{i}} <0\hspace{.1 cm}   \forall  \hspace{.1 cm}  \frac{n}{2} +1 \leqslant i \leqslant n\right\}
\]

Otherwise, we have the following generalized set equality without the restriction on $u_1$.
\[
\{M(\hat{\sigma}, \sigma^*) =1\}  \iff  (G^+ \cap G^-) \cup  (H^+ \cap H^-)
\]
\[
H^+:=\left\{ u_{1_{i}} <0 \hspace{.1 cm} \forall \hspace{.1 cm}  1 \leqslant i \leqslant \frac{n}{2}\right\} 
\]
\[
H^-:=\left\{u_{1_{i}} >0\hspace{.1 cm}   \forall  \hspace{.1 cm}  \frac{n}{2} +1 \leqslant i \leqslant n\right\}
\]

Now, we establish the right-hand side of the above \ref{def:ER1} by showing that $\mathbb{P}\left[G^+\cap G^-\right]\to 1$. We now apply Theorem \ref{thm:Spec} with the restriction on $u_1$ to have that for $k=1$
\[
\mathbb{P}\left[
\left\|u_1-  \frac{Bu_1^*}{\lambda_1^*}\right\|_{\infty} =O\left((\gamma+\varphi(\gamma))\left\|u_1^*\right\|_{\infty}\right)=o\left(\left\|u_1^*\right\|_{\infty}\right),\right] \to  1
\]

Finally,  it remains to establish the following.
\[
\mathbb{P}\left[ \left(\frac{Bu_1^*}{ \lambda_1^*}\right)_i > o\left(\left\|u_1^*\right\|_{\infty}\right) \quad \forall \quad  1\leqslant i \leqslant \frac{n}{2}\right] \to 1
\]
\[
\mathbb{P}\left[ \left(\frac{Bu_1^*}{ \lambda_1^*}\right)_i < -o\left(\left\|u_1^*\right\|_{\infty}\right) \quad \forall \quad  \frac{n}{2} +1 \leqslant i \leqslant n \right] \to 1
\]
\subsubsection{The proof -- application of the $l_{\infty}$ eigenvector bound}

We now prove Theorem \ref{thm:Spec1}. We start with the first step of verifying all the assumptions of Theorem \ref{thm:Spec}.

\begin{proof}
For $(a)$, we have $B^* = \lambda_1^* u_1^* u_1^{*^{\top}}$ has $r=1$ positive eigenvalue with $\lambda_1^* =  \frac{n (\mu_1 -\mu_2)}{2} = \frac{(\alpha -\beta)}{2} \sqrt{n \log n}$. 

For $(b)$ we need to prove concentration of $B-\mathbb{E}[B]$. We observe that  the spectral gap $\Delta^*$ is given by 
\[
\Delta^* = \lambda_1^* = \frac{(\alpha -\beta)}{2} \sqrt{n \log n}.
\]

Further, noting that $A-A^* =B-B^*$ and from standard bounds on norms of Gaussian random matrices \ref{eq:GNorm} we have
\[ 
\mathbb{P} \left[ \left\|B-B^* \right\| \geqslant \gamma_n \Delta^* = 3\tau \sqrt{n} \right] \leqslant 2 e^{-cn},
\]

where we choose $\gamma_n = \frac{6 \tau }{ (\alpha -\beta)} \frac{1}{\sqrt{\log n}} \to 0 $ as $n \to  \infty$.

For $(c)$ we observe that $ \langle \left(B-B^*\right)_m, w \rangle \sim N\left(0, \tau^2\|w\|_2^2\right)$ and   choose a linear function  $\varphi(\cdot)$ with a constant $c>1$ 
\[
\varphi(x)=\sqrt{2}c x \frac{\tau \sqrt{n \log n}}{\Delta^*}=  \frac{2\sqrt{2}c \tau}{ (\alpha -\beta)}  x
\]
 
It then clearly satisfies  $\varphi(0)=0, \varphi(x) / x$  \text{is non-increasing} and for any $w \in \mathbb{R}^n$
\[
\sum_{m=1}^{n} \mathbb{P}\left(\left| \langle \left(B-B^*\right)_m, \cdot w \rangle \right| \geqslant \Delta^*\|w\|_{\infty} \varphi\left(\frac{\|w\|_2}{\sqrt{n}\|w\|_{\infty}}\right) \right) = n \mathbb{P}\left[ N(0,1) \geqslant  c\sqrt{\log n}\right] \leqslant n ^{1- c^2} \to 0
\]

which directly follows from standard tail bounds on Gaussian random variables  \ref{eq:Gtail}. 

Having satisfied all the conditions $ (a), (b), $ and $(c)$ of Theorem \ref{thm:Spec}, and with the restriction on $u_1$ and $u_1^*$ we have 
\begin{equation}\label{eq:Sbound}
\mathbb{P}\left[\left\|u_1-  \frac{Bu_1^*}{\lambda_1^*}\right\|_{\infty}  \leqslant C\gamma_n \left\|u_1^*\right\|_{\infty},\right] \to  1 \quad \text{ for an absolute constant} \quad C>0
\end{equation}

In order to prove that $\mathbb{P} \left[G^+ \cap G^-\right] \to 1$   using \ref{eq:Sbound} we observe that 
\[
u_{1} = u_1-  \frac{Bu_1^*}{\lambda_1^*} +  \frac{Bu_1^*}{\lambda_1^*}.
\]

Noting the fact that $\langle u^*, u_1^*\rangle = 0$, we have $Bu_1^* =Au_1^*$.  So, it is enough to prove that 

 \[
\mathbb{P}\left[ \left(\frac{Au_1^*}{ \lambda_1^*}\right)_i >   2C \gamma_n ||u_1^*||_{\infty} \quad \forall \quad  1\leqslant i \leqslant \frac{n}{2}\right] = \mathbb{P}\left[ \left(Au_1^*\right)_i >   6C \tau \quad \forall \quad  1\leqslant i \leqslant \frac{n}{2}\right]\to 1
\]
\[
\mathbb{P}\left[ \left(\frac{Au_1^*}{ \lambda_1^*}\right)_i < -  2C \gamma_n ||u_1^*||_{\infty} \quad \forall \quad  \frac{n}{2}+1\leqslant i \leqslant n \right]= \mathbb{P}\left[ \left(Au_1^*\right)_i <  - 6C \tau \quad \forall \quad  \frac{n}{2} +1 \leqslant i \leqslant n\right] \to 1
\]

We prove these two bounds using union bounds and observing the following distributional identities.
\[
(Au_1^*)_i \overset{d}{=} \frac{n}{2} (\mu_1-\mu_2) - \tau \sqrt{n}N(0,1)  \quad \forall \quad  1\leqslant i \leqslant \frac{n}{2}
\]
\[
(Au_1^*)_i \overset{d}{=} \frac{n}{2} (\mu_2-\mu_1) + \tau \sqrt{n}N(0,1)  \quad  \forall \quad  \frac{n}{2}+1\leqslant i \leqslant n 
\]
\begin{align}
& \frac{n}{2} \mathbb{P}\left[ N(0,1) \geqslant  \frac{n(\mu_1-\mu_2)}{2\tau \sqrt{n}} -6\frac{C}{\sqrt{n}} \right] + \frac{n}{2} \mathbb{P}\left[ N(0,1) \geqslant  \frac{n(\mu_1-\mu_2)}{2\tau \sqrt{n}} -6\frac{C}{\sqrt{n}} \right] \\
& = n  \mathbb{P}\left[ N(0,1) \geqslant  \frac{(\alpha -\beta) \sqrt{\log n}}{2\tau } (1 + o(1))\right] \leqslant n e^{-\text{SNR} (1+ o(1)) \log n} = \frac{1}{n^{\text{SNR} -1 + o(1)}} \to 0
\end{align}
\end{proof}
    \section{Gaussian weighted planted densest subgraph model}

This section defines our Gaussian weighted planted dense subgraph model of linear proportional size.

\subsection{Model description}

Let PDSM$(n, \mu_1,\mu_2, \tau^2)$ denote a weighted graph $G_{n}$ on $n$  vertices with binary labels $(\zeta_i^*)_{i\in [n]}$ such  that  

\[
\#\{i\in [n]: \zeta_i^*=1\}= \gamma n
\]
\[
\#\{i\in [n]: \zeta_i^*=0\}= (1-\gamma) n
\] 

and conditioned on this labeling $(\zeta_i^*)_{i\in [n]}$, for each pair of distinct nodes $i,j \in [n]$, we have a weighted edge 
\begin{equation}\label{def:GWPDSM}
A_{ij}=A_{ji} \sim
\begin{cases}
  N(\mu_1, \tau^2) & \text{if } \zeta^*(i) =\zeta^*(j) =1   \\
   N(\mu_2, \tau^2) & \text{otherwise}   
\end{cases}
\end{equation} 

sampled independently for every distinct pair $\{i,j\}$.  To simplify the writing, we focus on the assortative case, $\mu_1>\mu_2$.  We further define the signal-to-noise ratio of the model PDSM$(n, \mu_1,\mu_2, \tau^2)$ as the following.
\begin{equation}  \label{eq:SNR_n1}
    \text{SNR}_n: =  \frac{(\mu_1-\mu_2)^2 n}{8\tau^2 \log n},
\end{equation}
\begin{equation}
    \text{SNR}:= \lim_{n\to\infty} \text{SNR}_n.
\end{equation}

To simplify our writings further, we work with the critical scaling of $\mu_1$ and $\mu_2$ as the following.
\begin{equation} \label{def:scaling1}
\mu_1= \alpha \sqrt{\frac{\log n}{n}}, \mu_2= \beta\sqrt{\frac{\log n}{n}}.
\end{equation}

In this notation, we have our signal-to-noise ratio given by the following with $\alpha >\beta $ and $ \tau >0$ constants.

\begin{equation} \label{eq:SNR1}
\text{SNR}:= \frac{(\alpha -\beta)^2}{8\tau^2}.
\end{equation}

 The problem of exact recovery of the densely weighted community asks to recover the planted community labeling $(\zeta_i^*)_{i\in [n]}$ exactly having observed an instance of the weighted graph $G_n\sim$ PDSM$(n, \mu_1\,\mu_2, \tau^2)$ or equivalently the random $n \times n $ symmetric (weighted) adjacency matrix $(A_{ij})_{i,j,\in [n]}$ of the graph. 
 Let us  denote the parameter space $\Theta^{\gamma}_n$ as 
 \begin{equation} \label{eq:SOL1'}
 \Theta^{\gamma}_n = \{\zeta\in \{0, 1\}^n, \langle\zeta, \mathbf{1}_n\rangle=\gamma n\},
 \end{equation}
 
 where 
 $\mathbf{1}_n := (1, \cdots, 1)\in \mathbb{R}^n$. To be more precise, we define the agreement ratio between two community vectors $\hat{\zeta},  \zeta  \in \{0, 1\}^n$ as 
 \begin{equation} \label{def:agr1'}
     M(\hat{\zeta}, \zeta)= \frac{1}{n}\sum_{i=1}^n \mathbf{1}_{\hat{\zeta}(i)=\zeta(i)}
 \end{equation}

Then, we say that the problem of exact recovery is solvable if having observed an instance of the weighted random graph $G_n \sim$ PDSM$(n,\mu_1, \mu_2,\tau^2)$ with an unobserved $\zeta^* \in \Theta^{\gamma}_n$, we output a community label $\hat{\zeta} \in \Theta^{\gamma}_n
$ such that 
\begin{equation} \label{def:ER2'}
    \p (M(\hat{\zeta},\zeta^*)=1) \to 1.
\end{equation} 

We say the problem of exact recovery is unsolvable otherwise. That is, with probability bounded away from zero, every estimator fails to exactly recover the densely weighted community, as the number of vertices grows.
 
 \par
 
 Without loss of generality, we assume that the unobserved labels $(\zeta_i^*)_{i\in [n]}$ are such that  $\zeta_i^*=1$  for all $1\leqslant i \leqslant \gamma n$ and $\zeta_i^*=0$  for all $\gamma n+1\leqslant i \leqslant  n$. In that case, the weighted adjacency matrix $A$ has the following form (except that the diagonal terms  $A_{ii} =0 $ for all $1\leqslant i\leqslant n$). 

\[ A=
\begin{bmatrix}
\begin{array}{c|c}
N(\mu_1, \tau^2) & N(\mu_2, \tau^2) \\[1ex]
\hline
\\[-2ex]
N(\mu_2, \tau^2) & N(\mu_2, \tau^2)
\end{array}
\end{bmatrix}.
\]

\subsection{Maximum likelihood estimation} 

In this model, we define the maximum likelihood estimator as the following.

\begin{equation} \label{def:MLEPD}  
\hat{\zeta}_{MLE}:= \underset{\zeta \in \Theta^{\gamma}_n}{\argmax} \hspace{.1cm} \log \p(G|((\zeta_i)_{i\in [n]})= \sum_{\zeta_i\zeta_j =1}\frac{-(A_{ij}-\mu_1)^2}{2\tau^2} + \sum_{\zeta_i\zeta_j =0}\frac{-(A_{ij}-\mu_2)^2}{2\tau^2} -\log c_n,
\end{equation}

where $c_n =(\tau \sqrt{2 \pi})^{\binom{n}{2}}$. Further simplifying the right-hand side we have the following.

\begin{lemma} \label{lem: MLE1}
    Let $G_n \sim \text{PDSM}(n,\mu_1, \mu_2, \tau^2)$ with an unknown $\zeta^* \in \Theta^{\gamma}_n$. Then we have the MLE given by the following.
    \begin{equation} \label{eq:MLE3}
   \hat{\zeta}_{MLE}= \underset{\zeta \in \Theta^{\gamma}_n}{ \argmax} \hspace{.1 cm}g_A(\zeta) \quad \quad \text{where} \quad g_A(\sigma):=\sum_{i,j} A_{ij}\zeta_i \zeta_j 
\end{equation}
\end{lemma}
\begin{proof}
    We have the following from the definition of the maximum likelihood estimator \ref{def:MLEPD}.
    \[\hat{\zeta}_{MLE}=  \underset{\zeta \in \Theta^{\gamma}_n}{\argmax} \sum_{\zeta_i\zeta_j =1}\frac{-(A_{ij}-\mu_1)^2}{2\tau^2} + \sum_{\zeta_i\zeta_j =0}\frac{-(A_{ij}-\mu_2)^2}{2\tau^2}.\]
    
 We start by simplifying the right-hand side of the expression above.
 
\begin{align*}
& \hspace{ .38 cm} \sum_{\zeta_i\zeta_j = 1}\frac{-(A_{ij}-\mu_1)^2}{2\tau^2} + \sum_{\zeta_i\zeta_j = 0}\frac{-(A_{ij}-\mu_2)^2}{2\tau^2} \\
&= \sum_{\zeta_i\zeta_j = 1}\frac{-A_{ij}^2 - \mu_1^2 + 2A_{ij}\mu_1}{2\tau^2} + \sum_{\zeta_i\zeta_j = 0}\frac{-A_{ij}^2 - \mu_2^2 +2A_{ij}\mu_2}{2\tau^2} \\
&= \sum_{\{i,j\}} \frac{-A_{ij}^2}{2\tau^2} + \frac{-\mu_1^2}{2\tau^2}\times \binom{\gamma n}{2}
+  \frac{-\mu_2^2}{2\tau^2} \times \left(\binom{n}{2} -\binom{\gamma n}{2}\right) + \sum_{\zeta_i\zeta_j =1} \frac{A_{ij}\mu_1}{\tau^2} + \sum_{\zeta_i\zeta_j =0} \frac{A_{ij}\mu_2}{\tau^2},
\end{align*}

where the index $\{ i,j\}$ means to sum over all the $\binom{n}{2}$ edges of the graph $G_n$. The total number of edges in the densely weighted community is given by $\binom{\gamma n}{2}$.  We further observe that the above expression simplifies to the following.

\begin{align*}
    & \hspace{.25 cm} \sum_{\zeta_i\zeta_j =1} A_{ij}\mu_1 + \sum_{\zeta_i\zeta_j =0} A_{ij}\mu_2 \\
    & = \sum_{\{i,j\}} A_{ij}\mu_1 \zeta_i\zeta_j +\sum_{\{i,j\}} A_{ij}\mu_2 (1-\zeta_i\zeta_j) \\
    & = (\mu_1-\mu_2)\sum_{\{i,j\}} A_{ij} \zeta_i\zeta_j +\sum_{\{i,j\}} A_{ij}\mu_2 \\
    & = \frac{(\mu_1-\mu_2)}{2}\sum_{i,j} A_{ij} \zeta_i\zeta_j   +\sum_{\{i,j\}} A_{ij}\mu_2
\end{align*} 

As most of the terms above are independent of $\zeta$, and $\mu_1 >\mu_2$, we have the following expression for MLE.

   \[\hat{\zeta}_{MLE}=  \argmax \langle A, \zeta \zeta^{\top} \rangle  : \zeta\in \{0, 1\}^n, \langle\zeta, \mathbf{1}_n\rangle=\gamma n
   \]
\end{proof}

    \subsection{Statistical impossibility}

In this section, we present the statistical impossibility result, providing an information theoretic limit for the exact recovery problem in Gaussian weighted planted dense subgraph problem by showing that MLE fails to recover the community label in a low SNR regime. More precisely, we prove that when  $\gamma \text{SNR}$ is less than three quarter, no algorithm (efficient or not) can exactly recover the community labels with probability approaching one.

 \subsubsection{The statement}
\begin{theorem}[\textbf{Statistical impossibility of exact recovery of the planted community}] \label{thm:Negative1}
    Under the Gaussian weighted planted dense subgraph model \ref{def:GWPDSM} with parameter $\gamma \in  (0,1)$, 
      If  $\gamma$ SNR$<\frac{3}{4}$ then MLE fails in recovering the community label $\zeta^*$ bounded away from zero. 
     \[ \underset{n \to \infty}{\limsup}\hspace{.1cm}\p(M(\hat{\zeta}_{MLE}, \zeta^*)=1) =0 \quad \iff \quad  \underset{n \to \infty}{\liminf}\hspace{.1cm}\p(M(\hat{\zeta}_{MLE}, \zeta^*)\neq 1) = 0.\]
\end{theorem}

\begin{remark} \label{rem:MLEPDSM}
To explore the information-theoretic limits, we examine the Maximum A Posteriori (MAP) estimation, designed to maximize the probability of accurately reconstructing communities. In the context of uniform community assignment with $\zeta^*$ being uniform on all subsets of size $\gamma n$, MAP estimation is equivalent to Maximum Likelihood estimation (MLE).  Consequently, since the prior on $\zeta^*$ is uniform, if MLE  fails in exactly recovering the planted community with probability bounded away from zero as the network size (n) grows, there is no algorithm, whether efficient or not, capable of reliably achieving this task. Therefore, to prove that exact recovery is not solvable, it is enough to show that MLE is unable to recover the planted community $\zeta^*$ with probability bounded away from zero. 
\end{remark}
 \par
 
Without loss of generality, we assume that the planted community is  $\zeta^* =(\mathbf{1}_{\gamma n}, \mathbf{0}_{(1-\gamma)n}) \in \Theta_n^{\gamma}$, and $G_n \sim $ PDSM$(n, \mu_1, \mu_2, \tau^2)$ conditioned on $\zeta^*$. We denote the planted community  as
    \[
    C^* := \left\{ i \in [n]: \zeta_i^*=1\right\}=\left\{1, \cdots, \gamma n\right\},
    \]
    \[
    [n] \setminus C^*:=\left\{\gamma n +1, \cdots, n\right\}.
    \]

\begin{definition} 

We also define the set of bad pairs of vertices in $G_n$ by the following.
    \begin{equation} \label{def:Bad1}
    \mathcal{C}(G_n) := \{(i,j): i\in C^*, j\in [n] \setminus C^*, g_A(\zeta^*)< g_A(\zeta^*[i\leftrightarrow j])\},
    \end{equation}
    
    where  $\zeta^*[i\leftrightarrow j]$ denotes the vector obtained by swapping the values of coordinates $i$ and $j$ in $\zeta^*$.
\end{definition}

From \ref{eq:MLE3} we observe that the event $g_A(\zeta^*)< g_A(\zeta^*[i\leftrightarrow j]) $ is the same as the event that the likelihood of the partition $\zeta^*[i\leftrightarrow j]$ is larger than the planted one $\zeta^*$. Hence, we have the following result.

\begin{lemma} \label{lem:MLEa'}
   MLE fails to solve the exact recovery with high probability if $\mathcal{C}(G_n)$ is non-empty with high probability. More precisely, we have the following lower bound on the error.
   \[\underset{n \to \infty}{\liminf}\hspace{.1cm}\p(M(\hat{\zeta}_{MLE}, \zeta^*)\neq1) \geqslant  \underset{n \to \infty}{\liminf}\hspace{.1cm} \p (\mathcal{C}(G_n)\neq \emptyset)\]
\end{lemma}

\begin{proof}
    If $\exists (i,j) \in \mathcal{C}(G_n)$, then we can swap the coordinates $i$ and $j$ in $\zeta^*$ and increase the likelihood of the partition (or equivalently $g_A(\cdot)$), thus obtaining a different partition than the planted one ($\zeta^*$) that is more likely than the planted one. Therefore, $\mathcal{C}(G_n) \neq \emptyset \implies M(\hat{\zeta}_{MLE}, \zeta^*)\neq1 $ More precisely, we have 
    \begin{align*}
        \p(M(\hat{\zeta}_{MLE}, \zeta^*)\neq 1) \geqslant \p (\mathcal{C}(G_n)\neq \emptyset)
    \end{align*} 
\end{proof}
With the lemma above, the impossibility of exact recovery reduces to proving that if $\gamma$ SNR$<\frac{3}{4}$ then 
\[\underset{n \to \infty}{\liminf}\hspace{.1cm}\p (\mathcal{C}(G_n)\neq \emptyset) >0.\]

We will show this by proving a stronger result that the following holds.
\[
\underset{n \to \infty}{\liminf}\hspace{.1cm}\p (\mathcal{C}(G_n)\neq \emptyset) =1.
\]
 \subsubsection{Key step --The second moment method}
 
We now analyze the condition $g_A(\zeta^*)< g_A(\zeta^*[i\leftrightarrow j])$ for $i \in C^*$, and $j \in [n]\setminus C^*$.

\par
\begin{definition}
    For a vertex $i \in C^*$, and $j \in [n]\setminus C^*$ we define the following.
    \[
    e(i, C^*) = \sum_{ k \in C^*} A_{ik} =e(i, C^* \setminus \{i\})
    \]
    \[
    e(j, C^* \setminus \{i\}) = \sum_{ k \in C^*  \setminus \{i\}} A_{jk}
    \]
\end{definition}
\begin{lemma} \label{lem:MLEb'}
With the notations above and $i \in C^*$, and $j \in [n]\setminus C^*$ we have the following.
 
    \[g_A(\zeta^*)< g_A(\zeta^*[i\leftrightarrow j]) \quad \iff \quad  e(i, C^*) < e(j, C^* \setminus \{i\}).\]
\end{lemma}  
\begin{proof}
We recall the definition of MLE estimator \ref{eq:MLE3} given by (recall $A_{uu}:=0$ for all $1\leqslant u\leqslant n$)
\[g_A(\zeta^*) :=  \sum_{1\leqslant u,v \leqslant n} A_{uv} \zeta^*_u\zeta^*_v.\]

Therefore, one has the following chain of equivalences for $i \in C^*$, and $j \in [n]\setminus C^*$.
    \[
    g_A(\zeta^*)< g_A(\zeta^*[i\leftrightarrow j]
    \]
    \[ 
    \leftrightarrow  \hspace{.1 cm} \sum_{1\leqslant u,v \leqslant n} A_{uv} \zeta^*_u\zeta^*_v < \sum_{1\leqslant u,v \leqslant n} A_{uv} \zeta^*_u[i\leftrightarrow j]\zeta^*_v[i\leftrightarrow j])
    \]
    \[ \leftrightarrow \hspace{.1 cm} \sum_{(u,v) \in C^* \setminus \{i\} \times  C^* \setminus \{i\}} A_{uv} + 2e(i, C^*) < \sum_{(u,v) \in C^* \setminus \{i\} \times  C^* \setminus \{i\}} A_{uv} + 2e(j, C^* \setminus \{i\})
    \]
    \[
    \leftrightarrow  e(i, C^*) < e(j, C^* \setminus \{i\}) 
    \]
\end{proof}

With the lemma above we apply the second moment method to  $Z = \sum_{ i \in C^*, j \in [n]\setminus C^*} Z_{ij}$
\[
Z_{ij}=\mathbf{1}_{\mathcal{C}_{ij}}
\quad \text{with} \quad 
\mathcal{C}_{ij} : =\{e(i, C^*) < e(j, C^*\setminus \{i\})\}
\]
\[
\mathcal{C}(G_n) =\{(i,j): i \in C^*, j \in [n]\setminus C^*, e(i, C^*)< e(j,C^*\setminus \{i\})\} = \bigcup_{ i \in C^*, j \in [n]\setminus C^* } \mathcal{C}_{ij}.
\]

It is immediate that $\mathbb{P}\left[\mathcal{C}(G_n) \neq \emptyset\right] =\mathbb{P}\left[Z \geqslant 1\right]$ and we are ready to apply the second moment method to $Z$.

\begin{lemma} \label{lem:MLEc'}
With the notations above, we have $\mathbb{P}\left[Z \geqslant1 \right] \to 1$ as soon as $ \gamma \text{SNR} <\frac{3}{4}$.
\end{lemma}
\begin{proof}
From the second moment method  \ref{eq:SMM} we have the following chain of inequalities.

\[
\begin{aligned}
& \mathbb{P}\left\{Z \geqslant 1\right\}
\geqslant 
\frac{\mathbb{E}\left[Z\right]^2}{ \mathbb{E}\left[Z^2\right]}
\geqslant 
\frac{\left(\mathbb{E} \left[\sum_{(i,j) \in C^* \times [n]\setminus C^*} Z_{ij}\right]\right)^2}{\mathbb{E}\left[\left(\sum_{(i,j) \in C^* \times [n]\setminus C^*} Z_{ij}\right)^2\right]} 
\\
&
=  \frac{\left(\sum_{(i,j) \in C^* \times [n]\setminus C^*}\mathbb{P}\left[\mathcal{C}_{ij}\right] \right)^2}{ \sum_{(i,j) \in C^* \times [n]\setminus C^*} \mathbb{P}\left[\mathcal{C}_{ij}\right] +\sum_{(i_1, j_1) \neq (i_2, j_2)   \in  C^* \times [n]\setminus C^*} \mathbb{P}\left[ \mathcal{C}_{(i_1, j_1)}\cap \mathcal{C}_{(i_2, j_2)}\right]}
\\
&  = \frac{1}{ \frac{1}{\sum_{(i,j) \in C^* \times [n]\setminus C^*} \mathbb{P}\left[\mathcal{C}_{ij}\right]} + \frac{\sum_{(i_1, j_1) \neq (i_2, j_2)   \in  C^* \times [n]\setminus C^*} \mathbb{P}\left[ \mathcal{C}_{(i_1, j_1)}\cap \mathcal{C}_{(i_2, j_2)}\right]}{\left(\sum_{(i,j) \in C^* \times [n]\setminus C^*} \mathbb{P}\left[\mathcal{C}_{ij}\right]\right)^2}}
\end{aligned}
\]

and the last bound tends to $1$ if the following two conditions hold, which we will prove in the following lemma
\begin{equation}
  \sum_{(i,j) \in C^* \times [n]\setminus C^*} \mathbb{P}\left[\mathcal{C}_{ij}\right]\to \infty,
  \end{equation}
  \begin{equation}
   \frac{\sum_{(i_1, j_1) \neq (i_2, j_2)   \in  C^* \times [n]\setminus C^*} \mathbb{P}\left[ \mathcal{C}_{(i_1, j_1)}\cap \mathcal{C}_{(i_2, j_2)}\right]}{\left(\sum_{(i,j) \in C^* \times [n]\setminus C^*} \mathbb{P}\left[\mathcal{C}_{ij}\right]\right)^2} \to 1 .
  \end{equation}

\end{proof}

It remains to show that both the conditions hold as soon as $\gamma \text{SNR} < \frac{3}{4}$ and reveal the bound in the theorem.

\begin{lemma} \label{lem:First1}
If $ \gamma \text{SNR} < 1$ then, we have the following   
\[
\sum_{(i,j) \in C^* \times [n]\setminus C^*} \mathbb{P}\left[\mathcal{C}_{ij}\right]\to \infty.
\]
\end{lemma}
\begin{proof}
 We start by observing that $e(i, C^*)$ and $e(j, C^* \setminus \{i\})$ are independent the following distributional identities  hold
 \begin{equation}
 e(i, C^*) \overset{d}{=} N((\gamma n -1)\mu_1,(\gamma n -1)\tau^2)
 \end{equation}
 \begin{equation}
 e(j, C^*\setminus \{i\}) \overset{d}{=} N((\gamma n -1)\mu_2,(\gamma n -1)\tau^2)
 \end{equation}
 Combining the two above, we have the following
 \begin{align}
 \sum_{(i,j) \in C^* \times [n]\setminus C^*} \mathbb{P}\left[\mathcal{C}_{ij}\right] 
 & = 
 \gamma (1- \gamma )n^2 \mathbb{P}\left[ N((\gamma n -1)\mu_1,(\gamma n -1)\tau^2) < N((\gamma n -1)\mu_2,(\gamma n -1)\tau^2)\right] \\
 & =
 \gamma (1- \gamma )n^2  \mathbb{P}\left[\tau \sqrt{2(\gamma n -1)} N(0,1)>  (\gamma n -1) (\mu_1 -\mu_2) \right] \\
 & =
 \gamma (1- \gamma )n^2  \mathbb{P}\left[ N(0,1) > (1+o(1))\frac{\sqrt{\gamma}(\alpha -\beta) \sqrt{\log n}}{\tau \sqrt{2}}\right] \\
 & \sim
 \gamma (1- \gamma )n^2 \frac{e^{-2\gamma \text{SNR}\log n (1+o(1))}}{\sqrt{2\pi }\frac{\sqrt{\gamma}(\alpha -\beta) \sqrt{\log n}}{\tau \sqrt{2}} } \\
 & \sim \frac{\gamma (1- \gamma )\tau \sqrt{2} }{(\alpha -\beta) \sqrt{2\pi \gamma }} \frac{n^{2(1-\gamma \text{SNR})}}{\sqrt{\log n}} \to \infty
 \end{align}
\end{proof}

We remark that the result holds verbatim even in the general model \ref{eq:SNR_n1} as soon as $\lim_{n \to \infty} 
 \gamma \text{SNR}_n = \gamma \text{SNR} <1$.

\begin{lemma} \label{lem:Second1}

If $ \gamma \text{SNR} <  \frac{3}{4}$ then, we have the following   
\[
\frac{\sum_{(i_1, j_1) \neq (i_2, j_2)   \in  C^* \times [n]\setminus C^*} \mathbb{P}\left[ \mathcal{C}_{(i_1, j_1)}\cap \mathcal{C}_{(i_2, j_2)}\right]}{\left(\sum_{(i,j) \in C^* \times [n]\setminus C^*} \mathbb{P}\left[\mathcal{C}_{ij}\right]\right)^2} \to 1 \iff
\]
\[
\sum_{(i_1, j_1) \neq (i_2, j_2)   \in  C^* \times [n]\setminus C^*} \mathbb{P}\left[ \mathcal{C}_{(i_1, j_1)}\cap \mathcal{C}_{(i_2, j_2)}\right] \sim \left(\sum_{(i,j) \in C^* \times [n]\setminus C^*} \mathbb{P}\left[\mathcal{C}_{ij}\right]\right)^2.
\]
\end{lemma}
\begin{proof}

We start by observing that for $i_1 \neq i_2 \in C^* $ and $j_1 \neq j_2 \in [n] \setminus C^*$ we have that
\[
e(i_1, C^* \setminus \{i_2\}) \overset{d}{=} N( (\gamma n -2)\mu_1, (\gamma n-2) \tau^2)
\]
\[
e(i_2, C^* \setminus \{i_1\}) \overset{d}{=} N( (\gamma n -2)\mu_1, (\gamma n-2) \tau^2)
\]
\[
e_{i_1, i_2} \overset{d}{=} \mu_1+ \tau N(0,1)
\]
\[
e(j_1, C^* \setminus \{i_1\}) \overset{d}{=}  N( (\gamma n -1)\mu_2, (\gamma n-1) \tau^2)
\]
\[
e(j_2, C^* \setminus \{i_2\}) \overset{d}{=} N( (\gamma n -1)\mu_2, (\gamma n-1) \tau^2)
\]

are jointly independent normal random variables and we have that following with  $Z_1, Z_2, Z$ being three IID standard Gaussian random variables.

\begin{align*}
 \mathbb{P}[\mathcal{C}_{(i_1, j_1)}\cap \mathcal{C}_{(i_2, j_2)}] 
&  =
\mathbb{P}\left[ \left\{e(i_1 , C^*) < e(j_1, C^* \setminus \{i_1\})\right\} \bigcap \left\{e(i_2 , C^*) < e(j_2, C^* \setminus \{i_2\})\right\} \right] \\
&  = \mathbb{P}  \left[ \left\{e(i_1 , C^* \setminus \{i_2\}) +e_{i_1, i_2} < e(j_1, C^* \setminus \{i_1\})\right\}  \right. \\
& \quad \bigcap  \left\{e(i_2 , C^*\setminus \{i_1\}) +e_{i_1, i_2} < e(j_2, C^* \setminus \{i_2\})\right\}\Big]  \\
&  = \mathbb{P}  \left[ \left\{N( (\gamma n -2)\mu_1, (\gamma n-2) \tau^2)  + e_{i_1, i_2} <  N( (\gamma n -1)\mu_2, (\gamma n-1) \tau^2)\right\} \right. \\
& \quad \bigcap \left\{N( (\gamma n -2)\mu_1, (\gamma n-2) \tau^2) + e_{i_1, i_2} <  N( (\gamma n -1)\mu_2, (\gamma n-1) \tau^2)\right\}\Big]  \\
&  = \mathbb{P}  \left[ \left\{ (\gamma n -1)(\mu_1 - \mu_2 ) + \tau Z < \tau \sqrt{2\gamma n -3}Z_1\right\}  \right. \\
& \quad \bigcap \left\{(\gamma n -1)(\mu_1-\mu_2)  + \tau Z <  \tau \sqrt{2\gamma n -3}Z_2\right\} \Big]  \\
&  = \mathbb{P}  \left[ \left\{Z_1  > \frac{Z}{ \sqrt{2\gamma n -3}} + (1+o(1))s \sqrt{\log n}\right\} \right. \\
& \quad \bigcap  \left\{Z_2  > \frac{Z}{ \sqrt{2\gamma n -3}} + (1+o(1))s \sqrt{\log n}\right\}\Big]
\end{align*}

where we denote that  
$s=\frac{ (\alpha -\beta) \sqrt{\gamma}}{\tau \sqrt{2}}$.
Further, for $ i\in C^*, j \in [n]\setminus C^*$ we have from Lemma \ref{lem:jointGauss} 
    \begin{equation} \label{eq:FM1}
    \mathbb{P}\left[\mathcal{C}_{ij}\right] = \mathbb{P}\left[ \left\{Z_1  > \frac{Z}{ \sqrt{2\gamma n -3}} + (1+o(1))s \sqrt{\log n}\right\}\right] \sim \frac{1}{ s \sqrt{\log n}}\phi (s \sqrt{\log n})
    \end{equation}
    \begin{equation}\label{eq:SM1}
\mathbb{P}[\mathcal{C}_{(i_1, j_1)}\cap \mathcal{C}_{(i_2, j_2)}] \sim \left(\frac{1}{ s \sqrt{\log n}}\phi (s \sqrt{\log n})\right)^2
\end{equation}

As a consequence, we have the following  asymptotic equality.
\[
\frac{\sum_{i_1 \neq i_2  \in  C^*,j_1\neq  j_2 \in [n]\setminus C^*} \mathbb{P}\left[ \mathcal{C}_{(i_1, j_1)}\cap \mathcal{C}_{(i_2, j_2)}\right]}{\left(\sum_{(i,j) \in C^* \times [n]\setminus C^*} \mathbb{P}\left[\mathcal{C}_{ij}\right]\right)^2} \sim  \frac{ \gamma n (\gamma n -1) ((1-\gamma) n -1) (1- \gamma)n}{\gamma ^2(1-\gamma)^2 n^4} \sim 1
\]

It remains to establish that the following two holds.
\[
\frac{\sum_{i_1 =i_2  \in  C^*,j_1\neq  j_2 \in [n]\setminus C^*} \mathbb{P}\left[ \mathcal{C}_{(i_1, j_1)}\cap \mathcal{C}_{(i_2, j_2)}\right]}{\left(\sum_{(i,j) \in C^* \times [n]\setminus C^*} \mathbb{P}\left[\mathcal{C}_{ij}\right]\right)^2} 
\sim 
\frac{\gamma n (1-\gamma)^2 n^2\mathbb{P}\left[  \mathcal{C}_{(i, j_1)}\cap \mathcal{C}_{(i, j_2)}\right] }{\gamma^2 n^2 (1-\gamma)^2 n^2 \mathbb{P}\left[\mathcal{C}_{ij}\right]^2}
\sim
\frac{\mathbb{P}\left[  \mathcal{C}_{(i, j_1)}\cap \mathcal{C}_{(i, j_2)}\right]}{ \gamma n \mathbb{P}\left[\mathcal{C}_{ij}\right]^2} 
\sim
0
\]
\[
\frac{\sum_{i_1 \neq i_2  \in  C^*,j_1 =  j_2 \in [n]\setminus C^*} \mathbb{P}\left[ \mathcal{C}_{(i_1, j_1)}\cap \mathcal{C}_{(i_2, j_2)}\right]}{\left(\sum_{(i,j) \in C^* \times [n]\setminus C^*} \mathbb{P}\left[\mathcal{C}_{ij}\right]\right)^2}
\sim
\frac{\gamma^2 n^2 (1-\gamma) n\mathbb{P}\left[  \mathcal{C}_{(i_1, j)}\cap \mathcal{C}_{(i_2, j)}\right] }{\gamma^2 n^2 (1-\gamma)^2 n^2 \mathbb{P}\left[\mathcal{C}_{ij}\right]^2}
\sim
\frac{\mathbb{P}\left[  \mathcal{C}_{(i_1, j)}\cap \mathcal{C}_{(i_2, j)}\right]}{(1-\gamma )n \mathbb{P}\left[\mathcal{C}_{ij}\right]^2}  \sim
0
\]

For $i \in C^*$ and $j_1 \neq j_2 \in [n] \setminus C^*$ we have for IID standard normals $Z_1, Z_2,$ and $ Z$
\begin{align*}
 \mathbb{P}[\mathcal{C}_{(i, j_1)}\cap \mathcal{C}_{(i, j_2)}] 
&  =
\mathbb{P}\left[ \left\{e(i , C^*) < e(j_1, C^* \setminus \{i\})\right\} \bigcap \left\{e(i , C^*) < e(j_2, C^* \setminus \{i\})\right\} \right] \\
&  = \mathbb{P}  \left[ \left\{N( (\gamma n -1)\mu_1, (\gamma n-1) \tau^2) <  N( (\gamma n -1)\mu_2, (\gamma n-1) \tau^2)\right\} \right. \\
& \quad \bigcap \left\{N( (\gamma n -1)\mu_1, (\gamma n-1) \tau^2)  <  N( (\gamma n -1)\mu_2, (\gamma n-1) \tau^2)\right\}\Big]  \\
& = \mathbb{P}  \left[ \left\{Z_1 > Z + s\sqrt{2 \log n} (1+o(1)) \right\} \bigcap \left\{Z_2 >  Z + s\sqrt{2 \log n} (1+o(1)) \right\} \right] \\
& = \int_{\mathbb{R}} \Phi^2(z +s\sqrt{2 \log n} (1+o(1))) \phi(z) dz
\end{align*}

For $i_1 \neq i_2 \in C^*$ and $j  \in [n] \setminus C^*$ we have for IID standard normals $Z_1, Z_2,Z_3, Z_4$ and $ Z,Z'$
\begin{align*}
 \mathbb{P}[\mathcal{C}_{(i_1, j)}\cap \mathcal{C}_{(i_2, j)}] 
&  
=
\mathbb{P}\left[ \left\{e(i_1 , C^*) < e(j, C^* \setminus \{i_1\})\right\} \bigcap \left\{e(i_2, C^*) < e(j, C^* \setminus \{i_2\})\right\} \right]
\\
&  
= \mathbb{P}  \left[ \left\{N( (\gamma n -2)\mu_1, (\gamma n-2) \tau^2) +e_{i_1, i_2} <  N( (\gamma n -2)\mu_2, (\gamma n-2) \tau^2) + e_{j, i_2}\right\} \right. 
\\
& 
\quad \bigcap \left\{N( (\gamma n -2)\mu_1, (\gamma n-2) \tau^2) + e_{i_1, i_2}  <  N( (\gamma n -2)\mu_2, (\gamma n-2) \tau^2) + e_{j, i_1}\right\}\Big] 
\\
& 
= \mathbb{P}  \left[ \left\{ \tau \sqrt{\gamma n -2} Z + \tau Z_3 >  (\gamma n -1) (\mu_1 -\mu_2) + \tau \sqrt{\gamma n -2} Z_1 + \tau Z'\right\} \right. 
\\
& 
\quad \bigcap \left\{ \tau \sqrt{\gamma n -2} Z +\tau Z_4 > (\gamma n -1) (\mu_1 -\mu_2) + \tau \sqrt{\gamma n -2} Z_2 +\tau Z' \right\}\Big]
\\
& 
=  \mathbb{P}  \left[ \left\{ \tau \sqrt{\gamma n -1} Z  >  (\gamma n -1) (\mu_1 -\mu_2) + \tau \sqrt{\gamma n -1} Z_1  \right\} \right. 
\\
& 
\quad \bigcap \left\{ \tau \sqrt{\gamma n -1} Z  > (\gamma n -1) (\mu_1 -\mu_2) + \tau \sqrt{\gamma n -1} Z_2  \right\}\Big] 
\\
& 
= \mathbb{P}  \left[ \left\{Z > Z_1 + s\sqrt{2 \log n} (1+o(1)) \right\} \bigcap \left\{Z >  Z_2 + s\sqrt{2 \log n} (1+o(1)) \right\} \right] \\
& = \int_{\mathbb{R}} \Phi^2(z +s\sqrt{2 \log n} (1+o(1))) \phi(z) dz
\end{align*}

So, whenever $\gamma \text{SNR} < \frac{3}{4}$ and with the notations above we have the following.
 \begin{align}
 \frac{\mathbb{P}\left[  \mathcal{C}_{(i_1, j)}\cap \mathcal{C}_{(i_2, j)}\right]}{  n \mathbb{P}\left[\mathcal{C}_{ij}\right]^2} 
 & \sim 
 \frac{\mathbb{P}\left[  \mathcal{C}_{(i, j_1)}\cap \mathcal{C}_{(i, j_2)}\right]}{  n \mathbb{P}\left[\mathcal{C}_{ij}\right]^2}
 \\
 & 
 \sim 
 \frac{\int_{\mathbb{R}} \Phi^2(z +s\sqrt{2 \log n} (1+o(1))) \phi(z) dz}{ n \left(\frac{1}{ s \sqrt{\log n}}\phi (s \sqrt{\log n})\right)^2 }
 \\
 &
 \sim
 \frac{\int_{-s \sqrt{2 \log n} (1+o(1))}^{s \sqrt{2 \log n} (1+o(1))} \Phi^2(z +s\sqrt{2 \log n} (1+o(1))) \phi(z) dz}{  n \left(\frac{1}{ s \sqrt{\log n}}\phi (s \sqrt{\log n})\right)^2}
 \\
 & 
 \sim  \frac{ 2 \pi s^2 \log n}{ n}\int_{0}^{2s \sqrt{2 \log n} (1+o(1))} \Phi^2(z) \phi(z) e^{z s \sqrt{2 \log n}} dz
 \\
 & 
 \leqslant 
 \frac{  \sqrt{2 \pi} e^2 s^2 \log n}{ n} \int_{0}^{2s \sqrt{2 \log n} (1+o(1))}   e^{ - \frac{3z^2}{2} +z s \sqrt{2 \log n}} dz 
 \\
 & 
 \leqslant 
 \frac{C \log n}{ n} e^{\frac{4 \gamma \text{SNR} \log n}{3}} \int_{0}^{2s \sqrt{2 \log n}  (1+o(1))}   e^{ - \frac{3 (z - \frac{s\sqrt{ 2\log n}}{3})^2}{2}} dz 
 \\
 & 
 \to 0
 \end{align}

It is also immediate that up to $\text{poly}(\log n)$ factors\footnote{$a_n \simeq b_n$ if $a_n \leqslant \text{poly}(\log n) b_n$ and $b_n \leqslant \text{poly}(\log n)$} we have 
\[
\frac{ 2 \pi s^2 \log n}{ n}\int_{0}^{2s \sqrt{2 \log n} (1+o(1))} \Phi^2(z) \phi(z) e^{z s \sqrt{2 \log n}} dz \simeq n^{\frac{4 \gamma \text{SNR}}{3}-1}
\]
\end{proof}

    \subsection{Statistical possibility}

In this section, we prove that as soon as the $ \gamma \text{SNR} >1$, we have that MLE can recover the planted  community $\zeta^*$ exactly with probability approaching one. We prove this result in two steps. On the impossibility side,  we already showed that below a weaker threshold, there exists a bad pair of vertices, one in the planted community and another one from the rest,  that can be swapped  while increasing the likelihood. Here we show that the  \ref{def:MLEPD} gives the planted community, by showing that there is no possibility to swap any set of $r$
 vertices from the planted community and increase the likelihood for any $r \in \{1, \cdots, K =\gamma n\}$ above the threshold $ \gamma \text{SNR} >1$.
\subsubsection{The statement}
\begin{theorem}[\textbf{Statistical possibility of exact recovery of the planted community}]  \label{thm: Positive1}
    Under the  Gaussian weighted planted dense subgraph model  PDSM$(n, \mu_1, \mu_2, \tau^2)$ [\ref{def:GWPDSM}], MLE can exactly recover the community labelling $\zeta^*$ as soon as $ \gamma \text{SNR} >1$. More precisely, $ \mathbb{P}\left[M(\hat{\zeta}_{\text{MLE}}, \zeta^*)= 1\right] \to 1$.
\end{theorem}
\begin{Notations} With  $\zeta^*=\{\mathbf{1}_{\gamma n}, \mathbf{0}_{(1-\gamma)n}\} \in \Theta_n^{\gamma}$  we have the true 
 planted community $C^*:=\{1, \cdots, \gamma n\}$, and the complement  $[n] \setminus C^*:=\{\gamma n +1, \cdots, n\}$. Further,  let  $\hat{\zeta}  \in  \Theta_n^{\gamma}$ and  denote the predicted community labels  with $\hat{C}=\{i \in [n]: \hat{\zeta}(i)=1\},$  and the complement $ [n] \setminus \hat{C}:=\{i \in [n]: \hat{\zeta}(i)=0\}$. 
We observe that  MLE \ref{def:MLEPD} fails to recover the true planted community labeling $\zeta^*$  if and only if there exists $ \hat{\zeta} \in \Theta_n^{\gamma}$ such that $g_A(\hat{\zeta}) > g_A(\zeta^*)$.
For  subsets of vertices $A,B \subseteq [n]$ with $A\cap B = \emptyset$ we define the following
\[
e(A, B) =\sum_{p \in A,q \in B} A_{pq}
\]
\[
e(A) = \sum_{p \in A, q \in A} A_{pq}
\]
\end{Notations}

\subsubsection{The proof -- union bound}
 \begin{lemma} \label{lem:MLEpos1}
 $\exists$ $\hat{\zeta} \in \Theta_n^{\gamma}$ with 
     $g_A(\hat{\zeta}) > g_A(\zeta^*) \smallimplies$ $\exists  S_+ \subseteq C^*, S_- \subseteq [n] \setminus C^*$ such that $1\leqslant |S_+| =|S_-| \leqslant \gamma n$ and  we have
     \begin{equation} \label{eq:Pos1}
     e(S_-) +
     e(C^* \setminus S_+ , S_-) > 
     e(S_+)+
     e(C^* \setminus S_+, S_+)
     \end{equation} 
 \end{lemma}
 \begin{proof}  
 For $\zeta \in \Theta_n^{\gamma}$ let $C_{\zeta} =\{ i \in [n]: \zeta_i =1\}$. We observe that $
     g_A(\zeta)= \sum_{i,j} A_{ij}\zeta_i\zeta_j=   \sum_{\zeta_ i=\zeta_j =1} A_{ij}\zeta_i\zeta_j = e(C_{\zeta})$
     \[
         g_A(\zeta) > g_A(\zeta^*)  \iff e(C_{\zeta}) > e(C^*)
     \]
     
     Now given $\hat{\zeta}$ and $\zeta^*$ we define $S_+:=C^*\setminus\hat{C},$  $S_-:=  \hat{C} \cap [n] \setminus C^* = \hat{C} \setminus C^* \implies |S_+| =|S_-|$, and observe that $g_A(\zeta) > g_A(\zeta^*)$ implies $|S_+| \geqslant 1$. Further,
     \[
     e(\hat{C}) =  
     e(\hat{C} \cap C^*) +
     e(\hat{C} \setminus C^*) +
     e(\hat{C} \cap C^*, \hat{C} \setminus C^*)
     \]
     \[
     e(C^*) =  
     e(C^* \cap \hat{C}) +
     e(C^* \setminus \hat{C})+
     e(C^* \cap \hat{C}, C^* \setminus \hat{C})
     \]

     Combining the above, we have the following equivalence.
     \begin{align*}
     g_A(\hat{\zeta}) > g_A(\zeta^*) 
     &  
     \iff  
     e(\hat{C} \setminus C^*) +
     e(\hat{C} \cap C^*, \hat{C} \setminus C^*) > 
     e(C^* \setminus \hat{C})+
     e(C^* \cap \hat{C}, C^* \setminus \hat{C})
     \\
     &
     \iff  e(S_-) +
     e(C^* \setminus S_+ , S_-) > 
     e(S_+)+
     e(C^* \setminus S_+, S_+)
     \end{align*}
 \end{proof}

 Now, to prove that MLE \ref{def:MLEPD} can solve exact recovery \ref{thm: Positive1} as soon as 
  $\gamma$ \text{SNR} $>1$  by proving that such events of the kind \ref{eq:Pos1} occur with probability approaching zero. We define the following events as the following.
 \[
 \begin{split}
E:=\{  \exists S_+ \subseteq C^*,& S_- \subseteq [n] \setminus C^* : 1\leqslant |S_+|=|S_-| \leqslant 
 \min(\gamma n, (1-\gamma)n), \\  &   e(S_-) + e(C^* \setminus S_+ , S_-) > 
 e(S_+) + e(C^* \setminus S_+, S_+) \}
\end{split}
 \]
 \[
 E_l= \{\exists S_+ \subseteq C^*, S_- \subseteq [n] \setminus C^* : |S_+|=|S_-| =l, e(S_-) +
     e(C^* \setminus S_+ , S_-) > 
     e(S_+)+
     e(C^* \setminus S_+, S_+)\}
 \]
 \begin{lemma} \label{lem:MLEpos2} Let $K=\gamma n$. Then
    \[
          \mathbb{P}\left[E\right] =  \mathbb{P}\left[\bigcup_{l=1}^{\min(K, n-K)}E_l \right] \to 0
    \]
 \end{lemma}
 \begin{proof}
     We prove this by observing the following for a fixed pair $(S_+, S_-)$ of size $l$.
     \begin{align}
     \mathbb{P}[E] 
     & 
     \leqslant \sum_{ 1\leqslant l \leqslant \min(K, n-K)} \mathbb{P}[E_l]
     \\
     &
     \leqslant  \sum_{ 1\leqslant l \leqslant \min(K, n-K)}  \binom{K}{l} \binom{n-K}{l} \mathbb{P}\left[e(S_-) +
     e(C^* \setminus S_+ , S_-) > e(S_+)+ e(C^* \setminus S_+, S_+)\right]\to 0
     \end{align}
     
     We observe that $e(S_-), e(S_+), e(C^* \setminus S_+, S_-),$  and $e(C^*\setminus S_+, S_+)$ are jointly independent random variables with 
     \[
     e(S_+)\overset{d}{=}  N\left( \binom{l}{2} \mu_1, \binom{l}{2} \tau^2\right) 
     \]
     \[
     e(S_-)\overset{d}{=}  N\left( \binom{l}{2} \mu_2, \binom{l}{2} \tau^2\right)
     \]
     \[
     e(C^*\setminus S_+, S_+) \overset{d}{=} N\left( l(K -l) \mu_1, l(K -l) \tau^2 \right)
     \]
     \[
     e(C^*\setminus S_+, S_-) \overset{d}{=} N\left( l(K -l) \mu_2, l(K -l) \tau^2\right)
     \]
     
     Combining the above distributional identities we have the following simplification.
     \begin{align*}
     \mathbb{P} 
     & \left[ e(S_-) +
     e(C^* \setminus S_+ , S_-) > e(S_+)+ e(C^* \setminus S_+, S_+)\right] 
     \\
     &
     =\mathbb{P}\left[ \tau \sqrt{2} \sqrt{\left(\binom{K}{2} - \binom{K-l}{2}\right)}N(0,1)> \left(\binom{K}{2} - \binom{K-l}{2}\right)(\mu_1- \mu_2)\right] 
     \\
     & 
     = \mathbb{P}\left[ N(0,1) > \sqrt{\left(\binom{K}{2} - \binom{K-l}{2}\right)} \left(\frac{\mu_1- \mu_2}{\tau \sqrt{2}}\right)\right]
     \\
     & 
     \leqslant  e^{ - \frac{2 \text{SNR} \left(\binom{K}{2} - \binom{K-l}{2}\right) \log n}{n}} 
     = \frac{1}{n^{ \frac{ l  (2K-1-l) \text{SNR}}{n}}}
     \end{align*}
     
     From the standard Gaussian tail bounds \ref{eq:Gtail} we have the following for $l=1$ with $ \gamma \text{SNR}= 1+2\delta$.
     \[ \mathbb{P}\left[E_1\right] \leqslant 
     K(n-K) \frac{1}{n^{ \frac{2(K-1) \text{SNR}}{n}}}\leqslant \gamma (1-\gamma) n^2 n^{-2 \gamma  
       \text{SNR} (1+o(1))} =  \frac{\gamma (1-\gamma)}{n^{4\delta +o(1)} }  \to 0
     \]

     More generally, let $ \epsilon = \frac{l}{n}$. Then we apply the bounds on the binomial coefficients for $1 \leqslant l \leqslant p$.
     \begin{equation} \label{eq:Binom}
    \left(\frac{p}{l}\right)^l \leqslant \binom{p}{l} \leqslant \left(\frac{ep}{l}\right)^l
     \end{equation}

     We choose  $\epsilon_{\delta} >0 $ be such that $ \epsilon_{\delta} \text{SNR} =  \delta \implies $
     For $\epsilon \leqslant \epsilon_{\delta}$ we have $-2\delta + \frac{(l+1)\text{SNR}}{2n} \leqslant -\delta$
     \[
       \frac{\binom{K}{l}\binom{n-K}{l}}{n^{ \frac{l(2K-1-l)\text{SNR}}{n}}}   \leqslant \left(\frac{e^2 \gamma (1-\gamma )}{l^2}
     \right)^{l} \frac{n^{2l}}{n^{2\gamma \text{SNR}l}} n^{\frac{l(l+1)\text{SNR}}{n}}\leqslant  n^{ 2l(-2\delta + \frac{(l+1)\text{SNR}}{2n}) } \leqslant \frac{1}{n^{2l\delta}}
     \]
     
     For $\epsilon \geqslant \epsilon_{\delta}$  we use the simpler binomial identity $\binom{p}{l} \leqslant 2^{p}$ and use the fact that $ l \leqslant \min(K, n-K)$
     \[
       \frac{\binom{K}{l}\binom{n-K}{l}}{n^{ \frac{l(2K-1-l)\text{SNR}}{n}}} \leqslant \frac{2^{n}}{n^{\delta (\gamma n -1)}}   \leqslant \frac{1}{n^{\delta \gamma n (1 - o(1))}}
     \]
     \begin{align*}
         \sum_{ 1\leqslant l \leqslant \min(K, n-K)}  \frac{\binom{K}{l}\binom{n-K}{l}}{n^{ \frac{l(2K-1-l)\text{SNR}}{n}}} 
         \leqslant
         \sum_{ \frac{1}{n} \leqslant \frac{l}{n}  \leqslant \epsilon_{\delta}} \frac{1}{n^{2l\delta}} + \sum_{ \epsilon_{\delta} \leqslant \frac{l}{n}  \leqslant  \min(\gamma, 1-\gamma)} \frac{1}{n^{\delta \gamma n (1 - o(1))}}  \to 0.
     \end{align*}
     
 \end{proof}
    \subsection{Semi-definite relaxation}

We begin with the observation that we can write the MLE \ref{eq:MLE3} in the form of an SDP with rank one constraints.

\subsubsection{The statement}
\begin{lemma} \label{lem:SDPeuiv}
The optimization problem corresponding to MLE \ref{eq:MLE3} is equivalent to the following optimization problem.

\begin{equation} \label{eq:MLE4}
\max_{Z, \zeta} \hspace{.1 cm} \langle A, Z \rangle :  Z =\zeta \zeta^{\top},  Z_{ii} \leqslant  1, \hspace{.1 cm} \forall 
 \hspace{.1 cm} i \in [n],  Z_{ij} \geqslant  0, \hspace{.1 cm} \forall \hspace{.1 cm} i, j \in[n], \langle\mathbf{I}, Z\rangle=\gamma n,  \langle\mathbf{J}, Z \rangle =\gamma^2 n^2,
\end{equation}

\begin{proof}
    The objective function is the same for both \ref{eq:MLE3} and \ref{eq:MLE4}. We also observe that if $\zeta$ is a solution to \ref{eq:MLE4}, then $-\zeta$ is a solution too.  It suffices to prove that the constraint sets are equal up to a global sign flip. Let $\zeta \in \Theta^{\gamma}_n$.   
    \[ 
    Z_{ii}= \zeta_i^2 \leqslant 1 \forall \hspace{.1 cm} i \in [n], 
    \]
    \[
    Z_{ij}= \zeta_i \zeta_j \geqslant 0 \hspace{.1 cm} \forall \hspace{.1 cm} i, j \in[n]
    \]
    \[
    \langle \mathbf{I}, Z \rangle  =\sum_{i } \zeta_i^2 = \sum_{i } \zeta_i =\gamma n
    \]
    \[
    \langle  \mathbf{J}, Z \rangle  = \left(\sum_{i } \zeta_i\right)^2= \gamma^2 n^2
    \]
    
    This clearly implies that the constraint set of \ref{eq:MLE3} is a subset of \ref{eq:MLE4}.
    On the other hand 
    \[ 
    \zeta_i^2 \leqslant 1 \forall \hspace{.1 cm} i \in [n] \implies  -1 \leqslant \zeta_i \leqslant 1, 
    \]
    \[
    Z_{ij}= \zeta_i \zeta_j \geqslant 0 \hspace{.1 cm} \forall \hspace{.1 cm} i, j \in[n] \implies  \hspace{.1 cm}\text{ all }  \hspace{.1 cm} \zeta_i \hspace{.1 cm}\text{have the same signs}
    \]
    
    Combining the above, we have either $\zeta \in [0,1]^{n}$ or $\zeta \in [-1, 0]^n$.
    
    \[
    \langle \mathbf{I}, Z \rangle  =\sum_{i } \zeta_i^2 = \gamma n
    \]
    \[
    \langle  \mathbf{J}, Z \rangle  = \left(\sum_{i } \zeta_i\right)^2= \gamma^2 n^2 \implies \sum_i \zeta_i = \pm \gamma n
    \]
    
    Assuming  $\zeta \in [0,1]^{n}$ we have $\sum_{i}\zeta_i^2 =\sum_{i} \zeta_i =\gamma n$. Therefore we have 
    \[
    \sum_{i} \zeta_i(1-\zeta_i)=0 \implies \zeta_i(1-\zeta_i) = 0 \forall i \in [n] \implies \zeta \in \{0,1\}^n \cap \langle \zeta , \mathbf{1}_n \rangle =\gamma n \equiv \zeta \in \Theta^{\gamma}_n.
    \]
    
    Similarly, assuming $\zeta \in [-1,0]^{n}$  we have $\sum_{i}\zeta_i^2 = -\sum_{i} \zeta_i =\gamma n$, and hence 
    \[
      \sum_{i} \zeta_i(1+\zeta_i) =0 \implies \zeta_i(1+\zeta_i) = 0 \forall i \in [n] \implies -\zeta \in \{0,1\}^n \cap \langle -\zeta , \mathbf{1}_n \rangle =\gamma n   \equiv - \zeta \in \Theta^{\gamma}_n
    \]
\end{proof}
\end{lemma}

In the above formulation \ref{eq:MLE4} the matrix $Z=\zeta \zeta^{\top}$ is positive semidefinite and rank-one. Removing the rank-one restriction leads to the following convex relaxation of \ref{eq:MLE4}, which is a semidefinite program.

\begin{equation} \label{def:SDPPDSM}
 \hat{Z}_{\mathrm{SDP}}=\underset{Z}{\arg \max } \langle A, Z\rangle:  Z \succeq 0,  Z_{ii} \leqslant 1, \hspace{.1 cm} \forall \hspace{.1 cm} i \in [n], Z_{ij} \geqslant 0, \hspace{.1 cm} \forall \hspace{.1 cm} i, j \in[n], \langle\mathbf{I}, Z\rangle=\gamma n, \langle\mathbf{J}, Z\rangle=\gamma^2 n^2 .
\end{equation}

Let $Z^*=\zeta^*\zeta^{*^{\top}}$ correspond to the true cluster and define $\mathcal{Z}_n=\left\{\zeta \zeta^{\top}: \zeta \in\{0,1\}^n, \langle \zeta, \mathbf{1} \rangle =\gamma n\right\}$. The threshold for the exact recovery of the densely weighted community for the SDP \ref{def:SDPPDSM} is given as follows.
\[ 
Z^* =
\begin{bmatrix}
\begin{array}{c|c}
\mathbf{1}_{\gamma n  \times \gamma n} & \mathbf{0}_{\gamma n \times n} \\[1ex]
\hline
\\[-2ex]
\mathbf{0}_{n \times \gamma n} & \mathbf{0}_{ (1-\gamma ) n \times (1-\gamma) n}
\end{array}
\end{bmatrix}.
\]
\begin{theorem}[\textbf{Semi-definite relaxation achieves exact recovery of the planted community}]  \label{thm:SDP1}
    Under the Gaussian weighted planted dense subgraph model \ref{def:GWPDSM} with parameter $\gamma \in  (0,1)$, 
 
\[ \text{if} \hspace{.1 cm} \gamma \text{SNR} \hspace{.1 cm}>1, \hspace{.1 cm} \text{then} \hspace{.1 cm}
\min _{Z \in \mathcal{Z}_n} \mathbb{P}\left[\hat{Z}_{\mathrm{SDP}}=Z^*\right]\to 1 \hspace{.1 cm} \text{as} \hspace{.1 cm} n \rightarrow \infty.
\]
\end{theorem} 

\subsubsection{A deterministic condition}

Now, we prove this theorem by first establishing a deterministic condition for a  unique solution.

\begin{lemma}\label{lem:DetPDSM}
 Suppose $\exists$  $D^{*}=\operatorname{diag}\left\{d_{i}^{*}\right\}$ with $\{d^*_i\}_{i=1}^{n} \in  \mathbb{R}_{\geqslant 0}^{n}$,  $B^{*} \in \mathbb{R}^{n\times n}_{\geqslant 0},$    and $\lambda^{*}, \eta^{*} \in \mathbb{R}$ such that 
 \begin{equation}\label{eq:First}
 S^{*} \triangleq D^{*}-B^{*}-A+\eta^{*} \mathbf{I}+\lambda^{*} \mathbf{J} \hspace{.1 cm} \text{satisfies} \hspace{.1 cm} S^{*} \succeq 0, \lambda_{2}\left(S^{*}\right)>0, \hspace{.1 cm} \text{and}  \hspace{.1 cm}S^{*}  \zeta^{*}  =0,
 \end{equation}
 \begin{equation}\label{eq:Second}
 \hspace{.1 cm}
d_{i}^{*}\left(Z_{i i}^{*}-1\right)  =0 \quad \forall i, \quad 
B_{i j}^{*} Z_{i j}^{*}  =0 \quad \forall i, j,
 \end{equation}

Then SDP recovers the true solution. More precisely, $\widehat{Z}_{\mathrm{SDP}}=Z^{*}$ is the unique solution to \ref{def:SDPPDSM}.
\end{lemma}
\begin{proof}
    For auxiliary Lagrangian multipliers denoted by $S \succeq 0, D = \operatorname{diag}\left\{d_{i}\right\}$ with $\{d_i\}_{i=1}^{n} \in  \mathbb{R}_{\geqslant 0}^{n}$, $B \geqslant 0$, and  $\lambda, \eta \in \mathbb{R}$, we define the Lagrangian dual function of the original semi-definite program \ref{def:SDPPDSM} as
\begin{align}
L(Z, S, D, B, \lambda, \eta) 
& :=\langle A, Z\rangle+\langle S, Z\rangle+\langle D, \mathbf{I}-Z\rangle+\langle B, Z\rangle+\eta( \gamma n-\langle\mathbf{I}, Z\rangle)+\lambda\left(\gamma ^2 n^2- \langle\mathbf{J}, Z\rangle\right)
\\
& = \langle  A+S -D +B -\eta \mathbf{I} -\lambda\mathbf{J}, Z\rangle  +\langle D, \mathbf{I} \rangle +\eta \gamma n + \lambda \gamma^2 n^2
\end{align}

 Then, for any $Z$ satisfying the constraints in \ref{def:SDPPDSM}, we have the following weak duality inequality.
\begin{align}
\langle A, Z\rangle & \stackrel{(a)}{\leqslant } L\left(Z, S^{*}, D^{*}, B^{*}, \lambda^{*}, \eta^{*}\right) \stackrel{(b)}{=}\left\langle D^{*}, I\right\rangle+\eta^{*} \gamma n+\lambda^{*} \gamma ^2 n^2 \stackrel{(c)}{=}\left\langle D^{*}, Z^{*}\right\rangle+\eta^{*} \gamma n+\lambda^{*} \gamma^{2} n^2 \\
& =\left\langle A+B^{*}+S^{*}-\eta^{*} \mathbf{I}-\lambda^{*} \mathbf{J}, Z^{*}\right\rangle+\eta^{*} \gamma n+\lambda^{*} \gamma^{2} n^2 \stackrel{(c)}{=}\left\langle A, Z^{*}\right\rangle
\end{align}

where (a) follows from \ref{eq:First} and \ref{def:SDPPDSM} implying $\left\langle S^{*}, Z\right\rangle \geqslant 0,\left\langle D^{*}, \mathbf{I}- Z\right\rangle \geqslant  0$, and $\left\langle B^{*}, Z\right\rangle \geqslant 0$; (b) holds directly from \ref{eq:First}; (c) holds because  $d_{i}^{*}\left(Z_{i i}^{*}-1\right)=0, \forall i$; (d) holds because $B_{i j}^{*} Z_{i j}^{*}=0, \hspace{.1 cm}\forall \hspace{.1 cm}i, j \in [n]$ from \ref{eq:Second} implying $\langle B^*, Z^* \rangle = 0$ and it follows from \ref{eq:First}  that $\left\langle S^{*},  Z^{*} \right\rangle=\zeta^{*^{\top}} S^{*} \zeta^{*}=0$. Therefore, $Z^{*}$ is an optimal solution to \ref{def:SDPPDSM}. 

We now establish its uniqueness. To this end, suppose $Z_2$ be another optimal solution to \ref{def:SDPPDSM}. Then we have,
\begin{align} \label{eq:ineq}
\left\langle S^{*}, Z_2\right\rangle 
& 
= \left\langle D^{*}-B^{*}-A+\eta^{*} \mathbf{I}+\lambda^{*} \mathbf{J}, Z_2\right\rangle \stackrel{(a)}{=}\left\langle D^{*}-B^{*}, Z_2\right\rangle + \langle  -A +\eta^* \mathbf{I}+ \lambda^* \mathbf{J},Z^*\rangle
\\ 
& 
\stackrel{(b)}{\leqslant}\left\langle D^{*}-B^*, Z^{*}\right\rangle + \langle  -A +\eta^* \mathbf{I}+ \lambda^* \mathbf{J},Z^*\rangle  =\left\langle S^{*}, Z^{*}\right\rangle=0 
\end{align}

where $(a)$ holds because 
$\langle\mathbf{I}, Z_2\rangle=\gamma n = \langle\mathbf{I}, Z^*\rangle$, $\langle\mathbf{J}, Z_2\rangle=\gamma^{2}n^2= \langle\mathbf{J}, Z^*\rangle$ and  $\langle A, Z_2\rangle=\left\langle A, Z^{*}\right\rangle$ ;(b) holds because from \ref{def:SDPPDSM} and \ref{eq:Second} we have  $B^{*}, Z_2 \geqslant  0$ and therefore $\langle -B^*, Z_2 \rangle 
 \leqslant 0 = \langle B^*, Z^* \rangle $. Furthermore, we have   $\left\langle D^{*}, Z_2\right\rangle =\sum_{i:\zeta_i =1} D^*_{ii} Z_{2_{ii}} \leqslant \sum_{i: \zeta_i=1} d_{i}^{*}=\left\langle D^{*}, Z^{*}\right\rangle$ since $d_{i}^{*} \geqslant 0$ and $Z_{2_{ii}} \leqslant 1$ for all $i \in[n]$. Since $Z_2 \succeq 0$ and $S^{*} \succeq 0$ we also have $\langle S^{*}, Z_2\rangle \geqslant 0$. Combining with the above we have that $\langle S^*, Z_2\rangle =0$.

 We are now in a position to apply lemma \ref{lem:uniq} with $S^* \succeq 0$, $\lambda_{2}\left(S^{*}\right)>0,$  and $\langle S^*, Z_2 \rangle = \langle S^*, Z^* \rangle =0$. By lemma \ref{lem:uniq} $Z_2$ needs to be a multiple of $Z^{*}=\zeta^{*}\zeta^{*^{\top}}$. Finally, we have  $Z_2=Z^{*}$ since $\operatorname{Tr}(Z_2)=\operatorname{Tr}\left(Z^{*}\right)=\gamma n$.
\end{proof}

\subsubsection{A high probability event and the proof} \label{sec:hpe1}
We are now in a position to provide a proof of theorem \ref{thm:SDP1}. Let $C^*:= \{ i \in [n] :\zeta_i=1\}$. We shall choose $\lambda^*$ later. For now, we  choose $\eta^{*}=2\|A-\mathbb{E}[A]\|, D^{*}=\operatorname{diag}\left\{d_{i}^{*}\right\}$ with $d^*_i= \langle (A-\eta^*\mathbf{I}- \lambda^*\mathbf{J})_i, \zeta^*\rangle \mathbf{1}_{i \in C^*}$. More precisely,

\begin{equation} \label{eq:Dstar}
d_{i}^{*} = \begin{cases}
    \sum_{j \in C^{*}} A_{i j} - \eta^{*} - \lambda^{*} \gamma n & \text{if } i \in C^{*} \\
    0 & \text{otherwise}
\end{cases}
\end{equation}

Define $b_{i}^{*} \gamma n \triangleq \lambda^{*} \gamma n - \sum_{j \in C^{*}} A_{i j} =  \langle (\lambda^* \mathbf{J} -A)_i, \zeta^* \rangle  $ for $i \notin C^{*}$. We define  $B^{*} \in \mathbb{R}^{n\times n}$ as the following.

\begin{equation}\label{eq:Bstar}
B_{i j}^{*}=b_{i}^{*} \mathbf{1}_{\left\{i \notin C^{*}, j \in C^{*}\right\}}+b_{j}^{*} \mathbf{1}_{\left\{i \in C^{*}, j \notin C^{*}\right\}}
\end{equation}

It remains to show that $\left(S^{*}, D^{*}, B^{*}\right)$ satisfies the conditions in Lemma  \ref{lem:DetPDSM} with probability approaching one.

By definition of $D^*$  and $B^*$, we have $d_{i}^{*}\left(Z_{i i}^{*}-1\right)=0$ $\forall$ $ i\in [n]$ and $B_{i j}^{*} Z_{i j}^{*}=0$ $\forall$ $i, j \in[n]$. Moreover, $\forall$ $i \in C^{*}$,

\begin{equation}
d_{i}^{*} \zeta_{i}^{*}=d_{i}^{*} =\sum_{j} A_{i j} \zeta_{j}^{*}-\eta^{*}-\lambda^{*} \gamma n \stackrel{(a)}{=}\sum_{j} A_{i j} \zeta_{j}^{*}+\sum_{j} B_{i j}^{*} \zeta_{j}^{*}-\eta^{*}-\lambda^{*} \gamma n =\langle (A+B^* -\eta^*\mathbf{I} -\lambda^* \mathbf{J})_i, \zeta^*\rangle 
\end{equation}

where equality $(a)$ holds because $B_{i j}^{*}=0$ if $(i, j) \in C^{*} \times C^{*}$. Furthermore,  $\forall$ $i \notin C^{*}$,

\begin{equation} 
\sum_{j} A_{i j} \zeta_{j}^{*}+\sum_{j} B_{i j}^{*} \zeta_{j}^{*}-\lambda^{*} \gamma n=\sum_{j \in C^{*}} A_{i j}+\gamma n b_{i}^{*}-\lambda^{*} \gamma n = \langle (A+ B^* - \eta^*\mathbf{I} -\lambda^* \mathbf{J})_i, \zeta^*\rangle \stackrel{(a)}{=} 0 =d^*_i\zeta^*_i
\end{equation}

where the final equality (a) follows from our choice of $b_{i}^{*}$. Combining the two above we have,
\[
D^{*} \zeta^{*}=A \zeta^{*}+B^{*} \zeta^{*}-\eta^{*} \zeta^{*}-\lambda^{*} \gamma n \mathbf{1} = (A+B^* -\eta^* \mathbf{I} -\lambda^* \mathbf{J})\zeta^*
\]

As a consequence $S^*$ satisfies the following condition.
\[
S^{*} \zeta^{*}= D^* \zeta^* - (A+B^* -\eta^* \mathbf{I} -\lambda^* \mathbf{J})\zeta^* = 0.
\]

We next prove that $D^{*} \geqslant 0, B^{*} \geqslant 0$ with probability approaching one. It follows again from  \ref{eq:GNorm} standard bounds on norms of Gaussian random matrices that $\eta^{*}= 2\|A-\mathbb{E}[A]\| \leqslant 6 \tau \sqrt{n}$ with probability at least $1-2e^{-cn}$ for some absolute positive constant $c$. Furthermore, let $X_{i} \triangleq \sum_{j \in C^{*}} A_{i j} = \langle A_i , \zeta^* \rangle $.  Then we have the distributional identity.

\[
 X_i \sim
\begin{cases}
    N(\gamma n \mu_1, \gamma n \tau^2)& \text{if } i \in C^* \\
    N(\gamma n \mu_2, \gamma n \tau^2)  &  i \notin C^* 
\end{cases}
\]

 We will show using union bound that  with probability approaching one
 \begin{equation}
 d^*_i =X_i - 6\tau \sqrt{n} - \lambda^{*} \gamma n \geqslant 0  \hspace{.1 cm} \forall \hspace{.1 cm}  i \in C^*
 \end{equation}
 \begin{equation}
 b_i^* \gamma n = \lambda^{*} \gamma n - X_i  \geqslant 0 \hspace{.1 cm} \forall \hspace{.1 cm} i \notin C^*
 \end{equation}
 
We  prove that $\mathbb{P}\left[ \cap _{i \in C^*} \{d^*_i \geqslant 0\}\right] \to 1 \iff \mathbb{P}\left[ \cup_{i \in C^*} \{d^*_i < 0\}\right] \to 0$  by showing the following with $\lambda^* = \frac{\mu_1 +\mu_2}{2}$.

\begin{align*} 
\gamma n \mathbb{P}\left[ d^*_i < 0\right] 
&  = \gamma n \mathbb{P}\left[ \gamma n (\mu_1 -\lambda^*) - \tau \sqrt{\gamma n} N(0,1) < 6\tau \sqrt{n}\right]
\\
& = \gamma n \mathbb{P}\left[ N(0,1) > \sqrt{\gamma n} \frac{(\mu_1 -\mu_2)}{2\tau} -    \frac{6}{\sqrt{\gamma}}\right] 
\\
& \leqslant  \gamma n e^{-\gamma \text{SNR} (1+ o(1)) \log n }  = \frac{\gamma }{n^{\gamma \text{SNR} -1 +o(1)}} \to 0
\end{align*}

Similarly, we  prove that $\mathbb{P}\left[ \cap _{i \notin C^*} \{b^*_i \geqslant 0\}\right] \to 1 \iff \mathbb{P}\left[ \cup_{i \notin C^*} \{b^*_i < 0\}\right] \to 0$  by showing the following.

\begin{align*} 
(1-\gamma) n \mathbb{P}\left[ b^*_i \leqslant 0\right] 
&  = (1-\gamma) n \mathbb{P}\left[ \tau \sqrt{\gamma n} N(0,1) > \gamma n (\lambda^* -\mu_2)\right]
\\
& = (1-\gamma) n \mathbb{P}\left[ N(0,1) > \sqrt{\gamma n} \frac{(\mu_1 -\mu_2)}{2\tau}\right] 
\\
& \leqslant  (1-\gamma) n e^{-\gamma \text{SNR}  \log n }  = \frac{(1-\gamma) }{n^{\gamma \text{SNR} -1}} \to 0
\end{align*}

Finally, it remains to prove that $S^* \succeq 0$ with $\lambda_2\left(S^*\right)>0$ with probability approaching one. More precisely,
\[
\mathbb{P}\left[\inf _{x \perp \zeta^*,\|x\|=1} \langle x, S^* x\rangle >0\right] \to 1
\]

Note that we have the following decomposition of the expected adjacency matrix $\mathbb{E}\left[A\right]$.
$$
\mathbb{E}[A]=(\mu_1-\mu_2) Z^*-\mu_1\left[\begin{array}{rr}
\mathbf{I}_{\gamma n \times \gamma n} & \mathbf{0} \\
\mathbf{0} & \mathbf{0}
\end{array}\right]-\mu_2\left[\begin{array}{cc}
\mathbf{0} & \mathbf{0} \\
\mathbf{0} & \mathbf{I}_{(1-\gamma)n \times(1- \gamma)n}
\end{array}\right]+\mu_2 \mathbf{J} .
$$

It follows that for any $x \perp \zeta^*$ and $\|x\|=1$, we have $\langle x, Z^* x\rangle = (\langle x, \zeta^* \rangle)^2 =0$ and
$$
\begin{aligned}
\langle x, S^* x \rangle  & = \langle x, D^* x \rangle - \langle x, B^* x \rangle +\left(\lambda^*-\mu_2\right) \langle x, \mathbf{J} x \rangle +\mu_1 \sum_{i \in C^*} x_i^2+\mu_2 \sum_{i \notin C^*} x_i^2+\eta^*- \langle x, (A-\mathbb{E}[A]) x \rangle  
\\
& 
\stackrel{(a)}{=} \sum_{i \in C^*}d_i^* x_i^2+\left(\lambda^*-\mu_2\right) \langle x, \mathbf{J} x\rangle +  \mu_1 \sum_{i \in C^*} x_i^2 + \mu_2 \sum_{i \notin C^*} x_i^2+\eta^*- \langle x, (A-\mathbb{E}[A]) x\rangle  
\\
& 
\stackrel{(b)}{\geqslant} \left(\min _{i \in C^*} d_i^* \right) \sum_{i \in C^*} x_i^2+ \frac{\left(\mu_1-\mu_2\right)}{2} (\langle x,  \mathbf{1}  \rangle)^2+ \mu_1 \sum_{i \in C^*} x_i^2+\mu_2 \sum_{i \notin C^*} x_i^2+\eta^*-\|A-\mathbb{E}[A]\|
\\
& 
\stackrel{(c)}{\geqslant}\left(\min _{i \in C^*} d_i^*\right) \sum_{i \in C^*} x_i^2+ \mu_1 \sum_{i \in C^*} x_i^2+ \mu_2 \sum_{i \notin C^*} x_i^2  + \|A-\mathbb{E}[A]\|
\\
& 
\stackrel{(d)}{\geqslant }   \|A-\mathbb{E}[A]\| +   (\mu_1 -\mu_2) \sum_{i \in C^*} x_i^2+ \mu_2  \quad \text{with probability approaching one}
\\
& 
\stackrel{(e)}{\geqslant}   \|A-\mathbb{E}[A]\| +\mu_2 \stackrel{(f)}{>}0 \quad \text{with probability approaching one}
\end{aligned}
$$

where (a) holds because $B_{i j}^*=0$ for all $(i, j) \in C^*\times C^*$ and  $(i, j) \in [n] \setminus C^*\times [n] \setminus C^*$
\begin{align}
\langle x, B^* x \rangle
& =  \sum_{i \notin C^*} \sum_{j \in C^*} x_i x_j B_{i j}^* +  \sum_{i \in C^*} \sum_{j \notin C^*} x_i x_j B_{i j}^* 
\\
& = 2 \sum_{i \notin C^*} \sum_{j \in C^*} x_i x_j B_{i j}^*
\\
&
= 2 \sum_{i \notin C^*} x_i b_i^* \sum_{j \in C^*} x_j
\\
&
= 2  \langle x , \zeta^*  \rangle \sum_{i \notin C^*} x_i b_i^*
\\
&
=0
\end{align}

(b) holds because  $\lambda^* =\frac{\mu_1+\mu_2}{2}$ and $\|A-\mathbb{E}[A]\| \geqslant  \langle x, A-\mathbb{E}[A] x \rangle $. Further, $(c)$ holds because  $\eta^*=2\|A-\mathbb{E}[A]\|$ and $\mu_1 \geqslant \mu_2$. (d) holds because  $\min _{i \in C^*} d_i^* \geqslant 0$ with  probability approaching one. Finally, again from standard bounds on norms on Gaussian random matrices, for some absolute positive constant $c >0$
\begin{equation}
\mathbb{P}\left[\|A-\mathbb{E}[A]\| \geqslant \tau c\sqrt{n}  \right]\to 1
\end{equation}

This along with  $\tau c\sqrt{n} \geqslant  |\mu_2| = |\alpha| \sqrt{\frac{\log n}{n}}$ implies $(f)$ holds with probability approaching one. Hence, the result.

\end{document}